\documentclass[12pt]{article} 
\usepackage{amsmath}
\usepackage{amssymb}
\usepackage{amsthm}
\usepackage[top=1in, bottom=1in, left=1in, right=1in]{geometry}
\usepackage{xcolor}
\usepackage{hyperref}
\usepackage{comment}
\usepackage{tikz-cd}
\usepackage{enumerate} 

\usepackage{mathtools}
\mathtoolsset{showonlyrefs}

\newtheorem{theorem}{Theorem}[section]
\newtheorem{definition}[theorem]{Definition}

\newtheorem{lemma}[theorem]{Lemma}
\newtheorem{conjecture}[theorem]{Conjecture}
\newtheorem{proposition}[theorem]{Proposition}
\newtheorem{corollary}[theorem]{Corollary}
\newtheorem{remark}[theorem]{Remark}

\newcommand{\PP}{{\mathbb P}} \newcommand{\RR}{\mathbb{R}}
\newcommand{\QQ}{\mathbb{Q}} \newcommand{\CC}{{\mathbb C}}

\newcommand{\TT}{{\mathbb T}}
\newcommand{\GG}{{\mathbb G}}
\newcommand{\KK}{{\mathbb K}}

\newcommand{\cX}{\mathcal{X}}

\newcommand{\cD}{\mathcal{D}}
\newcommand{\cE}{\mathcal{E}}

\newcommand{\cH}{\mathcal{H}}
\newcommand{\cA}{\mathcal{A}}

\newcommand{\cI}{\mathcal{I}}
\newcommand{\cJ}{\mathcal{J}}

\newcommand{\cL}{\mathcal{L}}
\newcommand{\cM}{\mathcal{M}}

\newcommand{\cO}{\mathcal{O}}

\newcommand{\cR}{\mathcal{R}}

\newcommand{\cT}{\mathcal{T}}

\newcommand{\cZ}{\mathcal{Z}}
\def\o{\omega}
\newcommand{\Ric}{\operatorname{Ric}}
\newcommand{\vol}{\operatorname{vol}}
\newcommand{\ord}{\mathrm{ord}}

\newcommand{\lct}{\mathrm{lct}}

\newcommand{\ddc}{\mathrm{dd^c}}
\newcommand{\tr}{\mathrm{tr}}

\makeatletter
\renewcommand\tableofcontents{%
  \section*{\contentsname
    \@mkboth{\MakeUppercase\contentsname}{\MakeUppercase\contentsname}}%
  \vspace{0.5em}
  \@starttoc{toc}%
}
\renewcommand\l@section[2]{%
  \vspace{0.5em}
  \@dottedtocline{1}{0em}{2.0em}{\normalsize\bfseries #1}{\normalsize\bfseries #2}}
\makeatother

\title{A YTD correspondence for constant scalar curvature metrics}
\author{Tam\'as Darvas and Kewei Zhang }
\date{}

\begin{document}
\maketitle

\begin{abstract}

Given a compact K\"ahler manifold, to better understand Mabuchi’s $K$ energy we introduce a family of $K^\beta$ energies, whose favorable properties are similar to those of the Ding energy from the Fano case. The construction uses Berman's transcendental quantization, and we show that the slope of the $K^\beta$ energies along test configurations can be computed using intersection theory. With these ingredients in place we provide a uniform Yau--Tian--Donaldson correspondence that characterizes the existence of a unique constant scalar curvature K\"ahler metric using test configurations. Combining our techniques with the non-Archimedean approach to $K$-stability pioneered by Boucksom--Jonsson, we show that the properness of the classical $K$ energy can be tested by checking its slope along a distinguished subclass of Chi Li-type models, called log discrepancy models, thus yielding another $\mathbb{G}$-uniform Yau--Tian--Donaldson correspondence.
\end{abstract}

\tableofcontents
\setcounter{tocdepth}{0}

\section{The main results}
\label{sec:Intro}

Let $(X,\omega)$ be a compact K\"ahler manifold. Classical formulations of the Yau–Tian–Donaldson (YTD) conjecture aim to characterize the existence of constant scalar curvature Kähler (cscK) metrics using algebraic invariants associated to test configurations \cite{YYTD, Tian97, Don02,Der16,Sz15,BHJ17,Der18,SD18}. Our work fits within this general framework. While we do not address any specific version of the YTD conjecture from the literature, we do establish YTD correspondences that characterize existence of cscK metrics using test configurations and novel algebraic invariants. 

To start, we recall that there is a one-to-one correspondence (up to a constant) between K\"ahler metrics cohomologous to $\omega$ and the space of K\"ahler potentials:
$$\mathcal H_\omega := \{u \in C^\infty(X,\RR) \textup{ s.t. } \omega_u :=\omega + \ddc u >0 \},$$
where $\ddc:=\frac{i}{2\pi}\partial\bar \partial.$
The study of canonical representatives on $\mathcal H_\omega$ goes back to questions of Calabi \cite{Ca54} and the famous theorem by Yau \cite{Yau78}.
Relevant to our work, the cscK equation can be formulated as follows:
\begin{equation}\label{eq: csck_eq}
S_{\omega_u}= \overline{S}, \ \ \ u \in \mathcal H_\omega,
\end{equation}
where $S_{\omega_u}$ is the scalar curvature of $\omega_u$, $\overline {S} = \frac{n}{V}\int_X \Ric \omega \wedge \omega^{n-1}$ is the average scalar curvature, and $V = \int_X \omega_u^n$ is the total K\"ahler volume.

Solutions to \eqref{eq: csck_eq}  are critical points of Mabuchi's $K$ energy functional $K: \mathcal H_\omega \to \Bbb R$ \cite{Mab86}:
\begin{equation}\label{Kendef}
K(u) = {Ent}(u) -n {I}_{\Ric(\omega)}(u)+ \bar S {I}(u),
\end{equation}
where ${Ent}(u) = \frac{1}{V}\int_X \log(\o^n_u/\o^n)\o^n_u$ is the entropy of $\o^n_u$ with respect to $\o^n$,  whereas ${I},{I}_{\textup{Ric}(\omega)}:\mathcal H_\o \to \Bbb R$ are the Monge--Amp\`ere (Aubin-Yau)  energy and its Ricci-twisted version:
\begin{equation}\label{AMAMtwistDef}
{I}(u) = \frac{1}{(n+1)V}\sum_{j=0}^{n} \int_X u \o^j_u \wedge \o^{n-j}, \ \ \  
{I}_{\textup{Ric}(\omega)}(u) = \frac{1}{n V}\sum_{j=0}^{n-1}\int_X u \textup{Ric}(\o) \wedge \o_u^{j} \wedge \o^{n-j-1}.\end{equation}

\paragraph{The $K^\beta$ energy and transcendental quantization.} Traditionally, the  most difficult term in \eqref{Kendef} is the entropy which is only lower semi-continuous (lsc) with respect to the $d_1$ metric of $\mathcal H_\omega$ (introduced in \cite{Dar15}). To deal with this, we propose to replace the entropy with a damper version, having much better properties. For $\beta>0$ we introduce the $\beta$-entropy:
\begin{equation}\label{eq: Ent_beta_def}
Ent^\beta(u) := \sup_{v\in\mathcal H_\omega}\left(-\log\int_Xe^{\beta(v-u)}\frac{\omega^n}{V}+\beta(I(v)-I(u))\right).
\end{equation}
Recalling the formula $Ent(u)=\sup_{w\in C^0(X)}\left(-\log\frac{1}{V}\int_Xe^{w}{\omega^n}+ \frac{1}{V}\int_X w \omega _u^n\right)$ (\cite[Proposition 2.10]{BBEGZ19}) and the basic inequality $I(v) - I(u) \leq \frac{1}{V}\int_X (v-u) \omega_u^n$, we immediately see that $Ent^\beta(u) \leq Ent(u)$. In Corollary \ref{cor: increasing_K_beta_conv} we will prove that $Ent^\beta(u) \nearrow Ent(u)$ as $\beta\nearrow\infty$.

It may seem surprising at first, but our formulation of the $\beta$-entropy is intimately tied to Berman's speculations  from \cite[\S 3.1]{Ber19}, where he suggests to view solutions $u^\beta$ to the following Aubin--Yau type Monge--Amp\`ere equation as `transcendental quantizations' of $u \in \mathcal H_\omega$, with $\beta^{-1}$ serving as the Planck constant:
\begin{equation}
    \label{eq:AY-eqn}
    (\omega+\ddc u^\beta)^n=e^{\beta(u^\beta-u)}\omega^n.
\end{equation}
Indeed, one can `linearize' this equation to obtain a formal asymptotic expansion for $u^\beta$:
\begin{equation}\label{eq: expansion}
u^{\beta}\sim u+\frac{1}{\beta}\log\frac{\omega^n_u}{\omega^n}+\frac{1}{\beta^2}(\tr_{\omega_u}\Ric(\omega)-S(\omega_u))+O\Big(\frac{1}{\beta^3}\Big),\ \beta\to\infty.
\end{equation}
Though not used (nor proved) in this work, this expansion gives a transcendental analogue of the classical Bergman kernel asymptotic expansions going back to Tian \cite{Tian89}. Intriguingly, the scalar curvature appears in the above expansion, as in the Bergman kernel expansion. For classical Bergman kernels, this point was crucially used by Donaldson \cite{Don01,D05}, showing that cscK metrics are closely related to the Chow stability in algebraic geometry.

For our $\beta$-entropy, the following simple but crucial formula is proved in Proposition \ref{prop:ent-beta=sup}:
\begin{equation}\label{eq: Ent_beta_formula}
Ent^\beta(u)=\beta(I(u^\beta)-I(u)).
\end{equation}
The expression on the right hand side is reminiscent of Donaldson’s $\tilde\cL$ function from \cite[(12)]{D05} in the classical Bergman kernel setting. Because of this analogy, we call $Ent^\beta$ the (transcendentally) quantized entropy at level $\beta$. The corresponding quantized $K^\beta$ energy functional $K^\beta: \mathcal H_\omega \to \mathbb{R}$—satisfying $K^\beta(u) \nearrow K(u)$ as $\beta \nearrow \infty$—is defined as follows:
\begin{equation}\label{Kbetadef}
K^\beta(u) = {Ent}^\beta(u) -n {I}_{\textup{Ric }\o}(u)+ n\bar S {I}(u).
\end{equation}

Due to Chen-Cheng, existence of cscK metrics is equivalent to properness of the K energy \cite{CC1,CC2}, which we show to be further equivalent with properness of the `quantized' $K^\beta$ functionals for $\beta$ big enough (Theorem \ref{thm:K-prop=K-beta-prop}). This is in line with the quantization philosophy extensively explored in the algebraic setting (see e.g. \cite{Tian89,Tia90,Don01,D05,PT06,Paul12,RTZ20,Zhang21YTD}). Our argument relies on estimates reminiscent of the partial $C^0$ estimate (Proposition~\ref{prop:PC0}), originally introduced by Tian \cite{Tia90} in the Fano setting.

In addition, we show that the $K^\beta$ energy and its radial version are continuous with respect to the $d_1$ and $d_1^c$ metrics, respectively (Lemma~\ref{lem: Ent_beta_cont}, Theorem~\ref{thm:Ent-beta=lim-Ent-beta}). The analogous properties are known to hold for the Ding energy (\cite[Proposition 5.19]{DR17}, \cite[Theorem 1.4]{DXZ25} \cite[Proposition 3.15]{BJ22}) but fail for the $K$ energy.  Thus, while the properness profile of $K^\beta$ agrees with that of the $K$ energy, its crucial continuity profile is in line with the much better behaving Ding energy from the Fano case.

Since, in the Fano case, the favorable properties of the Ding energy have driven important advances in $K$-stability, moduli theory, $\delta$-invariants, and other areas \cite{Bern15,CDS15,Tian15,Li17,Fuj19,FO18,BJ17,XuZhuang20,LXZ21,Zhang21YTD,DZ22}, we are motivated to make extensive use of $K^\beta$ energies in this work.

\paragraph{A YTD correspondence using $K^\beta$.} Given an ample line  bundle $L \to X$ with hermitian metric $h$ whose curvature equals our K\"ahler metric $\omega$, to test K-stability, one considers  test configurations. Historically this notion is due to Ding--Tian and Donaldson,  and involves a one parameter degeneration of the structure $(X,L)$ \cite{DT92,Tian97,Don02}. 
We give the precise definition of ample test configurations in \S \ref{sec:pre}. 
One may think of these as certain finitely generated $\mathbb{Z}$-filtrations of the section ring $R(X,L^k)$ for some $k \in \Bbb N$. This perspective was first considered in \cite{WN12,Sz15} and it was shown in \cite[Proposition 2.15]{BHJ17} that it is equivalent with the more classical notion introduced in \cite{Tian97,Don02}. 

To an ample test configuration $\mathcal T$, due to work initiated by Phong-Sturm, it is possible to associate a $C^{1,1}$ Mabuchi geodesic ray $[0,\infty) \ni t \to u^\mathcal T_t \in \mathcal H_\omega^{1,1}$ \cite{PS07a,PS10, RWN14,Ber16,CTW18}.

We introduce invariants associated with $K^\beta$ and $\mathcal T$, using asymptotic slopes along $\{u^\mathcal T_t\}$:
\begin{equation} \label{eq: K_beta_T_def}
K^\beta_\mathcal T := \liminf_{t \to \infty} \frac{K^\beta(u^\mathcal T_t)}{t}, \ \ J_\mathcal T:= \lim_{t \to \infty} \frac{J(u^\mathcal T_t)}{t}.
\end{equation}
Here $J(u) = \frac{1}{V} \int_X u \omega^n - I(u)$ is the well known $J$ functional, that is known to be convex. \smallskip

Given that $\mathcal T$ is algebraic, it is natural to seek for an interpretation of the associated invariants purely in algebraic terms. For instance, $J_{\mathcal T}$ admits such a description via certain intersection numbers attached to $\mathcal T$ \cite{PS08,BHJ17}. For $K^\beta_{\mathcal T}$, an explicit intersection-theoretic formula will be provided in Theorem \ref{thm: K_beta_intersection} below. In \S \ref{sec:slop-TC} we also provide a threshold type expression for $K^\beta_\cT$, that may open the door to new valuative criterions for K-stability, generalizing those in the Fano case \cite{Fuj19,FO18,BJ17,AZ20,Fano-3fold-book,Xubook}.

Given $\beta>0$, we say that $(X,L)$ is \emph{uniformly $ K^\beta$-stable} if there exists $\gamma>0$ such that
\begin{equation}\label{eq: K_stab_def}
K^\beta_\mathcal T \geq \gamma J_\mathcal T,
\end{equation}
for all ample test configurations $\mathcal T$.

The first main result of this paper is the following uniform YTD correspondence for cscK metrics. To our knowledge this is the first characterizing criterion with ample test configurations as test objects:

\begin{theorem} \label{thm:YTD'} Let $(X,L)$ be a polarized K\"ahler manifold. The following are equivalent. \smallskip

\noindent (i) $(X,L)$ admits a unique cscK metric in $c_1(L)$. \smallskip 

\noindent (ii) $(X,L)$ is uniformly $K^\beta$-stable for some $\beta>0$.
\end{theorem}

Given the nature of our techniques, our argument also yields a YTD correspondence for transcendental K\"ahler manifolds $(X,\omega)$ with virtually identical statement (see Theorem \ref{thm:YTD''}).

To explain the proof of Theorem \ref{thm:YTD'},
 let us recall a few concepts. By $(\mathcal E^1, d_1)$ we denote the $d_1$ metric completion of $\mathcal H_\omega$, explored in \cite{Dar15}, with the strongly topology of the space $\mathcal E^1$ already investigated in \cite{BBEGZ19}.
The space $(\mathcal E^1, d_1)$ is a geodesic metric space \cite[Theorem 2]{Dar15}. Our functionals $K$, $K^\beta$, and $J$ extend naturally to $\mathcal E^1$, preserving their $d_1$-(semi)continuity  properties \cite{BerBern17,BDL17}. Moreover, the extended $K^\beta$ continues to satisfy the formulas \eqref{eq: Ent_beta_def} and \eqref{eq: Ent_beta_formula}.

By $\mathcal R^1$ we denote the space of geodesic rays $t \mapsto u_t$ inside $\mathcal E^1$, that emanate from the background potential $0 \in \mathcal H_\omega$ (notation: $\{u_t \} \in \mathcal R^1$). Due to Buseman convexity of $d_1$ (\cite[Proposition 5.1]{BDL17}) one can introduce the chordal distance for $\{u_t\},\{v_t\} \in \mathcal R^1$ \cite{DL19,CC2}:
\begin{equation}
    \label{eq:def-chordal}
    d_1^c(\{u_t\},\{v_t\}) = \lim_{t \to \infty} \frac{d_1(u_t,v_t)}{t}.
\end{equation}
Thus defined, $(\mathcal R^1,d_1^c)$ is a complete geodesic metric space \cite[Theorem 1.4]{DL19}.

An important subspace of $(\mathcal R^1, d_1^c)$ is $\mathcal R^1_\mathcal I$, defined as the $d_1^c$-closure of rays $\{u_t^\mathcal T\}$ induced by test configurations $\mathcal T$. In \cite{BBJ18} the authors refer to such rays simply as maximal. This space admits a natural identification with $\mathcal E^{1,NA}$ (see \cite[Theorem 1.2]{DX20} and the discussion preceding it). For more detailed treatments of this topic, we refer to \cite{DL19, BBJ18, DX20, DarvasSurvey}.

Given $\{u_t\} \in \mathcal R^1$ it is convenient to introduce the radial version of our functionals:
$$K \{u_t\} = \lim_{t \to \infty} \frac{K(u_t)}{t}, \ \ \ K^\beta \{u_t\} = \liminf_{t \to \infty} \frac{K^\beta(u_t)}{t}, \ \ \ J \{u_t\} = \lim_{t \to \infty} \frac{J(u_t)}{t}.$$
In a recent exciting development, Chi Li showed that rays $\{u_t\} \in \mathcal R^1$ satisfying $K\{u_t\} < \infty$ must lie in $\mathcal R^1_\mathcal I$ \cite[Theorem 1.2]{Li20-cscK}. \smallskip

The proof of Theorem \ref{thm:YTD'} relies on four central components. From the existing literature, we will need that existence cscK metrics is equivalent with properness/geodesic stability of the K energy \cite{CC1,CC2}, together with the aforementioned result of  Chi Li \cite{Li20-cscK}. The remaining two components are facts we establish in this work: that the properness of the K energy is equivalent to the properness of the \textbf{$K^\beta$} energy for sufficiently large $\beta$ (Theorem \ref{thm:K-prop=K-beta-prop}); and that the radial $K^\beta$ energy is $d_1^c$-continuous (Theorem \ref{thm:Ent-beta=lim-Ent-beta}).

Once the key ingredients are in place, the argument unfolds naturally. 
We present it here, as this simple proof serves as the skeleton of the paper (cf.~\cite{BBJ18,Fuj19} for related ideas in the Fano case involving the Ding energy).

\begin{proof}[Proof of Theorem \ref{thm:YTD'}] Assume that (i) holds. By  \cite[Theorem 1.5]{BDL20} the K energy is proper, i.e., $K(u) \geq \gamma J(u) - C, \ u \in \mathcal E^1$ for some $\gamma,C$. By Theorem \ref{thm:K-prop=K-beta-prop} we obtain that $K^\beta$ is proper too, for $\beta > 0$ big enough. By our definition of $K^\beta_\mathcal T$ and $J_\mathcal T$, (ii) immediately follows.

Now we assume that (ii) holds. By \cite[Theorem 1.2]{Li20-cscK} any ray with bounded K energy slope is in $\mathcal R^1_\mathcal I$, i.e., it is approximable by test configurations. Thus, by \cite{CC2} (cf. \cite[Theorem 1.9]{DL19}), it suffices to show existence of $\gamma>0$ such that
$K\{u_t\} \geq \gamma J\{u_t\}$ for $\{u_t\} \in \mathcal R^1_\mathcal I.$
Let $\{u_t\} \in \mathcal R^1_\mathcal I$ and $\{u^{\mathcal T_j}_t\}$ be $d_1^c$-approximating rays of ample test configurations $\mathcal T_j$.

Condition (ii) implies that for some $\beta>0$ and $\gamma>0$ we have that 
$$ K^\beta\{u^{\mathcal T_j}_t\} = K^\beta_{\mathcal T_j} \geq \gamma J_{\mathcal T_j} = \gamma J\{u^{\mathcal T_j}_t\}.$$
The functional $J\{\cdot\}$ is known to be $d_1^c$-continuous (as follows from \cite[Lemma 4.15]{Dar15} and \cite[Lemma A.2]{BBJ18}). $K^\beta\{\cdot\}$ also has this property by Theorem \ref{thm:Ent-beta=lim-Ent-beta}. Thus, using $K \geq K^\beta$  (see \eqref{eq: K_beta_conv}), we can let $j \to \infty$ in the above inequality to obtain 
$ K\{u_t\} \geq K^\beta\{u_t\} \geq \gamma J\{u_t\}.$
\end{proof}

\paragraph{Algebraic interpretation of $K^\beta_\mathcal T$.}

As Theorem \ref{thm:YTD'} introduces a new stability condition in the Yau–Tian–Donaldson framework, the second goal of this paper is to better understand the meaning of $K^\beta$-stability and connect it with more classical notions of stability.

To this end, we shall work with the slightly more general class of semi-ample test configurations, which are just pull-backs of ample ones (see \S \ref{sec:pre}) and are birationally more convenient to work with. Given such test configuration $\cT=(\cX,\cL)$, after possibly composing with blowups, we can further assume that $\mathcal X$ is smooth and dominates $X\times\PP^1$ via a map $\pi$. 

Then  $K_{\mathcal T} := K\{u^\cT_t\}$ agrees with the non-Archimedean K energy \cite{BHJ19}, \cite[(8)]{Li20-cscK}:
\begin{equation}
    \label{eq:K-T-int-formula}
    K_\cT=V^{-1}\cL^n\cdot K^{\log}_{\cX/X\times\PP^1}+\frac{\bar S\cL^{n+1}}{(n+1)V}+\frac{K_X\cdot\cL^{n}}{V},
\end{equation}
which equals the Futaki invariant of $\cT$ if the central fiber of $\mathcal X$ is reduced \cite[(0.1)]{BHJ17}. Here 
$$
K^{\log}_{\cX/X\times\PP^1}:=K_\cX-\pi^*K_{X\times\PP^1}+\cX_{0,\mathrm{red}}-\cX_0 $$
is the relative log canonical divisor (see \cite[\S 4.4]{BHJ17}). 

Next we show that $K^\beta_\mathcal T$ also admits a purely algebraic expression. The main difficulty is that the explicit expression for $K^\beta$ (see \eqref{eq: Ent_beta_formula} and \eqref{Kbetadef}) involves solving the transcendental Aubin–Yau type equation \eqref{eq:AY-eqn}.
 This will be overcome using the next key concept.

Given a smooth test configuration $\mathcal T = (\mathcal X, \mathcal L)$ dominating $X\times\PP^1$, we call $\mathcal T_\beta := (\mathcal X,\mathcal L_\beta)$ the \emph{log discrepancy model} of $\cT$ at level $\beta$, where
\begin{equation}\label{eq: L_beta_def}
\mathcal L_\beta := \mathcal L + \frac{1}{\beta}K^{\log}_{\mathcal X/X \times \Bbb P^1},\ \beta\in\QQ_{>1}.
\end{equation}
The $\QQ$-line bundle $\mathcal L_\beta$ thus defined is typically not relatively semi-ample, making $\mathcal T_\beta = (\mathcal X, \mathcal L_\beta)$ a \emph{model} of $(X, L)$ in the sense of Chi Li \cite[Definition 2.1]{Li20-cscK} (see  \S\ref{sec:pre}). 

\begin{theorem}[Theorem \ref{thm:K_beta_intersection'}]\label{thm: K_beta_intersection}
For a test configuration $\mathcal T = (\mathcal X, \mathcal L)$ as above, and  $\beta \in \mathbb{Q}_{>1}$, we have the following:
\begin{equation}
\label{eq:K-beta-T-intersection-formula_main}     
K^\beta_\cT=\frac{\beta(\langle\cL_\beta^{n+1}\rangle-\cL^{n+1})}{(n+1)V}+\frac{\bar S\cL^{n+1}}{(n+1)V}+\frac{K_X\cdot\cL^{n}}{V}.
\end{equation}
Moreover, $K^\beta_\cT \nearrow  K_\cT$ as $\beta \nearrow \infty$.
\end{theorem}

The term $\langle \mathcal L_\beta^{n+1} \rangle$ appearing in \eqref{eq:K-beta-T-intersection-formula_main} denotes the volume of the (relatively) big $\QQ$-line bundle $\mathcal L_\beta$, which can be expressed using movable intersection products (see \eqref{eq:def-mov-top-prod} and \cite{BFJ09,Li21}). \smallskip

The proof of Theorem \ref{thm: K_beta_intersection} relies on a slope formula for $Ent^\beta\{u_t\}:= \liminf_{t \to \infty} Ent^\beta(u_t)/t$, for any bounded geodesic ray $\{u_t\}$. It is a radial version of the identity \eqref{eq: Ent_beta_def} (Theorem \ref{thm:rad-Ent-beta=sup}):

\begin{equation}\label{eq: ent_beta_radial_intr}
Ent^\beta\{u_t\}=\sup_{\{v_t\}}\left(L^\beta(\{v_t\},\{u_t\})+\beta(I\{v_t\}-I\{u_t\})\right),
\end{equation}
where the sup is over geodesic rays $\{v_t\}$ arising from test configurations, and $L^\beta(\{v_t\},\{u_t\})$ denotes the following radial Ding type functional:
$$
L^\beta(\{v_t\},\{u_t\}):=\liminf_{t\to\infty}\left(-\frac{1}{t}\log\int_Xe^{\beta(v_t-u_t)}\frac{\omega^n}{V}\right).
$$

To prove \eqref{eq:K-beta-T-intersection-formula_main}, we will show in \S \ref{sec:int-for-of-K-beta} that, when the ray $\{u_t\}$ is induced by a test configuration $\mathcal T = (\mathcal X, \mathcal L)$, the sup in \eqref{eq: ent_beta_radial_intr} is actually attained by the geodesic ray $\{v_t\}$ associated with the log discrepancy model $\mathcal T_\beta = (\mathcal X, \mathcal L_\beta)$.

\paragraph{Interactions with the non-Archimedean point of view.} In the last part of our paper we point out the connections between our $K^\beta$-stability and the non-Archimedean approach to K-stability pioneered by Boucksom--Jonsson and collaborators \cite{BBJ18, BHJ17, BoJ18,BoJ21a,BJ22,BJ25preprint}. \smallskip

For details on terminology and non-Archimedean geometry we refer to \cite{BHJ17} and \S\ref{sec:pre}.
When dealing with non-Archimedean data, it is possible to describe the radial $Ent^\beta$ functional using a concise formula.

Let $u \in \mathcal H^{NA}$ be a non-Archimedean metric, as defined in \cite{BHJ17}, encoding the data of test configurations using divisorial information. As such, to $u$ one can again associate a $C^{1,1}$ geodesic ray $\{u_t\}$. As a consequence of \eqref{eq: ent_beta_radial_intr}, we obtain the following formula for the radial $\beta$-entropy of $\{u_t\}$ (Corollary \ref{cor: L_beta_L_NA}):

\begin{equation}\label{eq: ent_NA_rad}
    Ent^\beta\{u_t\}=\sup_{w \in \mathcal H^{NA}}\left(L^{NA}(\beta (u-w))+\beta(I^{NA}(w)-I^{NA}(u))\right).
\end{equation}
where $L^{NA}$ and $I^{NA}$ are non-Archimedean functions defined in \cite[Section 2]{BoJ18}. As we will show in \S \ref{sec:NA}, due to basic non-Archimedean formalism, the above further implies that
$$Ent^\beta\{u_t\} \leq Ent^{NA}(u), \ \ u \in \mathcal E^{1,NA},$$
where $Ent^{NA}(u)$ denotes the non-Archimedean entropy introduced in \cite{BHJ17,BoJ18}.
Adding the terms coming from the radial Monge--Amp\`ere energy and its twisted version, we immediately obtain that $K^\beta\{u_t\} \leq K^{NA}(u), \ \ u \in \mathcal E^{1,NA}.$

The usefulness of this inequality is similar to that of the `Archimedean' counterpart  $K^\beta \leq K$, already discussed. Sending $\beta\nearrow\infty$ and using $K^\beta\{u_t\}\nearrow K\{u_t\}$ (Theorem \ref{thm: K_beta_conv}), we obtain that $K\{u_t\} \leq K^{NA}(u)$. Since the reverse inequality is known \cite[Theorem 1.2]{Li20-cscK}, we obtain an alternative argument for the following key formula of Boucksom--Jonsson \cite{BJ25preprint}:
\begin{equation}\label{eq: BJ_slope_formula}
    K\{u_t\} = K^{NA}(u), \ \ u \in \mathcal E^{1,NA}.
\end{equation}
Coupled with the results of \cite{Li20-cscK} and additional non-Archimedean techniques, this formula yields the following YTD correspondence, also obtained in \cite{BJ25preprint}:
\begin{theorem} \textup{\cite{BJ25preprint}} \label{thm:BJ25} Let $(X,L)$ be a polarized manifold. The following are equivalent:\smallskip\\
\noindent (i) $(X,L)$ admits a cscK metric in $c_1(L)$. \smallskip\\
\noindent (ii) $\GG:=\mathrm{Aut}_0(X,L)$ is reductive and $(X,L)$ is $\GG$-uniform K-stable with respect to models.
\end{theorem}

For the definition of $\GG$-uniform K-stability (with respect to test configurations, models, or log discrepancy models), we refer to \eqref{eq: def_K-stab_log_discrep}. This result confirmed a conjecture of C.~Li, who previously proved the existence direction in \cite{Li20-cscK}. In fact, \cite{BJ25preprint} carries out a considerably more refined analysis, establishing YTD-type correspondences involving \(K\)-polystability and prescribed symmetry in the framework of weighted cscK metrics as well.  

By employing the framework provided by $K^{\beta}$-stability, we will refine Theorem \ref{thm:BJ25},  demonstrating that the existence of cscK metrics can be tested using only the distinguished subclass of models introduced in \eqref{eq: L_beta_def} (see Theorem \ref{thm:YTD'''} below). This will be achieved by making progress on a well-known regularization conjecture.

As emphasized in a series of papers \cite{BoJ18,BoJ21a,BJ22,Li20-cscK,Li21} of Boucksom--Jonsson and Chi Li, the remaining step for addressing the strong form of the uniform YTD conjecture -- as formulated in \cite{Sz15,Der16,BHJ17, Li20-cscK} -- is to establish a regularization property of the radial entropy:

\begin{conjecture}
    [Boucksom--Jonsson]
    \label{conj:BJ}
    Let $(X,L)$ be a polarized manifold with $\omega\in c_1(L)$. Let $\{u_t\}\in\cR^1$ be a geodesic ray with $Ent\{u_t\}<\infty$. Then there exists a sequence of geodesic rays $\{v_{t,k}\}$ arising from test configurations of $(X,L)$ such that
\[
\lim_{k\to\infty}d_1^c(\{u_t\},\{v_{t,k}\})=0 \quad \text{and} \quad Ent\{u_t\}=\lim_{k\to\infty}Ent\{v_{t,k}\}.
\]
\end{conjecture}

Regarding this conjecture, remarkable progress was made by Chi Li \cite[Proposition 6.3]{Li20-cscK}, showing that one can carry out the approximation via a sequence of geodesic rays $\{v_{t,k}\}$ attached to models. Li constructs these approximating models implicitly, by solving non-Archimedean Calabi--Yau equations \cite{BFJ16}.

Using our $\beta$-entropy we refine Li’s approximation result and make the construction more explicit. Specifically, we construct a regularizing sequence of log discrepancy models obtained as small perturbations of the test configurations arising from the classical multiplier ideal sheaf approximation procedure of \cite{BBJ18} (recalled in  \S\ref{sec:pre})).

\begin{theorem}[Theorem \ref{thm:BJ-conj_main1}]
\label{thm:BJ-conj_main}
Let $(X,L)$ be a polarized manifold with $\omega\in c_1(L)$. Let $\{u_t\}\in\cR^1$ be a geodesic ray with $Ent\{u_t\}<\infty$. Then there exists a sequence of smooth test configurations $(\cX^{(k)},\cL^{(k)})$ dominating $X\times\PP^1$ and rational numbers $\beta_k\to\infty$ such that the geodesic rays $\{v_{t,k}\}$ attached to the log discrepancy models $(\cX^{(k)},\cL^{(k)}_{\beta_k})$ satisfy
\[
\lim_{k\to\infty}d_1^c(\{u_t\},\{v_{t,k}\})=0 \quad \text{and} \quad Ent\{u_t\}=\lim_{k\to\infty}Ent\{v_{t,k}\}.
\]
\end{theorem}

With this regularization theorem in hand, the techniques of \cite{DR17} and \cite{DL19} immediately imply the following $\GG$-uniform YTD correspondence, sharpening Theorem \ref {thm:BJ25}:

\begin{theorem}[Theorem \ref{thm:YTD'''1}]
\label{thm:YTD'''} Let $(X,L)$ be a polarized K\"ahler manifold. The following are equivalent:\smallskip\\
\noindent (i) $(X,L)$ admits a cscK metric in $c_1(L)$. \smallskip\\
\noindent (ii) $\GG:=\mathrm{Aut}_0(X,L)$ is reductive and $(X,L)$ is $\GG$-uniform K-stable with respect to log discrepancy models.
\end{theorem}

\begin{remark}Fix any $\beta_0\gg 1$. Our argument shows that for condition (ii), it is sufficient to consider log discrepancy models $(\mathcal X, \mathcal L_\beta)$ with $\beta>\beta_0$ (recall \eqref{eq: L_beta_def}). Thus, it is enough to work with models that are in some sense `arbitrarily close' to being test configurations.
That being said, it remains unclear to what extent this refinement could lead to verifiable criteria for stability in concrete situations, suggesting avenues for future research.
\end{remark}

In light of the above theorems, to resolve the strong form of the ($\GG$-)uniform YTD conjecture with respect to test configurations (see, for instance, \cite[Conjecture 2.26]{Li20-cscK}), it remains to determine whether one can regularize geodesic rays $\{u_t\}$ associated with log discrepancy models in the sense of Conjecture \ref{conj:BJ}. At first sight, this appears to be more realistic, but time will tell whether the concrete structure of these log-discrepancy models will play a decisive role in the resolution of Conjecture \ref{conj:BJ}.

It would also be desirable to extend Theorem \ref{thm:YTD'''} to the transcendental setting. This seems to require more work particularly when $\GG$ is non-trivial, since an equivariant non-Archimedean theory for Kähler manifolds has yet to be fully developed (cf. \cite{DXZ25,Pic24,MPWN}).

To the best of our knowledge, our \(K^\beta\) functionals have not been previously considered in the K-stability literature. We note, however, that \cite[Section 4.1]{CLP} considered a similar-flavored functional in their alternative proof of the convexity of the \(K\) energy. Their construction involves an additional Banach-Saks averaging, together with a different type of quantization in the definition of their $\mathcal M_k$ functional. At present, it remains unclear whether variants of our $K^\beta$ functionals yield other novel notions of stability.

Although we did not make explicit use of the expansion \eqref{eq: expansion}, it provided key motivation for our extensive reliance on transcendental quantization. Such expansions have already played an important role in influential studies of Kähler–Einstein type metrics by Wu and Jeffres-Mazzeo-Rubinstein \cite{Wu08}, \cite[Section 9]{JMR} (cf. \cite[(6.3)]{Rub_survey}), where they were employed to initiate a Ricci continuity path from $-\infty$, following a suggestion of Tian–Yau \cite{TY87}.

In this work we did not study the more general extremal metrics and their weighted versions. Instead, we refer the reader to \cite{HL25,Hash25,BJ25preprint} for exciting recent developments in this direction.

\paragraph{Acknowledgments.} K.Z. is especially grateful to J. Chu for discussions on the expansion of $u^\beta$ back in 2021. 
We are grateful to P. Mesquita Piccione, J. Han, M. Xia for answering numerous questions, and we thank Y. Liu, Y. Odaka, Y.A. Rubinstein for suggestions that improved the presentation. 

T.D. was partially supported by a Sloan Fellowship, Simons Fellowship and NSF grant DMS-2405274. This material is based upon work supported by NSF grant DMS-1928930, while T.D. was in residence at the Simons Laufer Mathematical Sciences Institute  during Fall 2024. K.Z. is supported by NSFC grants 12571060, 12271038 and 12271040.

Both authors visited Chalmers University in Spring 2025, where part of this work was completed. We thank the mathematics department for the hospitality.

During the final stage of this project, we became aware of related work announced by Boucksom–Jonsson \cite{BJ25preprint}. Until then we had not considered working with Chi Li's models, or exploring connections with non-Archimedean geometry, but afterwards we naturally wondered about the relationships between their work and ours. This line of thought led to our alternative argument for their key formula \eqref{eq: BJ_slope_formula} using $K^\beta$ energies, and the $\Bbb G$–uniform YTD correspondence involving log discrepancy models (Theorem \ref{thm:YTD'''}), making contact with their Theorem \ref{thm:BJ25}. We are grateful to them for sharing their preprint \cite{BJ25preprint} and also for related discussions. As far as we can tell, our approaches are fundamentally different. The work of Boucksom–Jonsson is a culmination of the non-Archimedean approach to K-stability, developed over the past decade (see \cite{BHJ17, BBJ18, BoJ18, BJ22} etc.), and relies on an approximation argument at the level of model filtrations/non-Archimedean data, using \cite{BD12, DT24}. In contrast, our approach is largely `Archimedean'. We introduce novel $K^\beta$ energies that approximate the $K$ energy directly via transcendental quantization at the level of potentials, inspired by \cite{Ber19}.

Paralleling our transcendental YTD correspondence  Theorem \ref{thm:YTD''}, we acknowledge independent work by Mesquita--Piccione and Witt Nystr\"om \cite{MPWN}, proving a different kind of transcendental YTD existence theorem, more in the spirit of \cite{Li20-cscK}, obtained through the resolution of non-Archimedean Monge--Amp\`ere equations.

\section{Preliminaries}
\label{sec:pre}
We collect some standard results from the literature that will be used in this work. 
Throughout this paper $(X,\omega)$ denotes a compact K\"ahler manifold of dimension $n$.

\paragraph{Pluripotential theory.}

The space of smooth K\"ahler potentials is defined as follows:
$$\mathcal H_\omega := \{u \in C^\infty(X,\RR) \textup{ s.t. } \omega_u :=\omega + \ddc u >0 \},$$
where $\ddc:=\frac{i}{2\pi}\partial\bar \partial.$ The metric completion of $\mathcal H_\omega$ with respect to the $d_1$ metric \cite{Dar15} is denoted by $\cE^1$, which is called the space of finite energy potentials \cite{GZ07}. More precisely, given $u,v\in\cE^1$, the $d_1$-distance between $u$ and $v$ is given as
$
d_1(u,v)=I(u)+I(v)-2I(P(u,v)),
$
where $I(\cdot)$ denotes the Monge--Amp\`ere energy (recall \eqref{AMAMtwistDef}) and $P(u,v) \in \mathcal E^1$ is the potential theoretic minimum of $u$ and $v$, i.e., the largest potential in $\cE^1$ lying below $u$ and $v$. In particular, $d_1(0,u)=-I(u)$ when $u\leq 0$.

Given $a,b \in [-\infty,\infty]$ with $a < b$, a path $(a,b)\ni t\mapsto u_t
\in \cE^1$ is called a finite energy subgeodesic segment if $U_z(x):=u_{-\log |z|^2}(x)$ (with $x \in X$ and $z$ being a complex) satisfies
 $
 p_1^*\omega+\mathrm{dd}^c_{X\times\CC}U_z\geq 0
 $
 as a current on $X\times\{e^{-a/2}<|z|<e^{-b/2}\}$, where $p_1:X\times\CC\to X$ is the projection map.

Let $h_0,h_1 \in \mathcal E^1$. As shown in \cite{Dar15}, a distinguished $d_1$-geodesic $(0,1) \ni t \mapsto h_t 
\in \mathcal 
E^1_\o$ connecting $h_0,h_1$ can be obtained as Perron envelope of finite energy subgeodesics:
\begin{equation}
\label{fegeod}
h_t := \sup \{v_t \ | \ (0,1) \ni l \ \mapsto v_l  \in \mathcal E^1 \textup{ is a subgeodesic with } \limsup_{l \to 0,1} v_l \leq h_{0,1}\}, \ t \in (0,1).
\end{equation}
We call \eqref{fegeod}
the  \emph{finite energy geodesic} connecting $h_0,h_1$.

In case $a=0$ and $b=\infty$ then the curve $(0,\infty) \ni t \mapsto u_t \in \cE^1$ is called a finite energy subgeodesic ray, which we will denote by $\{u_t\}$. All rays in this paper will be assumed to emanate from $u_0= 0\in\cH_\omega$, meaning that $d_1(u_t,0) \to 0$ as $t\searrow 0$. It is well known that $t\mapsto I(u_t)$ is convex (resp. affine linear) when $\{u_t\}$ is a subgeodesic (resp. geodesic). 

A subgeodesic ray $\{u_t\}$ is called sublinear if $u_t\leq Ct$ for some $C>0$. Recall that a geodesic ray is automatically sublinear due to \cite[Theorem 3.7(iv)]{DX20}.
Finally, following the formalism of \cite{DL19}, by $\cR^1$ we denote the space of finite energy geodesic rays emanating from $0$, which is complete with respect to the chordal metric $d_1^c$ defined in \eqref{eq:def-chordal} \cite[Theorem 1.4]{DL19}.

A geodesic ray $\{u_t\} \in \mathcal R^1$ is \emph{bounded} if $u_t \in \mathcal E^1 \cap L^\infty, \ t  \geq 0$. \cite[Theorem 1]{Dar17c} implies that such rays are automatically linearly bounded: there exists $C>0$ such that $|u_t | \leq C t, \ t \geq 0$.

For more comprehensive works on the above topics we refer to \cite{DarvasSurvey, DL19}.

\paragraph{Test configurations and geodesic rays.} For the rest of this section we assume that $(X,\omega)$ is polarized, with $\omega\in c_1(L)$ for some ample line bundle $L$. 

Following \cite[\S 3.2]{BBJ18}, an \emph{ample test configuration} $\cT:=(\cX,\cL)$ of $(X, L)$ is a $\CC^*$-equivariant partial compactification over $\CC$ of $(X, L)\times\CC^*$ which consists of a normal variety $\cX$, a flat projective morphism $\pi: \cX \to\CC$, a relatively ample $\QQ$-line bundle $\cL\to \cX$, a $\CC^*$-action on $(\cX,\cL)$ lifting the standard one on $\CC$, and an identification of the fiber $(\mathcal X_1, \mathcal L_1)$ with $(X, L)$. 

In practice, one often encounters \emph{semi-ample test configurations}, for which $\cL$ is only assumed to be relatively semi-ample. For brevity, we will refer to these simply as test configurations. By \cite[\S 2.5]{BHJ17}, such a generalization is hardly essential. Indeed, a semi-ample test configuration 
is simply the pullback of an ample one, and they both induce the same geodesic ray. As a result, in this work, slightly abusing precision, we think of ample and semi-ample test configurations as the same objects, and just call them as \emph{test configurations}.

Having roots in works of Y. Odaka, one can alternatively describe test configuration using certain flag ideals \cite{Od13,Oda13,BHJ17,BBJ18}. A flag ideal $\mathfrak{a}$ is a $\CC^*$-invariant coherent ideal sheaf on $X \times\CC$. Assume in addition that for some $m>0$, the sheaf $p_1^*(mL)\otimes\mathfrak{a}$ is globally generated on $X\times\CC$. Take the normalized blowup $\cX'\xrightarrow{\mu}X\times\CC$ along $\mathfrak{a}$ with exceptional divisor denoted by $E$. Set $\cL':=\mu^*p_1^*L-\frac{1}{m}E$. Then $(\cX',\cL')$ is a test configuration for $(X,L)$. 

Conversely (see \cite[\S 2.6]{BHJ17}), for any test configuration $(\cX,\cL)$, after replacing $\cL$ with $\cL+a\cX_0$ for appropriate $a\in\QQ$ (where $\cX_0\subset\cX$ denotes the central fiber over $0\in\CC$), the resulting test configuration has the same Futaki invariant, but induces a flag ideal $\mathfrak{a}$ on $X\times\CC$ such that the sheaf $p_1^*(mL)\otimes\mathfrak{a}$ is globally generated  for some divisible $m>0$. This ideal gives rise  to a test configuration $(\cX',\cL')$ with a  morphism $\cX'\xrightarrow{\tau}\cX$ such that $\tau^*\cL=\cL'$.

Most importantly, to any test configuration $\cT=(\cX,\cL)$ one can attach a $C^{1,1}$ geodesic ray $\{u_t^\cT\}$ \cite{PS07a,PS10, RWN14,Ber16,CTW18}. After adding a multiple $a\cX_0$ to $\cL$ (which corresponds to adding $at$ to $u_t^\cT$),  
we can assume that $(\cX,\cL)$ is obtained via a flag ideal $\mathfrak{a}$ as described above. Then from the construction of $\{u_t^\cT\}$ the singularity type of $U^\cT_z:=u^\cT_{-\log|z|^2}$ near $X\times\{0\}$ is comparable to $\mathfrak{a}^{1/m}$, in the sense that
\begin{equation}
    \label{eq:U-has-analytic-sing}
    U^\cT_z=\frac{1}{m}\log \sum_i|f_i|^2+O(1),
\end{equation}
where $\{f_i\}$ is a finite set of local generators of $\mathfrak{a}$. For this reason, $\{u_t^\cT\}$ is called a geodesic ray with algebraic singularities, in the terminology of \cite[\S 4.4]{BBJ18}. 

\paragraph{Compactification and domination.}
In practice, it is also useful to compactify  a test configuration $(\cX,\cL)$ trivially over $\infty\in\PP^1$ (see \cite[Definition 2.7]{Xubook}), which we still denote as $(\cX,\cL)$. In this paper we assume that all test configurations $(\cX,\cL)$ are compactified in this sense, making intersection theory more convenient on $\cX$. In this context $\cL$ being relatively (semi)ample means that $\cL+c\cX_0$ is (semi)ample for some $c\in\QQ_{>0}$.

Most of our discussions will involve birational invariance, hence we will often assume that the total space $\cX$ of a test configuration $(\cX,\cL)$ is \emph{smooth}, where $\cX_0$ has simple normal crossing support. This can always achieved after taking a resolution (which does not affect the associated geodesic ray). Such $(\cX,\cL)$ is called a smooth model (of the original test configuration). 

Given two test configurations $(\cX,\cL)$ and $(\cX',\cL')$, we say $(\cX',\cL')$ dominates $(\cX,\cL)$ if there is a $\CC^*$-equivariant birational morphism $\cX'\xrightarrow{\pi}\cX$. So any test configuration is dominated by a smooth model. 

Let $(\cX,\cL)$ be a smooth test configuration dominating $X\times\PP^1$, say $\cX\xrightarrow{\pi}X\times\PP^1$. Denote by $p_1:X\times\PP^1\to X$ the projection map. By \cite[Proposition 3.10]{SD18} there is a uniquely determined $\QQ$-Cartier divisor $D$ supported in $\cX_0$ such that
\begin{equation}
    \label{eq:def-D}
    \pi^*p_1^*L-\cL\sim_\QQ D.
\end{equation}
Let $\{u_t\}$ be the geodesic ray associated to $(\cX,\cL)$. By the analysis in \cite[Lemma 4.6]{SD18} and \cite[\S 3]{CTW18} the singularity type of $U_z:=u_{-\log|z|^2}$ when pulled back to $\cX$ is $L^\infty$-compatible with the divisor $D$, in the sense of \cite[Definition 4.5]{SD18} (up to a sign). When $D$ is effective, this means that we can find an $S^1$-invariant qpsh function $\psi_D$ in a neighborhood of $\cX_0$ such that
\begin{equation}
    \label{eq:sing-type-of-alg-ray}
    \theta_D+\ddc\psi_D=[D], \text{ and }|\pi^*U_z-\psi_D|=O(1),
\end{equation}
where $\theta_D$ is a smooth $(1,1)$-form on $\cX$ representing the cohomology class $\{D\}$.

\paragraph{Demailly regularization of rays.}
We recall an important approximation procedure going back to \cite{DEL00,BBJ18}. Let $\{u_t\}$ be a finite energy subgeodesic ray with $u_t\leq 0$. Put $U:=U_z=u_{-\log|z|^2}$, which extends to a quasi-plurisubharmonic (qpsh) function on $X\times\Delta$. 

For any $m\geq 1$, let $\mathfrak{a}_m:=\cI(mU)$ be the multiplier ideal sheaf of $mU$. Note that $\mathfrak{a}_m$ is trivial away from $X\times\{0\}$ (by \cite{Ze01}), so it extends to a flag ideal on $X\times\CC$. As shown in \cite[Lemma 5.6]{BBJ18}, one can find a sufficiently large $m_0>0$, depending only on $(X,L)$, such that $p_1^*((m+m_0)L)\otimes\mathfrak{a}_m$ is globally generated for any $m\geq 1$. Let $\cX_m\xrightarrow{\mu_m}X\times\CC$ be the normalized blowup along $\mathfrak{a}_m$ with exceptional divisor $E_m$. Put $\cL_m:=\mu_m^*p_1^*L-\frac{1}{m+m_0}E_m$. Then $\cT_m:=(\cX_m,\cL_m)$ is a test configuration of $(X,L)$, inducing a geodesic ray $\{u_t^{\cT_m}\}$. Using \eqref{eq:U-has-analytic-sing}, one then observes that
\begin{equation}
    \label{eq:int-ineq-of-Li}
    \int_{X\times\Delta} e^{(m+m_0)U_z^{\cT_m}-mU_z}\omega^n\wedge i dz\wedge d\bar z<\infty.
\end{equation}
Using this simple fact, a remarkable result was obtained by Li \cite{Li20-cscK}: if $\{u_t\} \in \mathcal R^1$ is  a geodesic ray with $Ent\{u_t\}<\infty$, then it is maximal in the sense of \cite[Definition 6.5]{BBJ18}. More specifically, we have  $d_1(\{u_t^{\cT_m}\}, \{u_t\}) \to 0\text{ as }m\to\infty$ (cf. \cite[Lemma 5.7]{BBJ18}, \cite[Theorem 4.4.1]{Reb23}).

\paragraph{Models and non-Archimedean geometry.}
In this work we also need to work with \emph{models}, introduced in \cite[Definition 2.7]{Li20-cscK}. These are generalizations of test configurations $(\cX,\cL)$, where $\cL$ is only relatively big on $\cX$ (namely, $\cL+c\cX_0$ is big for some $c>0$). One can also naturally attach a geodesic ray to a model, since a model determines a filtration (as shown in \cite[Definition 2.7]{Li20-cscK}), and a filtration in turn also determines a geodesic ray \cite[\S 9]{RWN14}, which is constructed as an increasing limit of a sequence of geodesic rays of test configurations. As a consequence, the singularity type of such a geodesic ray is typically no longer of algebraic type (c.f \eqref{eq:sing-type-of-alg-ray}), but instead shall be seen as certain envelope of algebraic singularities. A convenient way of describing it is to use some terms from non-Archimedean geometry, that we now turn to explain. \smallskip

Let $X$ be a projective manifold.
Following \cite{BJ18v1}, denote by $X^{an}$ the Berkovich analytification of $X$ with respect to the trivial absolute value on the ground field $\CC$. We will mostly work with the subspace $X_\QQ^{div}$ consisting of $\QQ$-valued divisorial valuations on $X$. Elements in $X^{div}_\QQ$ can be explicitly described as follows. Any $w\in X^{div}_\QQ$ is of the form $\lambda\ord_F$, with $\lambda\in\QQ_{\geq 0}$ and $F\subset Y\xrightarrow{\tau}X$ being some prime divisor over $X$. When $\lambda=0$, we get the trivial valuation $v_{triv}.$ Moreover, we define the log discrepancy of such $w$ to be
$$
A_X(w)=A_X(\lambda \ord_F):=\lambda(1+\ord_F(K_Y-\tau^*K_X)).
$$
The notion of log discrepancy further extends in a linear way to the more general quasi-monomial valuations in $X^{an}$ (see e.g. \cite[Definition 1.34]{Xubook}).

Let $(X,L)$ be a polarized K\"ahler manifold. Let $\cH^{NA}$ be the space of non-Archimedean potentials coming from   test configurations of $(X,L)$ (see also \cite[\S 3.2]{BBJ18}). Any element of $\cH^{NA}$ is of the form $u^\cT$, where $\cT$ is some test configuration $\cT=(\cX,\cL)$ of $(X,L)$. Here $u^\cT$ is a continuous function on $X^{an}$, whose value on any $w\in X^{an}$ is explicitly determined by
\begin{equation}
    \label{eq:def-NA-u-T}u^{\cT}(w):=-\sigma(w)(D),
\end{equation}
where $\sigma(w)$ denotes the Gauss extension of $w$ (see \cite[Definition 4.3]{BHJ17}) and $D$ is given by \eqref{eq:def-D}. When $w\in X^{div}_\QQ$, it follows from \cite[Theorem 4.6]{BHJ17} that, up to passing to a higher model, $\sigma(w)$ is of the form $\frac{1}{b_E}\ord_E$ for some irreducible component $E$ in $\cX_0$, with $b_E:=\ord_E(\cX_0)$. Moreover, by \cite[Theorem A.10]{BJ22} we have the relation
$
A_X(w)=\frac{1}{b_E}A_{X\times\PP^1}(E)-1.
$
If we assume in addition that the geodesic ray $\{u_t^\cT\}$ attached to $\cT$ satisfies $u_t^\cT\leq 0$, then $u^\cT$ can be alternatively defined using Lelong numbers (see \cite[Definition 4.2]{BBJ18}):
\begin{equation}
    \label{eq:nu-T-along-E}
    u^\cT(w)=-\frac{1}{b_E}\nu(u^\cT_{-\log|z|^2},E).
\end{equation}
Here recall that $\nu(\phi,E)$ denotes the Lelong number of a qpsh function $\phi$ along $E$, which is defined as
$$
\nu(\phi,E):=\inf_{x\in E}\sup\{c>0: \phi(y)-c\log|y-x|^2\text{ is bounded above near }x\}.
$$
(Our convention of the Lelong number differs from the one used in \cite{BBJ18} by a factor of $2$).

Denote by $PSH^{NA}$ the space of $L$-psh functions on $X^{an}$. Recall that $u\in PSH^{NA}$ if and only if $u\not\equiv-\infty$ and $u$ is equal to the decreasing limit of a sequence of functions in $\cH^{NA}$.
 
Now let $\cM:=(\cX,\cL)$ be a model of $(X,L)$, so $\cL$ is not assumed to be relatively semi-ample. In this case, we can still define $D$ as in \eqref{eq:def-D}. Put
$$
f_D(w):=-\sigma(w)(D),\ w\in X^{an}.
$$
Then $f_D$ is a continuous function on $X^{an}$, but in general $f_D(w)\notin \cH^{NA}$. 
By \cite[Theorem 8.3]{BJ16} the non-Archimedean potential $u^\cM$ associated to the model is then given by the following envelope:
\begin{equation}
    \label{eq:def-NA-of-Model}
    u^{\cM}:=P(f_D):=\sup\{v\in PSH^{NA}: v\leq f_D\}.
\end{equation}
It follows from \cite[Theorem 8.5]{BJ16} (see also \cite[Proposition 4.5]{BoJ18}) that $u^\cM$ is continuous.

We can also describe $u^{\cM}$ using Lelong numbers as before. Let $\{u_t^\cM\}$ be the geodesic ray attached to $\cM$. After subtracting some multiple of $\cX_0$ from $\cL$ we can always assume that $u_t^\cM\leq 0$. Then for any $w\in X^{div}_\QQ$ with $\sigma(w)=\frac{1}{b_E}\ord_E$, we have
\begin{equation}
    \label{eq:nu-u-M-along-E}
    u^{\cM}(w)=-\frac{1}{b_E}\nu(u^\cM_{-\log|z|^2},E).
\end{equation}
This characterization of the singularity types of model rays will be used in \S \ref{sec:int-for-of-K-beta}.

\paragraph{Complex reductive Lie groups.}
In the proof of Theorem \ref{thm:YTD'''}, some classical Lie group theory will be needed, which we now briefly recall.
Let $\GG$ be a complex reductive Lie group with maximal compact subgroup $\KK$. Due to reductivity of $\GG$ there exists a compact semi-simple group $\Bbb L \subset \KK$ such  that $\Bbb L_{\Bbb C} \leq \GG$ and 
\begin{equation}\label{eq: G_decomp}
\GG \simeq \big(\Bbb L_{\Bbb C}  \times Z(\GG)\big)/ \Gamma_1,
\end{equation}
where $\Gamma_1$ is a discrete subgroup of $\Bbb L_{\Bbb C}  \times Z(G)$, and $Z(\GG)$ is the center of $\GG$ that can be further written as
\begin{equation}\label{eq: Z_decomp}
Z(\GG) \simeq (\CC^*)^m \times (\CC^k/\Lambda),
\end{equation}
where $\Lambda \leq \CC^k$ is a sublattice making $\CC^k/\Lambda \leq \KK$ compact \cite[Proposition, p. 181]{Bor}.
Using our notation, $Z(\GG) \cap \KK = (S^1)^m \times (\CC^k/\Lambda)$ is the maximal compact subgroup of $Z(\GG)$ and we have the following decomposition of  $\KK$:
\begin{equation}\label{eq: K_decomp}
\KK \simeq \big(\Bbb L \times (S^1)^m \times \CC^k/\Lambda\big)/\Gamma_2,
\end{equation}
where $\Gamma_2$ is a discrete subgroup of $\Bbb L \times (S^1)^m \times \CC^k/\Lambda$. 

Finally, we denote by $\TT = (\CC^*)^m \leq Z(\GG)$ the maximal torus of $Z(\GG)$, satisfying $\KK \cap \TT = (S^1)^m$. With the above decompositions in hand, let us recall the following well known result relating the normalizer $N_\GG(\KK)$ of the maximal compact group $\KK$ to the center of $\GG$ and $\TT$ \cite[Proposition B.1]{Li19-G-uniform}:

\begin{lemma}\label{lem: CLi_App} Let $\GG$ be a complex reductive Lie group with maximal compact subgroup $\KK$. With the above notation, $N_\GG(\KK) = Z(\GG) \KK = \TT \KK$.
\end{lemma}

\section{Transcendental quantization}
Let $(X,\omega)$ be a compact K\"ahler manifold of dimension $n$.
Recall that for any $\beta>0$ and $u\in\cE^1$, $e^{-\beta u}\in L^p$ for any $p>1$ \cite{Ze01}. Thus, by the theory of Monge--Amp\`ere equations \cite[Theorem C]{BBGZ}, there exists a unique continuous $\omega$-psh function $u^\beta$ satisfying
\begin{equation}
    \label{def:u-beta}
    (\omega+\ddc u^\beta)^n=e^{\beta(u^\beta-u)}\omega^n.
\end{equation}
Based on the speculations of \cite[Section 3.1]{Ber19}, we call $u^\beta$ the transcendental quantization of $u$ at level $\beta$ (see \cite[Appendix A.3]{DarvasSurvey} for a short survey on this topic).

There are some immediate properties of $u^\beta$, reminiscent of K\"ahler quantization:

\begin{lemma}\label{lem: u_beta_basic}
    One has \smallskip\\
\noindent(i) $(u+C)^\beta=u^\beta+C$ for any $u\in\cE^1$ and $C\in\RR$. \smallskip\\
\noindent(ii) $u\geq v$ implies that $u^\beta\geq v^\beta$. \smallskip\\
\noindent(iii) $Ent(u^\beta)<\infty$ for any $u\in\cE^1$.
\end{lemma}

\begin{proof}The first point is trivial and the second follows from the comparison principle \cite[Lemma 2.5]{Ber19} (whose proof applies to potentials of $\mathcal E^1$).
    Finally, we argue the last point:
    \begin{equation*}
        \begin{split}
             Ent(u^\beta)&=\frac{1}{V}\int_X\log\frac{\omega^n_{u^\beta}}{\omega^n}\omega^n_{u^\beta}=\frac{\beta}{V}\int_X(u^\beta-u)\omega^n_{u^\beta}\leq \frac{\beta}{V}\int_X|u^\beta-u|\omega^n_{u^\beta}\leq \beta C_n d_1(u^\beta,u)<\infty.
        \end{split}
    \end{equation*}
\end{proof}

The preceding observation indicates that \( u^\beta \) serves as a regularization of \( u \), particularly in cases where \( u \in \mathcal{E}^1_\omega \) and \( \operatorname{Ent}(u) = \infty \). We now proceed to show that $u^\beta \xrightarrow{d_1} u$.

\begin{lemma}
\label{lem:u-beta-d1-u}
    For any $u\in\cE^1$, one has $u^\beta\xrightarrow{d_1}u$ as $\beta\to\infty$.
\end{lemma}

\begin{proof}

We first show that $u^\beta\xrightarrow{a.e.}u$ as $\beta\to\infty$. Let $u\in\cE^1$ and $\delta>1$.
    Consider $u_\delta:=\delta^{-1}P(\delta u) \in \mathcal E^1$ as in \cite[Proposition 2.15]{DLR19} (Proposition 2.18 in the published version). Then
    $
    u_\delta\leq u\text{ and }\omega+\ddc u_\delta\geq (\frac{\delta-1}{\delta})\omega.
    $
Thus:
    $$
    (\omega+\ddc u_\delta)^n\geq e^{\beta((u_\delta + \frac{n}{\beta}\log(\frac{\delta-1}{\delta}))-u_\delta)}\omega
    ^n \ \ \ \textup{ and } \ \ \ (\omega+\ddc u^\beta)^n\leq e^{\beta(u^\beta-u_\delta)}\omega
    ^n,
    $$
   Applying \cite[Lemma 2.5]{Ber19} again, we obtain that
\begin{equation}\label{eq: u_beta_est}
    u^\beta\geq u_\delta+ \frac{n}{\beta}\log\Big(\frac{\delta-1}{\delta}\Big).
\end{equation}
    In particular,
    $
    \liminf_{\beta\to\infty} u^\beta(x)\geq u_\delta(x),\ \text{for all }\delta>1, x\in X.
    $
    Since $u_\delta\xrightarrow{a.e.}u$ as $\delta\searrow 1$ (\cite[Proposition 2.15]{DLR19}), we obtain that
$$
    \liminf_{\beta\to\infty} u^\beta(x)\geq u(x)\text{ for almost all }x\in X. 
    $$
    
    We now bound $u^\beta$ from above. Pick a smooth sequence $u_k\searrow u$ pointwise. 
    Using the fact that $u_k^\beta\geq u^\beta$ and that $u_k^\beta\xrightarrow{C^0} u_k$ as $\beta\to\infty$ (\cite[Theorem 1.1]{CZ19a}, \cite[Theorem A.10]{DarvasSurvey}),
    $$
    \limsup_{\beta\to\infty} u^\beta(x)\leq \lim_{\beta\to\infty} u_k^\beta(x)=u_k(x)\text{ for any }k \text{ and }x\in X.
    $$
Thus, $
    \limsup_{\beta\to\infty} u^\beta(x)\leq u(x), x\in X$. 
So $u^\beta\xrightarrow{a.e.}u$ as $\beta\to\infty$, as claimed.

Let $\delta>1$. We estimate $d_1(u^\beta,u)$ as follows:
    \begin{equation*}
        \begin{split}
            d_1(u,u^\beta)&\leq d_1(u_\delta+\frac{n}{\beta}\log\Big(\frac{\delta-1}{\delta}\Big),u^\beta)+d_1(u,u_\delta)-\frac{n}{\beta}\log\Big(\frac{\delta-1}{\delta}\Big)\\
            &=I(u^\beta)-I(u_\delta+\frac{n}{\beta}\log\Big(\frac{\delta-1}{\delta}\Big))+d_1(u,u_\delta)-\frac{n}{\beta}\log\Big(\frac{\delta-1}{\delta}\Big)\\
            &\leq \frac{1}{V}\int_X|u^\beta-u_\delta|\omega^n_{u_\delta}+d_1(u,u_\delta)-\frac{2n}{\beta}\log\Big(\frac{\delta-1}{\delta}\Big).
        \end{split}
    \end{equation*}
    Sending $\beta\to\infty$, using dominated convergence and \cite[Theorem 3]{Dar15},
    $$
    \limsup_{\beta\to\infty} d_1(u,u^\beta)\leq \frac{1}{V}\int_X|u-u_\delta|\omega^n_{u_\delta}+d_1(u,u_\delta)\leq C_n d_1(u,u_\delta).
    $$
    Finally, sending $\delta\searrow 1$ and using \cite[Proposition 2.15]{DLR19}, we conclude.
\end{proof}

Next we show convergence of the entropy functional under transcendental quantization:

\begin{lemma}\label{lem: entr_conv}
    For any $u\in\cE^1$, $Ent(u^\beta)\to Ent(u)$ as $\beta\to\infty$.
\end{lemma}

\begin{proof}
    If $Ent(u)=\infty$, then trivially $Ent(u)\geq Ent(u^\beta)$.    If $Ent(u)<\infty$, we can use Jensen's inequality to derive that
    \begin{equation*}
        \begin{split}
            0&=\log\int_X\frac{\omega^n_{u^\beta}}{V}=\log\int_Xe^{\beta(u^\beta-u)}\frac{\omega^n}{\omega^n_u}\frac{\omega^n_u}{V}\geq\frac{1}{V}\int_X\beta(u^\beta-u)\omega^n_u-Ent(u)\\
            & \geq\beta(I(u^\beta)-I(u))-Ent(u)\geq\frac{1}{V}\int_X \beta(u^\beta-u)\omega^n_{u^\beta}-Ent(u)=Ent(u^\beta)-Ent(u).\\
        \end{split}
    \end{equation*}
    Thus $Ent(u)\geq Ent(u^\beta)$ for any $u\in\cE^1$. So we conclude since $Ent(\cdot)$ is $d_1$-lsc (\cite{BBEGZ19}, cf. \cite[Corollary 4.40]{DarvasSurvey}).
\end{proof}

The following simple but important estimate can be extracted from the above proof.

\begin{corollary}
\label{cor:Ent>beta>Ent}
    For any $u\in\cE^1$,
    $
    Ent(u)\geq\beta(I(u^\beta)-I(u))\geq Ent(u^\beta).
    $
\end{corollary}

The next result shows that our quantization procedure improves convergence:

\begin{lemma}
\label{lem:u-beta-d1-continuous}
    If $u_i\xrightarrow{L^1}u$ in $\cE^1$, then $||u^\beta_i-u^\beta||_{C^0}\to 0$ for any $\beta>0$. In particular, $u_i^\beta\xrightarrow{d_1}u^\beta$.
\end{lemma}

\begin{proof}
    Since for any $p>0$, $e^{-p u_i}$ and $e^{-p u}$ are integrable by \cite{Ze01}, \cite[Theorem 0.2]{DK01} implies that
    $
    \int_X|e^{-p u_i}-e^{-p u}|\omega^n\to 0.
    $
    In particular, there exists $A_p>0$ such that
    $$
    ||e^{-p u_i}||_{L^1}+||e^{-p u}||_{L^1}\leq A_p.
    $$

    On the other hand, by the stability estimate \cite[Remark 3.2]{GLZ18}, for any $\beta>0$,
    $$
    ||u^\beta_i-u^\beta||_{C^0}\leq C||e^{-\beta u_i}-e^{-\beta u}||_{L^1}^{\frac{1}{n+1}},
    $$
    where $C$ depends on $X,\omega,p, A_p$ for some $p>\beta$. We then conclude since $e^{-\beta u_i}\xrightarrow{L^1}e^{-\beta u}.$  
\end{proof}

\paragraph{The $\beta$-entropy and quantization.}

For $u \in \mathcal E^1$ and $\beta>0$, the $\beta$-entropy $Ent^\beta(u)$ is defined by extending \eqref{eq: Ent_beta_def}:
\begin{equation}\label{eq: Ent_beta_def-E1}
Ent^\beta(u) := \sup_{v\in\cE^1}\left(-\log\int_Xe^{\beta(v-u)}\frac{\omega^n}{V}+\beta(I(v)-I(u))\right).
\end{equation}
We link this functional to our quantization scheme:

\begin{proposition}
\label{prop:ent-beta=sup}
    For any $u\in\cE^1$ and $\beta>0$,
    $
Ent^\beta(u)=\beta(I(u^\beta)-I(u)).$
\end{proposition}

\begin{proof}
   That $Ent^\beta(u) \geq \beta(I(u^\beta)-I(u))$ follows from \eqref{eq: Ent_beta_def-E1} and \eqref{def:u-beta}. Conversely, for  $v \in \mathcal E^1_\omega$ Jensen's inequality implies that  
    \begin{equation*}
        \begin{split}
            -\log\int_Xe^{\beta(v-u)}\frac{\omega^n}{V}+\beta(I(v)-I(u))&=-\log\int_Xe^{\beta(v-u^\beta)}\frac{\omega^n_{u^\beta}}{V}+\beta(I(v)-I(u))\\
            &\leq \frac{\beta}{V}\int_X(u^\beta-v)\omega^n_{u^\beta}+\beta(I(v)-I(u))\\
            &\leq \beta(I(u^\beta)-I(v))+\beta(I(v)-I(u))=\beta(I(u^\beta)-I(u)).\\
        \end{split}
    \end{equation*}
    This completes the proof.
\end{proof}

The above argument shows that $u^\beta$ attains the supremum in \eqref{eq: Ent_beta_def-E1}. Moreover, taking $\cH_\omega\ni u_i\searrow u$ and applying Lemma \ref{lem:u-beta-d1-continuous}, then $u_i^\beta\in\cH_\omega$ is a smooth maximizing sequence for \eqref{eq: Ent_beta_def-E1}, showing that \eqref{eq: Ent_beta_def} holds for $u\in\cE^1$ as well.

We also observe that $\beta\mapsto Ent^\beta(\cdot)$ satisfies certain concavity. 

\begin{corollary}
\label{cor:Ent-beta-concave-in-beta}
    For any $0<\beta<\beta'$ and $u\in\cE^1$, one has
    $
    \frac{Ent^\beta(u)}{\beta}\geq\frac{Ent^{\beta'}(u)}{\beta'}.
    $
\end{corollary}

\begin{proof}
    Using \eqref{eq: Ent_beta_def-E1}, we can write
    \begin{equation*}
        \begin{split}
            Ent^{\beta}(u)&\geq-\log\int_Xe^{\beta(u^{\beta'}-u)}\frac{\omega^n}{V}+\beta(I(u^{\beta'})-I(u))\geq 0+\frac{\beta}{\beta'}Ent^{\beta'}(u),
        \end{split}
    \end{equation*}
    where in the last line we used H\"older's inequality and that $\int_Xe^{\beta'(u^{\beta'}-u)}\omega^n/V=1$.
    This completes the proof.
\end{proof}

As a direct consequence of the identity $Ent^\beta(u)=\beta(I(u^\beta)-I(u))$, we now derive several properties of the associated \( \beta \)-entropy.

\begin{lemma}\label{lem: Ent_beta_first}
    One has \smallskip \\
\noindent (i) $Ent(u)\geq Ent^\beta(u)\geq Ent(u^\beta)\geq 0$ for any $u\in\cE^1$ and $\beta>0$. \smallskip\\
\noindent (ii)  $Ent^\beta(u)\to Ent(u)$ as $\beta\to\infty$ for any $u\in\cE^1$. \smallskip\\
\noindent (iii) $\frac{1}{\beta}Ent^\beta(u)\to 0$ as $\beta\to\infty$ for any $u\in\cE^1$.  
\end{lemma}

\begin{proof}
    These follow from Lemma \ref{lem: entr_conv}, Corollary \ref{cor:Ent>beta>Ent} and  Lemma \ref{lem:u-beta-d1-continuous}.
\end{proof}

Compared to the usual entropy, our $Ent^\beta(\cdot)$ has better continuity properties on $\cE^1$:

\begin{lemma} \label{lem: Ent_beta_cont}Let $\beta>0.$ The following holds.\smallskip\\
\noindent (i) $Ent^\beta(\cdot)$ is $L^1$-lsc. Namely, if  $u_i\xrightarrow{L^1}u$ in $\cE^1$ then
    $    Ent^\beta(u)\leq\liminf_{i\to\infty} Ent^\beta(u_i).
    $\smallskip\\
\noindent (ii) $Ent^\beta(\cdot)$ is $d_1$-continuous. Namely, if  $u_i\xrightarrow{d_1}u$ then
    $
    Ent^\beta(u)=\lim_{i\to\infty} Ent^\beta(u_i).
    $
\end{lemma}

\begin{proof}
    The first point follow from Lemma \ref{lem:u-beta-d1-continuous} and that $I(\cdot)$ is $L^1$-usc \cite{BBEGZ19} (see also \cite[Corollary 4.14]{DarvasSurvey}). Then second follows from Lemma \ref{lem:u-beta-d1-continuous} as well.
\end{proof}

Another interesting feature of the quantized entropy $Ent^\beta(\cdot)$ is that it is monotone with respect to the parameter $\beta$.

\begin{proposition}
\label{prop:Ent-beta-increase}
    Let $0<\beta_1\leq \beta_2$. Then for any $u\in\cE^1$,
    $
    Ent^{\beta_1}(u)\leq Ent^{\beta_2}(u).
    $
\end{proposition}

\begin{proof}
    By approximation, it suffices to prove the monotonicity for $u\in\cH_\omega$. In this case, $\{u^\beta\}_{\beta\in(0,\infty)}$ forms a smooth family of K\"ahler potentials. Therefore it suffices to show that
    $$
    \frac{d}{d\beta}Ent^\beta(u)=I(u^\beta)-I(u)+\frac{\beta}{V}\int_X\frac{d u^\beta}{d\beta}\omega^n_{u^\beta}\geq 0.
    $$

    To this end, we differentiate \eqref{def:u-beta} with respect to $\beta$ to get
    $
    \Delta_{\omega_{u^\beta}}\frac{d u^\beta}{d\beta}=(u^\beta-u)+\beta\frac{d u^\beta}{d\beta}.
    $
    Integrating against $\omega^n_{u^\beta}$, we derive that
$
    \frac{\beta}{V}\int_X\frac{d u^\beta}{d\beta}\omega^n_{u^\beta}=-\frac{1}{V}\int_X(u^\beta-u)\omega^n_{u^\beta}.
$
    Thus,
    $$
    I(u^\beta)-I(u)+\frac{\beta}{V}\int_X\frac{d u^\beta}{d\beta}\omega^n_{u^\beta}= I(u^\beta)-I(u)-\frac{1}{V}\int_X(u^\beta-u)\omega^n_{u^\beta}\geq 0,
    $$
    finishing the proof.
\end{proof}

As a consequence of the above result, Lemma \ref{lem: entr_conv} and Corollary \ref{cor:Ent>beta>Ent} we note the following:

\begin{corollary}\label{cor: increasing_K_beta_conv}
    For $u\in\cE^1$ we have $Ent^\beta(u)\nearrow Ent(u)$ as $\beta\nearrow\infty$.
\end{corollary}

We also record a simple estimate, which will be used later.

\begin{lemma}
\label{lem:d1-u-beta<d1-u}
    Let $u \in \mathcal E^1$ with $u \leq 0$. Then $d_1(0,u^\beta)\leq d_1(0,u)$ for any $\beta>0$.
\end{lemma}

\begin{proof}
By Lemma \ref{lem: u_beta_basic}(ii) we have $u^\beta\leq 0$.
    Thus by Proposition \ref{prop:ent-beta=sup},    \begin{equation*}
        \begin{split}
            d_1(u^\beta,0)&=-I(u^\beta)=-I(u)-\frac{1}{\beta}Ent^\beta(u)=d_1(0,u)-\frac{1}{\beta}Ent^\beta(u)\leq d_1(0,u),\\
        \end{split}
    \end{equation*}
    where we used $Ent^\beta(u)\geq 0$ (see Lemma \ref{lem: Ent_beta_first}(i)). 
\end{proof}

\paragraph{P\u{a}un's positivity result.} In the classical setting of algebraic quantization, a useful positivity result due to Berndtsson roughly says that the quantization of a subgeodesic is still a subgeodesic \cite[Proposition 3.1]{Bern09}, a crucial point in the second author's alternative proof of the YTD conjecture for K\"ahler-Einstein type metrics \cite{Zhang21YTD}. Later in \S \ref{sec:slope-formula}, a transcendental replacement for Berndtsson's result will be needed. It is given by the following positivity result that can be extracted from P\u{a}un's paper \cite[\S 3]{P17} (cf. also \cite[Theorem 4.1]{CLP}).

\begin{theorem}
\label{thm:paun}
    Let $\Sigma:=\{a<|z|<b\}\subset\CC$ be an annulus. Let $X$ be a compact K\"ahler manifold of dimension $n$ and $\xi$ be a K\"ahler class on $X$. Let $\Omega$ be a smooth positive volume form on $X$. We denote by $\Ric(\Omega):=-\ddc\log\det\Omega$ the Ricci form associated with $\Omega$. Denote by $\pi$ the projection map $X\times\Sigma\to X$. Assume that there exists a smooth semipositive form $\cA$ on $X\times\Sigma$ in the $(1,1)$-class $\pi^*\xi$ such that $\cA|_{X\times\{z\}}-Ric(\Omega)$ is a K\"ahler form on $X\times\{z\}$ for all $z\in\Sigma$. For any $z\in\Sigma$, consider the unique smooth solution $\Phi_z$ of the Aubin--Yau type equation on $X\times\{z\}$:
    $$
    (\cA|_{X\times\{z\}}-\Ric(\Omega)+\mathrm{dd}_X^c \Phi_z)^n=e^{\Phi_z}\Omega.
    $$
    Then $\cA-\pi^*\Ric(\Omega)+\mathrm{dd}^c_{X\times\Sigma} \Phi_z$
    is a smooth semipositive $(1,1)$-form on $X\times\Sigma.$
\end{theorem}

Actually in \cite{P17} P\u{a}un proved a much more general positivity result for a family of K\"ahler manifolds. But for us the above simple product case is just enough.

Let us explain how to apply the above result to our setting. First assume that $\{u_t\}_{t\in (0,1)}$ is a smooth subgeodesic segment in $\cE^1$, i.e.,  under the coordinate change $t:=-\log|z|^2$,
$$
\pi^*\omega+\mathrm{dd}^c_{X\times \Sigma} u_{-\log|z|^2}\geq 0 \ \textup{ on } \ X\times\Sigma, 
$$
where $\Sigma:=\{e^{-1}<|z|^2<1\}\subset\CC$.
Assume further that there exists $\eta>0$ such that $$\pi^*\omega+\ddc_{X\times\Sigma}{u_{-\log|z|^2}}\geq\eta\pi^*\omega.$$
Then for large enough $\beta>0$,
$
\beta\omega_{u_t}+\Ric(\omega)>0.
$
So putting
$$
\cA:=\pi^*\beta\omega+\mathrm{dd}^c_{X\times \Sigma} \beta u_{-\log|z|^2}+\pi^*\Ric(\omega),
$$
we obtain a smooth semipositive $(1,1)$-form on $X\times\Sigma$ whose restriction on each fiber $X\times\{z\}$ is the K\"ahler form $\beta\omega_{u_t}+\Ric(\omega)$. Let
$\Phi_z:=\beta u_{-\log|z|^2}^\beta-\beta u_{-\log|z|^2}.$
Then the Monge--Amp\`ere equation of $u_t^\beta$ \eqref{def:u-beta} translates to 
$$
(\cA|_{X\times\{z\}}-\Ric(\omega)+\mathrm{dd}_X^c\Phi_z)^n=e^{\Phi_z}\beta^n\omega^n
$$
for any $z\in\Sigma$.
Denoting $\Omega:=\beta^n\omega^n$, we have  $\Ric(\Omega)=\Ric(\omega)$. So Theorem \ref{thm:paun} applies, giving us that
$$
\pi^*\omega+\mathrm{dd}^c_{X\times\Sigma}u^\beta_{-\log|z|^2}\geq 0,
$$
i.e., $\{u_t^\beta\}_{t\in(0,1)}$ is a subgeodesic segment.

By approximation, the above consideration extends to a general subgeodesic segment in $\cE^1$, as shown in the next result.

\begin{proposition}
\label{prop:u-t-beta-is-subgeodesic}
    Assume that $(0,1)\ni t\mapsto u_t\in\cE^1$ is subgeodesic segment. Suppose that $\Ric \omega \geq -A \omega$ for some $A \geq 0$. Then for any $\eta \in (2A/\beta,1)$ and $\beta > 2A$ the curve  $(0,1)\ni t\mapsto ((1-\eta)u_t)^\beta\in\cE^1_{\omega}$ is subgeodesic segment as well.
\end{proposition}

\begin{proof}
Assume without loss of generality that $u_t\leq 0$. For any $\varepsilon\in(0,\eta)$ put
    $$u_{t,\varepsilon}:=(1-\eta)u_t+\frac{\varepsilon}{e^t}.$$
    Consider the projection $X\times\{e^{-1}<|z|^2<1\}\xrightarrow{\pi}X$ and $t=-\log|z|^2$. Then
    $$
    \pi^*\omega+\ddc u_{-\log|z|^2,\varepsilon}\geq \eta\pi^*\omega+\varepsilon idz\wedge d\bar z>0.
    $$
    Thus $\{u_{t,\varepsilon}\}$ is subgeodesic and also it is clear that $u_{t,\varepsilon}\searrow (1-\eta)u_t$ pointwise as $\varepsilon\searrow 0$. Therefore, it suffices to show that $\{u_{t,\varepsilon}^\beta\}$ is subgeodesic (note from Lemma \ref{lem:u-beta-d1-continuous} that  $u_i\xrightarrow{L^1}u$ implies $u_i^\beta\xrightarrow{C^0}u^\beta$).

    To this end, we need to further regularize $u_{t,\varepsilon}$. For any $0<\delta< \varepsilon$, by \cite[Theorem 2]{BK07} we can find smooth approximating functions $U_{\varepsilon,\delta}$ on $X\times\{e^{-1}+\delta<|z|^2<1-\delta\}$ such that
    $$
    \pi^*\omega+\ddc U_{\varepsilon,\delta}\geq \frac{2A}{\beta}\pi^*\omega+(\varepsilon-\delta)idz\wedge d\bar z>0.
    $$
    Moreover  (by averaging), we can arrange that $U_{\varepsilon,\delta}$ is $S^1$-invariant in $z$, yielding a smooth subgeodesic $\{u_{t,\varepsilon,\delta}\}\subset\cH_\omega$ with $u_{t,\varepsilon,\delta}\searrow u_{t,\varepsilon}$ pointwise as $\delta\searrow 0$.
    
It follows that
    $$
    \beta \pi^*\omega+\ddc \beta U_{\varepsilon,\delta} + \Ric \omega \geq A \pi^*\omega+\beta(\varepsilon-\delta)idz\wedge d\bar z>0.
    $$
    Therefore, it suffices to argue that $\{u^\beta_{t,\varepsilon,\delta}\}$ is subgeodesic for all $\beta > 2A$, which now follows  Theorem \ref{thm:paun}, and the discussion following it.
\end{proof}

\begin{remark}
If $\Ric(\omega)\geq 0$, then we can take $\eta =0$ (using a limit process $\eta \searrow 0$) in the above result. So in this case $ t\mapsto u^\beta_t$ is subgeodesic for any $\beta>0$ whenever $t \mapsto u_t$ is.
\end{remark}

\section{Properness of the $K^\beta$ energy}

In this section we turn to the analysis of the K energy and $K^\beta$ energy. Recall from \eqref{Kendef} that
$$
K(u):=Ent(u)-\cJ_{\Ric(\omega)}(u),\ u\in\cE^1\ ,
$$
where for some smooth closed $(1,1)$-form $\chi$ we define
$$
\cJ_\chi(u):=nI_\chi(u)-\bar\chi I(u)=\frac{1}{V}\int_X u\chi\wedge\sum_{i=0}^{n-1}\omega^i\wedge\omega^{n-1-i}_u-\bar \chi I(u),\ \bar \chi=\frac{n}{V}\int_X\chi\wedge\omega^{n-1}.
$$
For $\beta>0$, recall from \eqref{Kbetadef} the definition of our $K^\beta$ energy:
    $$
    K^\beta(u):=Ent^\beta(u)-\cJ_{\Ric(\omega)}(u).
    $$
As a consequence of Proposition \ref{prop:Ent-beta-increase} we note the following fact for $u\in\cE^1$:
\begin{equation}\label{eq: K_beta_conv}
K^\beta(u)\nearrow K(u) \ \textup{ as } \ \beta\nearrow\infty
\end{equation}

In our first technical result we show that the distance $d_1(u,u^\beta)$ is comparable to the growth of $\frac{1}{\beta}Ent^\beta(u)$. The role of this result in our work is analogous with the role of the partial $C^0$ estimate from the Fano case, first proposed by Tian (cf. \cite[(0.2)]{Tia90}).

\begin{proposition}
    \label{prop:PC0}
    There exists $C_n>0$ such that for any $u\in\cE^1$ with $u\leq 0$ and $\beta>1,$ 
    $$
    \frac{1}{\beta}Ent^\beta(u)\leq d_1(u^\beta,u)\leq \frac{1}{\beta}\left(Ent^\beta(u)+C_n\log ( d_1(0,u)+2)+C_n\log\beta\right).
    $$
\end{proposition}

\begin{proof}
    For any $\delta \in (1,2)$, let $u_\delta:=\delta^{-1}P(\delta u)$. Then, using \eqref{eq: u_beta_est} and $u\leq 0$
    $$
    u^\beta\geq u_\delta+\frac{n}{\beta}\log\Big(\frac{\delta-1}{\delta}\Big)\geq P(\delta u)+\frac{n}{\beta}\log\Big(\frac{\delta-1}{\delta}\Big).
    $$
    Thus,
    \begin{equation*}
    \begin{split}
        d_1(u^\beta,u)&\leq d_1(u^\beta,P(\delta u)+\frac{n}{\beta}\log\Big(\frac{\delta-1}{\delta}\Big))+d_1(P(\delta u),u)-\frac{n}{\beta}\log\Big(\frac{\delta-1}{\delta}\Big)\\
        &=I(u^\beta)-I(P(\delta u))+d_1(P(\delta u),u)-\frac{2n}{\beta}\log\frac{\delta-1}{\delta}\\
        &=\frac{1}{\beta}Ent^\beta(u)+I(u)-I(P(\delta u))+d_1(P(\delta u),u)-\frac{2n}{\beta}\log\frac{\delta-1}{\delta}\\
        &\leq\frac{1}{\beta}Ent^\beta(u)+2d_1(P(\delta u),u)-\frac{2n}{\beta}\log\frac{\delta-1}{\delta}.\\
    \end{split}
    \end{equation*}
    In the last line we used $I(u)-I(v)\leq d_1(u,v)$.
    
We control $d_1(P(\delta u),u)$ as follows (using \cite[Theorem 3]{Dar15} and \cite[Lemma 4.4]{DDL5}):
 \begin{equation}
 \label{eq:d1-u-P-delta-u}
     \begin{split}
         d_1(P(\delta u),u)&\leq \frac{C_n}{V}\int_X|u-P(\delta u)|\omega^n_{P(\delta u)}\leq \frac{C_n \delta^n(\delta-1)}{V}\int_X|u|\omega^n_u\leq C_n'\delta^n(\delta-1)d_1(0,u).
     \end{split}
 \end{equation}
 So we arrive at
 $$
 d_1(u^\beta,u)\leq\frac{1}{\beta}Ent^\beta(u)+2C_n'\delta^n(\delta-1)d_1(0,u)-\frac{2n}{\beta}\log\frac{\delta-1}{\delta}.
 $$

 Next, we choose $\delta$ such that
 $
(\delta-1)(d_1(0,u)+1)=\frac{1}{\beta}.
 $
Since $\beta>1$, such $\delta$ satisfies $1<\delta<2$, and we have
 $$
 d_1(u^\beta,u)\leq \frac{1}{\beta}Ent^\beta(u)+\frac{2^{n+1}C_n'}{\beta}+\frac{2n}{\beta}\log(2\beta d_1(0,u)+2\beta).
 $$
This establishes the upper bound for $d_1(u^\beta,u)$.

Finally, we also have $ d_1(u^\beta,u)\geq I(u^\beta)-I(u)=\frac{1}{\beta}Ent^\beta(u)$,  finishing the proof.
\end{proof}

\begin{proposition}\label{prop: J_ric_perturb} Let $\chi$ be a closed smooth $(1,1)$-form. 
    There exists a constant $C>0$, depending only on $X,\omega,\chi, n$ such that
    $$
    |\cJ_{\chi}(u^\beta)-\cJ_{\chi}(u)|\leq\frac{C}{\beta^{1/2^n}}(Ent^\beta(u)+d_1(0,u)+1+\log\beta)
    $$
    for any $\beta>1$ and $u\in\cE^1$ with $u\leq 0$.
\end{proposition}

\begin{proof}
In what follows, the constant $C>0$ only depends on $X,\omega,\chi,n$, but may change from line to line.

Pick $\delta\in(1,2)$. Then, as before, $u\geq P(\delta u)+\frac{n}{\beta}\log\frac{\delta-1}{\delta}$ and $ u^\beta\geq P(\delta u)+\frac{n}{\beta}\log\frac{\delta-1}{\delta}.$ 

Let $C>0$ such that $-C\omega\leq\chi\leq C\omega$. We then have
\begin{equation*}
    \begin{split}
        |\cJ_{\chi}&(u^\beta)-\cJ_{\chi}(u)|\leq C\left(\Big|\int_X(u^\beta-u)\sum_{i=0}^{n-1}\chi\wedge\omega^i_u\wedge\omega^{n-1-i}_{u^\beta}\Big|+|I(u^\beta)-I(u)|\right)\\
        &\leq C \int_X|u^\beta-u|\sum_{i=0}^{n-1}\omega\wedge\omega^i_u\wedge\omega^{n-1-i}_{u^\beta}+Cd_1(u^\beta,u)\\
        &\leq C\int_X(u^\beta-(P(\delta u)+\frac{n}{\beta}\log\frac{\delta-1}{\delta}))\sum_{i=0}^{n-1}\omega\wedge\omega^i_u\wedge\omega^{n-1-i}_{u^\beta}\\
        &+C\int_X(u-(P(\delta u)+\frac{n}{\beta}\log\frac{\delta-1}{\delta}))\sum_{i=0}^{n-1}\omega\wedge\omega^i_u\wedge\omega^{n-1-i}_{u^\beta}+Cd_1(u^\beta,u).\\
        &=C\int_X(u^\beta-u)\sum_{i=0}^{n-1}\omega\wedge\omega^i_u\wedge\omega^{n-1-i}_{u^\beta}+2C\int_X(u-P(\delta u))\sum_{i=0}^{n-1}\omega\wedge\omega^i_u\wedge\omega^{n-1-i}_{u^\beta}\\
        &\ \ \ \ \ \ \ \ \ \ \ +Cd_1(u^\beta,u)-\frac{C}{\beta}\log\frac{\delta-1}{\delta}.\\
    \end{split}
\end{equation*}

We bound the terms on the last line one by one.
By \cite[Lemma A.2]{BBJ18}, we have
$$
\int_X(u^\beta-u)\omega\wedge\omega^{n-i-1}_u\wedge\omega_{u^\beta}^{i}\leq C d_1(u^\beta,u)^{1/2^n}M^{1-1/2^n},
$$
where
$
M:=\max \{\cI(u),\cI(u^\beta)\},
$
and $\cI(u):=\frac{1}{V}\int_Xu(\omega^n-\omega^n_u)$ denotes Aubin's $\cI$-functional. We assumed that $u\leq 0$, so $u^\beta\leq 0$ as well. Due to \cite[Theorem 3]{Dar15},
$$
M\leq (n+1)\max\{J(u),J(u^\beta)\}\leq (n+1)\max\{d_1(0,u),d_1(0,u^\beta)\}=(n+1)d_1(0,u),
$$
where we used Lemma \ref{lem:d1-u-beta<d1-u}(i) in the last equality. Therefore,
\begin{equation}\label{eq: u_u_beta_est}
\int_X(u^\beta-u)\omega\wedge\omega^{n-i-1}_u\wedge\omega_{u^\beta}^{i}\leq C d_1(u^\beta,u)^{1/2^n}d_1(0,u)^{1-1/2^n}.
\end{equation}

Similarly, we have
$$
\int_X(u-P(\delta u))\omega\wedge\omega^{n-i-1}_u\wedge\omega_{u^\beta}^{i}\leq C d_1(u,P(\delta u))^{1/2^n}(\max\{d_1(0,u),d_1(0,P(\delta u))\})^{1-1/2^n}.
$$
Using $d_1(u,P(\delta u))\leq C(\delta-1)d_1(0,u)$ (recall \eqref{eq:d1-u-P-delta-u}), we obtain
$$
d_1(0,P(\delta u))\leq d_1(0,u)+d_1(u,P(\delta u))\leq Cd_1(0,u),
$$
which further implies that
$$
\int_X(u-P(\delta u))\omega\wedge\omega^{n-i-1}_u\wedge\omega_{u^\beta}^{i}\leq C(\delta-1)^{1/2^n}d_1(0,u).
$$

Putting these estimates together, we arrive at
\begin{equation*}
    \begin{split}
        |&\cJ_{\chi}(u^\beta)-\cJ_{\chi}(u)|\leq\\ &C\Big(d_1(u^\beta,u)^{1/2^n}d_1(0,u)^{1-1/2^n}+(\delta-1)^{1/2^n}d_1(0,u)+d_1(u^\beta,u)-\frac{1}{\beta}\log\frac{\delta-1}{\delta}\Big).\\
    \end{split}
\end{equation*}

Now we choose $\delta\in (1,2)$ such that
$
(\delta-1)^{1/2^n}=\frac{1}{\beta},
$
which implies that
\begin{equation*}
    \begin{split}
        |&\cJ_{\chi}(u^\beta)-\cJ_{\chi}(u)|\leq\\ &C\left(d_1(u^\beta,u)^{1/2^n}d_1(0,u)^{1-1/2^n}+d_1(u^\beta,u)+\frac{1}{\beta}d_1(0,u)+\frac{1}{\beta}\log(2\beta^{2^n})\right)\\
        &\leq C\left(d_1(u^\beta,u)^{1/2^n}d_1(0,u)^{1-1/2^n}+d_1(u^\beta,u)+\frac{1}{\beta}d_1(0,u)+\frac{2^n}{\beta}\log(2\beta)\right).\\
    \end{split}
\end{equation*}

To control the terms involving $d_1(u^\beta,u)$, we apply Proposition \ref{prop:PC0} to get
\begin{flalign}\label{eq: PC0_cons}
d_1(u^\beta,u)^{1/2^n}d_1(0,u)^{1-1/2^n}&\leq
        \frac{C}{\beta^{1/2^n}}\left(Ent^\beta(u)+\log ( d_1(0,u)+2)+\log\beta\right)^{1/2^n}d_1(0,u)^{1-1/2^n} \nonumber\\
        &\leq \frac{C}{\beta^{1/2^n}}\left(Ent^\beta(u)+ d_1(0,u)+1+\log\beta\right)^{1/2^n}d_1(0,u)^{1-1/2^n}\\
        &\leq \frac{C}{\beta^{1/2^n}}\left(Ent^\beta(u)+d_1(0,u)+1+\log\beta\right). \nonumber
\end{flalign}
Lastly, since $\beta >1$, Proposition \ref{prop:PC0} again implies that
\begin{equation*}
    \begin{split}
    d_1(u^\beta,u)&\leq \frac{C}{\beta}(Ent^\beta(u)+\log ( d_1(0,u)+2)+\log\beta)\leq\frac{C}{\beta^{1/2^n}}(Ent^\beta(u)+ d_1(0,u)+1+\log\beta).
    \end{split}
\end{equation*}
Putting everything together, our result follows.
\end{proof}

\begin{lemma}\label{lem: sup_perturb_lemma}
There exists a constant $C>0$, depending only on $X,\omega,n$ such that
    for any $u\in\cE^1$ with $u\leq 0$,
    $$
    \left|\sup_X u^\beta-\sup_X u\right|\leq C+\frac{C}{\beta^{1/2^n}}(Ent^\beta(u)+d_1(0,u)+1+\log\beta).
    $$
\end{lemma}

\begin{proof}
    First, it is well known that (see e.g. \cite[Lemma 3.45]{DarvasSurvey})
    $
    |\sup_X u-\frac{1}{V}\int_Xu\omega^n|\leq C.$
    Therefore, it suffices to argue that
\begin{equation}\label{eq: u_beta_sup}
    \Big|\int_X(u^\beta-u)\omega^n\Big|\leq \frac{C}{\beta^{1/2^n}}(Ent^\beta(u)+ d_1(0,u)+1+\log\beta).
\end{equation}
Using \cite[Lemma A.2]{BBJ18} exactly as in the proof of the previous argument we obtain that 
$
\Big|\int_X(u^\beta-u)\omega^n\Big|\leq C d_1(u^\beta,u)^{1/2^n}d_1(0,u)^{1-1/2^n}.
$
By \eqref{eq: PC0_cons} we conclude \eqref{eq: u_beta_sup}.
\end{proof}

Recall that an energy functional $G$ defined on $\mathcal E^1$ is proper (or coercive) if there exists $\gamma,\gamma' >0$ such that $$G(u)  \geq \gamma J(u) - \gamma', \ \ u \in \mathcal E^1.$$
Since $J(u+c) = J(u), \ c \in \Bbb R$, this is equivalent to  
$$G(u)  \geq \gamma d_1(0,u) - \gamma', \ \ u \in \mathcal E^1 \textup{ with }\sup_X u =0.$$ 
Indeed if $\sup_X u =0$ then 
\begin{equation}\label{eq: J_d_1_eqv}
J(u) \leq d_1(0,u)  \leq J(u) + C,
\end{equation}
because $d_1(0,u)= - I(u) = J(u) - \int_X u \omega^n,$ 
and it is well known that $\sup_X u- \int_X u \omega^n$ is always uniformly bounded (cf. \cite[Proposition 5.5]{DR17}).

\smallskip
We now prove the main result of this section, pointing out a strong relationship between properness of the $K$ energy and $K^\beta$ energies.
\begin{theorem}
\label{thm:K-prop=K-beta-prop}
     The following are equivalent: \smallskip\\
\noindent (i) $K^\beta$ energy is proper for some $\beta> 1$.\smallskip\\
\noindent (ii) $K$ energy is proper.\smallskip\\
\noindent (iii) $X$ admits a unique cscK metric in $\{\omega\}$.   
\end{theorem}

\begin{proof} Due to \cite{CC1,CC2}, it suffices to show that (ii) implies (i). Assume that for some $\delta,C_0>0$,
    $$
    K(u)=Ent(u)-\cJ_{\chi}(u)\geq\delta d_1(0,u)-C_0\text{ for any }u\in\cE^1 \textup{ with } \sup_X u =0.
    $$
    We show that $K^\beta$ is proper as well for all $\beta\gg 1$.
    
    Let $u\in\cE^1$ with $\sup_X u=0$. Using Lemma \ref{lem: Ent_beta_first}, we start a chain of inequalities:
    \begin{equation*}
    \begin{split}
         K^\beta(u)&=Ent^\beta(u)-\cJ_{\Ric(\omega)}(u)\geq Ent(u^\beta)-\cJ_{\Ric(\omega)}(u)\\
         &=Ent(u^\beta)-\cJ_{\Ric(\omega)}(u^\beta)+\cJ_{\Ric(\omega)}(u^\beta)-\cJ_{\Ric(\omega)}(u)\\
         &\geq\delta d_1(0,u^\beta-\sup_X u^\beta)-C_0-\frac{C}{\beta^{1/2^n}}\left(Ent^\beta(u)+d_1(0,u)+1 + \log\beta\right)\\
         &\geq \delta (d_1(0,u) -  d_1(u,u_\beta) - |\sup_X u^\beta|)-C-\frac{C}{\beta^{1/2^n}}\left(Ent^\beta(u)+d_1(0,u)+\log\beta\right)\\
                  &\geq \delta d_1(0,u) -C-\frac{(1+2\delta)C}{\beta^{1/2^n}}\left(Ent^\beta(u)+d_1(0,u)+\log\beta\right),\\
    \end{split}  
    \end{equation*}
    where we have repeatedly used the estimates of Proposition \ref{prop:PC0}, Proposition \ref{prop: J_ric_perturb} and Lemma \ref{lem: sup_perturb_lemma}. Reorganizing terms in the above estimate, we arrive at 
    \begin{equation*}
        \begin{split}
            &\left(1 +\frac{(1+2\delta)C}{\beta^{1/2^n}}\right)K^\beta(u)+\frac{(1+2\delta)C}{\beta^{1/2^n}}\cJ_{\Ric(\omega)}(u)\\
            &\geq\left(\delta-\frac{(1+2\delta)C}{\beta^{1/2^n}}\right)d_1(0,u)-C-\frac{(1+2\delta)C}{\beta^{1/2^n}}\log\beta.
        \end{split}
    \end{equation*}
   Finally, using that $|\cJ_{\Ric(\omega)}(u)|\leq Cd_1(0,u)$ (\cite[Proposition 2.5]{DH17}) and choosing $\beta$ to be sufficiently large (depending only on $X,\omega,n,\delta$), we obtain $\delta'>0$ and $C_0'>0$ such that
   $
   K^\beta(u)\geq\delta' d_1(0,u)-C_0'
   $
   for any $u\in\cE^1$ with $\sup_X u=0$. This completes the proof.
\end{proof}

The following estimate can be extracted from the above proof:

\begin{corollary}
\label{cor:Kbeta-u>K-u-beta} For any $\beta > 1$, there exists a constant $\varepsilon_\beta>0$, depending only on $X,\omega,n,\beta$ such that $\varepsilon_\beta\to 0$ as $\beta\to \infty$, and that
    $$
   (1+\varepsilon_\beta)K^\beta(u)\geq K(u^\beta)-\varepsilon_\beta d_1(0,u)-\varepsilon_\beta, \ \ u\in\cE^1 \ \textup{ with } \ \sup_X u = 0.
    $$
\end{corollary}

\section{Continuity of the radial $\beta$-entropy}

Given $\{u_t\} \in \mathcal R^1$, we define the radial $\beta$-entropy, as follows:
\begin{equation}\label{eq: beta_Ent_rad_def}
Ent^\beta\{u_t\}:=\liminf_{t\to\infty}\frac{Ent^\beta(u_t)}{t}.
\end{equation}
By Lemma \ref{lem:d1-u-beta<d1-u}, one has $Ent^\beta(u)\leq \beta d_1(u,u^\beta)\leq 2\beta d_1(0,u)$ for any $u\in \cE^1$ with $\sup_X u=0$, from which it follows easily that 
\begin{equation}
    \label{eq:rad-ent-beta-finite}
    Ent^\beta\{u_t\}<\infty \text{ for any } \{u_t\} \in \mathcal R^1.
\end{equation}

To start, we prove the following $d_1$-stability result for our transcendental quantization.

\begin{proposition} Let $u_0,u_1 \in \mathcal E^1$, $u_0,u_1 \leq 0$. For some $C_n>0$ the following holds:
$$d_1(u_0^\beta, u^\beta_1) \leq C_n d_1(u_1,u_0)^{\frac{1}{2^n}}\max(d_1(0,u_0),d_1(0,u_1))^{1-\frac{1}{2^n}}.
$$
\end{proposition}
\begin{proof} Start with $ u_1 \geq u_0$ both smooth K\"ahler potentials. Let $u_t := (1-t)u_0 + t u_1, \ t \in [0,1]$. Differentiating the defining equation of $u^\beta_t$ (see \eqref{def:u-beta}) we get that $\frac{d}{dt} u^\beta_t  =  (u_1 - u_0) + \frac{1}{\beta} \Delta_{\omega_{u_t^\beta}} \frac{d}{dt} u^\beta_t$. This implies that
$$d_1(u^\beta_1,u^\beta_0) = I(u^\beta_1) - I(u^\beta_0) = \frac{1}{V}\int_0 ^1\int_X \frac{d}{dt} u^\beta_t \omega_{u^\beta_t}^n dt=\frac{1}{V}\int_0 ^1\int_X (u_1 - u_0) \omega_{u^\beta_t}^n dt. $$

From \cite[Lemma A.2]{BBJ18} it follows that
$$d_1(u^\beta_1,u^\beta_0) \leq C_n 
d_1 (u_0,u_1)^{\frac{1}{2^n}} \int_0^1 \max (d_1(0, u^\beta_t),d_1(0,u_0), d_1(0,u_1))^{1 - \frac{1}{2^n}} dt,$$
where $C_n>0$ is a constant depending only on the dimension.

Using Lemma \ref{lem:d1-u-beta<d1-u}, we can write
$
d_1(0,u_t^\beta)\leq d_1(0,u_t)=-I(u_t)\leq-(1-t)I(u_0)-tI(u_1),
$
where we used the simple fact that $t\mapsto I(u_t)$ is concave. Thus,
$
d_1(0,u_t^\beta)\leq\max\{d_1(0,u_0),d_1(0,u_1)\}.
$
As a result, we obtain that 
$$d_1(u^\beta_1,u^\beta_0) \leq  C_n d_1 (u_0,u_1)^{\frac{1}{2^n}} \max (d_1(0,u_0), d_1(0,u_1))^{1 - \frac{1}{2^n}}.$$
This can now be extended for $u_1 \geq u_0$, both potentials in $\mathcal E^1$, using approximation \cite{BK07}.

For general $u_0,u_1 \in \mathcal E^1$, we can write that
\begin{flalign*}    d_1(u^\beta_1,P(u_0,u_1)^\beta) &\leq  C_n d_1 (u_1,P(u_0,u_1))^{\frac{1}{2^n}}  \max (d_1(0,u_1), d_1(0,P(u_0,u_1)))^{1 - \frac{1}{2^n}}\\
&\leq C'_n d_1(u_1,u_0)^{\frac{1}{2^n}}\max(d_1(0,u_0),d_1(0,u_1))^{1-\frac{1}{2^n}},
\end{flalign*}
where in the last inequality we used that $d_1(0,P(u_0,u_1)) \leq d_1 (0,u_0) + d_1 (u_0, P(u_0,u_1)) \leq d_1 (0,u_0) + d_1 (u_0, u_1) \leq 3 \max (d_1 (0,u_0), d_1 (0,u_1))$ .

Therefore, using symmetry, the Pythagorean identity \cite{Dar15} and the fact that $P(u^\beta_0,u^\beta_1) \geq P(u_0,u_1)^\beta$, we obtain that 
\begin{flalign*}   d_1(u_1^\beta,u_0^\beta)&=d_1(u_1^\beta,P(u_0^\beta,u_1^\beta))+d_1(u_0^\beta,P(u_0^\beta,u_1^\beta))\leq d_1(u^\beta_1,P(u_0,u_1)^\beta)+d_1(u^\beta_0,P(u_0,u_1)^\beta)\\
&\leq C''_n d_1(u_1,u_0)^{\frac{1}{2^n}}\max(d_1(0,u_0),d_1(0,u_1))^{1-\frac{1}{2^n}}.
\end{flalign*}
\end{proof}

Using the above result, we obtain that $Ent^\beta\{\cdot\}$ and $K^\beta\{\cdot\}$ are $d_1^c$-continuous, the main result of this section. The following argument shows that they are in fact H\"older continuous.

\begin{theorem}\label{thm:Ent-beta=lim-Ent-beta} 
    Let $\{u^m_{t}\},\{u_{t}\} \in \mathcal R^1$ such that $d_1^c(\{u^m_{t}\},\{u_t\})\to 0$. Then
    \begin{equation}
        \label{eq:K-beta-chordal-continuous}
        Ent^\beta\{u_t\}=\lim_{m\to\infty}Ent^\beta\{u^m_{t}\}, \ \ \         K^\beta\{u_t\}=\lim_{m\to\infty}K^\beta\{u^m_{t}\}.
    \end{equation}
\end{theorem}

\begin{proof} Recall from \cite[Lemma~3.2]{DZ22} that for \(\{v_t\} \in \mathcal{R}^1\) there exists \(c_v \in \mathbb{R}\) with \(\sup_X v_t = c_v t\) for \(t>0\), and \(\{v_t\} \mapsto c_v\) is \(d_1^c\)-continuous. Thus, after possibly replacing $u_t$ and $u^m_{t}$ by \(u_t - ct\) and \(u^m_{t} - ct\) for sufficiently large \(c>0\), we may assume that \(u_t, u^m_{t} \leq 0\).

Using the previous result and Proposition \ref{prop:ent-beta=sup}, for $l>0$ fixed, we have:
$$\Big|\frac{Ent^\beta(u_l) - Ent^\beta(u^m_{l})}{l}\Big| \leq  C_n\beta  \frac{d_1 (u_l,u^m_{l})^{\frac{1}{2^n}}}{l^{\frac{1}{2^n}}} \max \Big(\frac{d_1(0,u_l)}{l}, \frac{d_1(0,u^m_{l})}{l}\Big)^{1 - \frac{1}{2^n}} + \beta \frac{d_1 (u_l,u^m_{l})}{l}.$$
By Buseman convexity, $\frac{d_1 (u_l,u^m_{l})}{l} \leq d_1^c(\{u_t\},\{u^m_{t}\}) \to 0$ for fixed $l >0$, hence:
\begin{flalign*}
\big|Ent^\beta\{u_{t}\} - Ent^\beta\{u^m_{t}\}\big| \leq &  C_n\beta d_1^c (\{u_{t}\},\{u^m_{t}\})^{\frac{1}{2^n}} \max (d^c_1(\{0\},\{u_{t}\}), d^c_1(\{0\},\{u^m_{t}\}))^{1 - \frac{1}{2^n}}\\
&+ \beta d_1^c (\{u_{t}\},\{u^m_{t}\}) .
\end{flalign*}
The $d_1^c$-continuity of $Ent^\beta\{\cdot\}$ follows. The $d_1^c$-continuity of $K^\beta\{\cdot\}$ immediately follows from the next lemma.
\end{proof}

\begin{lemma}\label{I_Ric_cont} Let $u,v \in \mathcal E^1$ with $u,v \leq 0$. Then
$$\big| I_{\Ric \omega}(u)-I_{\Ric \omega}(v) \big| \leq C d_1 (u,v)^{\frac{1}{2^n}} \max(d_1 (0,u), d_1 (0,v))^{1 - \frac{1}{2^n}}.$$
\end{lemma}
\begin{proof}
Using the cocycle formula of $I_{\Ric \omega}(u)$ we obtain existence of $C(X,\omega)>0$
such that 
$$C| I_{\Ric \omega}(u)-I_{\Ric \omega}(v) \big| \leq  \int_X |u -v| \omega^n_{ \frac{u+v}{4}} = \int_X (\max(u,v) - u) +  (\max(u,v) - v)\omega^n_{ \frac{u+v}{4}}.$$
Separating the last integral, we estimate the first term.  Using \cite[Lemma A.2]{BBJ18} the same way as in \eqref{eq: u_u_beta_est} we obtain 
\begin{flalign*}
\int_X (\max(u,v) - u) & \omega^n_{ \frac{u+v}{4}} \leq C d_1(\max(u,v), u)^\frac{1}{2^n} \max(J(u),J(\max(u,v)),J((u+v)/4))^{1 - \frac{1}{2^n}}\\
&\leq C d_1(\max(u,v), u)^\frac{1}{2^n} \max(-I(u),-I(\max(u,v)),-I((u+v)/4))^{1 - \frac{1}{2^n}}.
\end{flalign*}
Where in the last line we used that  $0 \geq \max(u,v) , u, (u+v)/4$. Using concavity and monotonicity of $I(\cdot)$, together with $d_1(0,u) = - I(u),d_1(0,v) = - I(v)$, we arrive at:
\begin{flalign*}
\int_X (\max(u,v) - u)\omega^n_{ \frac{u+v}{4}} &\leq C d_1(\max(u,v), u)^\frac{1}{2^n} \max(d_1(0,u),d_1(0,v))^{1 - \frac{1}{2^n}}.
\end{flalign*}
Similarly,
$\int_X (\max(u,v) - v)\omega^n_{ \frac{u+v}{4}} \leq C d_1(\max(u,v), v)^\frac{1}{2^n} \max(d_1(0,u),d_1(0,v))^{1 - \frac{1}{2^n}}.
$
After adding the last two estimates, the result follows from \cite[Remark 5.6]{Dar15}.
\end{proof}

In conclusion, all the ingredients necessary for the proof of Theorem \ref{thm:YTD'}, as presented in the introduction, have now been established.

\paragraph{A transcendental YTD correspondence}
We briefly discuss the YTD correspondence for transcendental $(X,\omega)$, not necessarily induced by an ample line bundle. In this case one can still introduce notions of test configurations and K-stability \cite{DR17a,SD18}.

\begin{definition}
\label{def:Kahler-TC}
A K\"ahler test configuration $\mathcal T$ of $(X,\omega)$ is a pair $\mathcal T := (\mathcal{X},\mathcal{A})$ such that: \smallskip

\noindent $\bullet$ $\mathcal{X}$ is a normal, compact, complex space,  
    $\pi : \mathcal{X} \to \mathbb{P}^1$ is a flat, surjective morphism,  
    and there exists a holomorphic $\mathbb{C}^*$-action $\rho$ on $\mathcal{X}$ compatible with the standard $\Bbb C^*$-action on $\mathbb{P}^1$.
\smallskip

\noindent $\bullet$ There exists a $\mathbb{C}^*$-equivariant biholomorphism
    $
        \mathcal{X} \setminus \pi^{-1}(0) \ \cong\ X \times (\mathbb{P}^1 \setminus \{0\}),
    $
    identifying $\pi$ with projection onto $\mathbb{P}^1$.
\smallskip

\noindent $\bullet$ $\mathcal{A} \in H^{1,1}_{BC}(\mathcal{X},\mathbb{R})$ is a $\mathbb{C}^*$-invariant Bott-Chern cohomology class,  and 
    $
        \mathcal{A}|_{X \times (\mathbb{P}^1 \setminus \{0\})} = p_1^*[\omega],
    $
    where $p_1$ is projection to $X$.
\smallskip

\noindent $\bullet$   $\mathcal{A}$ is \emph{relatively K\"ahler}, i.e. $\cA+c \pi^*c_1(\cO_{\PP^1}(1))$ is a K\"ahler class for some $c>0$.
\end{definition}

As in the ample case, it is possible to associate to $\mathcal T$ a $C^{1,1}$ geodesic ray $\{ u^\mathcal T_t \}$ (\cite[Lemma 4.6]{SD18}, \cite{CTW18}), allowing to define $K^\beta$-stability in this context in the same way as in the algebraic case (recall \eqref{eq: K_stab_def}). 

A combination of \cite[Proposition 6.3.1]{Pic24} and  \cite[Theorem 5.1.7]{Pic24} implies that finite energy geodesic rays with bounded $K$ energy slope are $d_1^c$-approximable by rays of K\"ahler test configurations. As a result, the exact same proof as that of Theorem 1.1 implies the following:

\begin{theorem} Let $(X,\omega)$ be a compact K\"ahler manifold.
\label{thm:YTD''}
    The following are equivalent: \smallskip\\
\noindent (i) $X$ admits a unique cscK metric in $\{\omega\}$. \smallskip\\
\noindent (ii) $X$ is uniformly $K^\beta$-stable for some $\beta>0$.
\end{theorem}

\section{A slope formula for the $\beta$-entropy}
\label{sec:slope-formula}

The rest of this paper is devoted to establishing radial and algebraic formulations of $K^\beta$-stability.
Let $(X,L)$ be a polarized manifold, with $\omega\in c_1(L)$.
In this technical section we provide a useful slope formula for our $\beta$-entropy, which is the radial version of  \eqref{eq: Ent_beta_def-E1}:

\begin{theorem}\label{thm:rad-Ent-beta=sup}
    For any $\beta>1$ and any  bounded geodesic ray $\{u_t\}$,
    we have that
\begin{equation}
    \label{eq:rad-Ent-beta-formula}
    Ent^\beta\{u_t\}=\sup_{\{v_t\}}\left(L^\beta(\{v_t\},\{u_t\})+\beta (I\{v_t\}-I\{u_t\})\right),
\end{equation}
    where the sup is over all geodesic rays $\{v_t\}$ arising from test-configurations of $(X,L)$, and
    \begin{equation}
    \label{eq:def-L-beta}
    L^\beta(\{v_t\},\{u_t\}):=\liminf_{t\to\infty}\left(-\frac{1}{t}\log\int_Xe^{\beta(v_t-u_t)}\frac{\omega^n}{V}\right).
\end{equation}
\end{theorem}

The boundedness assumption on $\{u_t\}$ is naturally satisfied when it comes to applications, since it holds for a large class of geodesic rays (including the ones induced from filtrations; see \cite[\S 9]{RWN14}). That being said, it would be nice to know if this condition is necessary.

The proof of Theorem \ref{thm:rad-Ent-beta=sup} relies on several technical ingredients. The main idea is that, although the ray \(\{u_t^\beta\}\) attains the supremum in \eqref{eq:rad-Ent-beta-formula}, it is far from being geodesic. We therefore approximate \(\{u_t^\beta\}\) by a sequence of geodesic rays arising from test configurations.

We begin with a slope formula for \(L^\beta(\{v_t\},\{u_t\})\), generalizing the corresponding formula for the Ding energy \cite{Ber16,DZ22}. As shown later (Theorem \ref{thm:L-beta-for-two-TC}), when both \(\{v_t\}\) and \(\{u_t\}\) arise from test configurations, \(L^\beta(\{v_t\},\{u_t\})\) can be expressed more precisely as the log canonical threshold of certain divisors.

\begin{proposition} \label{prop: liminf}
Let $\{u_t\}$ and $\{v_t\}$ be two sublinear finite energy subgeodesic rays. Suppose one of these subgeodesics is a bounded geodesic ray. Then 
$$L^\beta(\{v_t\},\{u_t\})= \sup\left\{\tau\in\RR\ :\int_0^\infty e^{\tau t}\int_Xe^{\beta(v_t-u_t)}\omega^n dt<\infty\right\}.$$
\end{proposition}

\begin{proof} We can assume that both sides are finite.
Denote the right hand side by $\tau_0$.

For any $\tau < L^\beta(\{v_t\},\{u_t\})$, pick $\varepsilon>0$ such that $\tau+\varepsilon<L^\beta(\{v_t\},\{u_t\})$. Then for $t\gg 1$,
    $$
    -\log\int_X e^{\beta(v_t-u_t)}\omega^n\geq\tau t+\varepsilon t.
    $$
    Thus for $t_0$ big enough,
    $
    \int_{t_0}^\infty e^{\tau t}\int_X e^{\beta(v_t-u_t)}\omega^n dt\leq\int_{t_0}^\infty e^{-\varepsilon t}dt<\infty.
    $
    So, $L^\beta(\{v_t\},\{u_t\})\leq \tau_0$.

For the converse inequality, we rely on the boundedness assumption. In what follows, we suppose that $\{u_t\}$ is a bounded ray, the proof in the other case is the same.

Since we are dealing with sublinear subgeodesic rays, we can assume that $t \to u_t$ and $t \to v_t$ are both $t$-decreasing. Also, \cite[Theorem 1]{Dar17c} gives some $\gamma >0$ such that
\begin{equation}\label{eq: LIp_const}
|u_t(x) - u_{t'}(x)| \leq \gamma |t - t'|, \ \ x \in X
\end{equation}

Suppose that $ \tau < \tau_0$. Fixing $t_0 >1$, we can use the above inequality to start writing 
\begin{flalign*}
\int_{t_0-1}^{t_0} e^{\tau t} & \int_X e^{\beta(v_{t_0}-u_{t_0} - \gamma)}\omega^n dt   \leq\int_{t_0-1}^{t_0} e^{\tau t}\int_X e^{\beta(v_{t_0}-u_{t_0-1})}\omega^n dt  \\
&\leq \int_{t_0-1}^{t_0 } e^{\tau t}\int_X e^{\beta(v_t-u_t)}\omega^n dt 
\leq     \int_0^\infty e^{\tau t}\int_X e^{\beta(v_t-u_t)}\omega^n dt< C < \infty.
\end{flalign*}

Carrying out the integration in $t$ on the left hand side we arrive at
$$ \frac{e^{-\beta \gamma + t_0 \tau}}{\tau} \big(1 -  e^{- \tau}\big)   \int_X e^{\beta(v_{t_0}-u_{t_0})}\omega^n \leq C .$$
Taking $-\log$ we conclude at
$$   -\log \int_X e^{\beta(v_{t_0}-u_{t_0})}\omega^n  \geq t_0 \tau -\beta \gamma +\log \frac{1 -  e^{- \tau}}{\tau}-\log C. $$
Dividing by $t_0$ and letting $t_0 \to \infty$ we arrive at
$  \liminf_{t_0 \to \infty} - \frac{1}{t_0}\log \int_X e^{\beta(v_{t_0}-u_{t_0})}\omega^n \geq  \tau.$
So we obtain $L^\beta(\{v_t\},\{u_t\})\geq\tau_0$, finishing the proof.
\end{proof}

Now we move on to proving Theorem \ref{thm:rad-Ent-beta=sup}, which will be carried out in three steps.

\paragraph{Step 1. Replacing $\{u_t^\beta\}$ with a maximal geodesic ray.} For this step we fix a geodesic ray $\{u_t\} \in \mathcal R^1$ normalized by
$\sup u_t=0$ (this is possible due to \cite[Lemma 3.2]{DX20}).

For any $\beta>1$, consider the (non-geodesic) ray $\{u_t^\beta\}$ arising from solving \eqref{def:u-beta}. 
The goal of this step is to find a geodesic replacement for $\{u_t^\beta\}$, with favorable properties.

By Proposition \ref{prop:PC0}, the definition of $Ent^\beta\{u_t\}$ and \eqref{eq:rad-ent-beta-finite}, it is possible to choose a sequence $t_j\to\infty$ such that
\begin{equation}\label{eq: main_limit}
\lim_{j\to\infty}\frac{d_1(u_{t_j}^\beta,u_{t_j})}{t_j}=\lim_{j\to\infty}\frac{Ent^\beta(u_{t_j})}{\beta t_j}=\lim_{j\to\infty}\frac{I(u^\beta_{t_j})-I(u_{t_j})}{t_j}
=\frac{Ent^\beta\{u_t\}}{\beta}<\infty.
\end{equation}
Now let
$
[0,t_j]\ni t\mapsto\phi_{j,t,\beta} \in \mathcal E^1$
be the weak geodesic segment connecting $0$ and $u^\beta_{t_j}$.
Due to the convexity of K energy and Corollary \ref{cor:Kbeta-u>K-u-beta}, for $t < t_j$ we have
$$
\frac{K(\phi_{j,t,\beta})}{t}\leq\frac{K(u^\beta_{t_j})}{t_j} \leq \frac{(1+\varepsilon_\beta)K^\beta(u_{t_j})+\varepsilon_\beta d_1(0,u_{t_j})+\varepsilon_\beta}{t_j}.
$$
So by a standard compactness argument (cf. \cite{DH17,BBJ18,DL19}), we can extract a subsequence, again denoted by $\{\phi_{j,t,\beta}\}$, and obtain a finite energy geodesic ray $\{\phi_{t,\beta}\}$ such that
$$
\frac{K(\phi_{t,\beta})}{t}\leq (1+\varepsilon_\beta)K^\beta\{u_t\}- \varepsilon_\beta I\{u_t\}<\infty \ \ \textup{ and } \ \ 
\phi_{j,t,\beta}\xrightarrow{d_1}\phi_{t,\beta}\text{ as }j\to\infty.
$$

Then for any $t>0$, using \eqref{eq: main_limit}, we obtain a chain of inequalities:
\begin{equation*}
    \begin{split}
        \frac{I(\phi_{t,\beta})-I(u_t)}{t}&\leq\frac{d_1(\phi_{t,\beta},u_t)}{t}=\lim_{j\to\infty}\frac{d_1(\phi_{j,t,\beta},u_t)}{t}\leq\lim_{j\to\infty}\frac{d_1(u_{t_j}^\beta,u_{t_j})}{t_j}\\
        &=\lim_{j\to\infty}\frac{I(u^\beta_{t_j})-I(u_{t_j})}{t_j}=\lim_{j\to\infty}\frac{I(\phi_{j,t,\beta})-I(u_t)}{t}=\frac{I(\phi_{t,\beta})-I(u_t)}{t},
    \end{split}
\end{equation*}
where in the second inequality we used the convexity of $t\mapsto d_1(\phi_{j,t,\beta},u_t)$ \cite[Proposition 5.1]{BDL17} and in the second line we used the linearity of $t\mapsto I(\phi_{j,t,\beta})-I(u_t)$.

As a consequence, we obtain equality everywhere:
\begin{equation}\label{eq: I_d_1_eqv}
\frac{I(\phi_{t,\beta})-I(u_t)}{t}=\frac{d_1(\phi_{t,\beta},u_t)}{t}=\lim_{j\to\infty}\frac{I(u^\beta_{t_j})-I(u_{t_j})}{t_j}\text{ for any }t>0.
\end{equation}
Therefore, due to \eqref{eq: main_limit},
\begin{equation}\label{eq: lim_cons}
Ent\{u_t\}   \geq \lim_{j\to\infty}\frac{Ent^\beta(u_{t_j})}{t_j}=Ent^\beta\{u_t\}=\beta(I\{\phi_{t,\beta}\}-I\{u_t\})=\beta d_1^c(\{\phi_{t,\beta}\},\{u_t\}),
\end{equation}
and 
\begin{equation}
    \label{eq:rad-K-phi-finite}
    K\{\phi_{t,\beta}\}\leq (1+\varepsilon_\beta)K^\beta\{u_t\}-\varepsilon_\beta I\{u_t\}<\infty.
\end{equation}
In particular, due to \cite[Theorem 1.2]{Li20-cscK}$, \{\phi_{t,\beta}\}$ is maximal in the sense of \cite{BBJ18}. \smallskip

Due to \eqref{eq: I_d_1_eqv} we have that $I(\phi_{t,\beta}) - I(u_t) = I(\phi_{t,\beta}) + I(u_t) - 2 I(P(u_t, \phi_{t,\beta}))$. Using the $\mathcal E^1$ version of \cite[Proposition 2.3]{Dar17c} we obtain that $P(u_t, \phi_{t,\beta}) = u_t$. From this and the fact that  $u^\beta_{t_j} \leq 0$ we obtain that 
\begin{equation}\label{eq: phi_u_ineq}
0 \geq \phi_{t, \beta} \geq u_t \ \textup{ and } \ \sup_X \phi_{t,\beta}=\sup_X u_t =0.
\end{equation}

A byproduct of the above construction is the following important theorem. In the setting of algebraic quantization, analogous results have been conjectured by the second author in \cite[Conjecture 3.3]{Zha-Chow}, but are  still open to our knowledge.

\begin{theorem}\label{thm: K_beta_conv} 
Let $\{u_t\} \in \mathcal R^1$.  
    Then
    $
    \lim_{\beta \to \infty } K^\beta \{u_t\}=K\{u_t\}.
    $
\end{theorem}
\begin{proof} Without loss of generality we can assume that $\sup_X u_t =0, \ t \geq 0$.

When $Ent\{u_t\}<\infty$, by the above construction, for any $\beta>1$ the geodesic rays $\{\phi_{t,\beta}\}$ satisfy  $d_1^c(\{\phi_{t,\beta}\}, \{u_t\}) \to 0$  as $\beta\to\infty$ (recall \eqref{eq: lim_cons}. Moreover, due to   \eqref{eq:rad-K-phi-finite}) we have that  
   $$
  K\{\phi_{t,\beta}\}\leq (1+\varepsilon_\beta)K^\beta\{u_t\}-\varepsilon_\beta I\{u_t\}\leq (1+\varepsilon_\beta)K\{u_t\}-\varepsilon_\beta I\{u_t\}.
   $$
   Now sending $\beta\to\infty$ and using that $K\{\cdot\}$ is $d_1^c$-lsc (\cite[Proposition 5.9]{CC3})
    we obtain that
  $
\lim_{\beta\to\infty}K^\beta\{u_t\}=K\{u_t\}.
$

If $Ent\{u_t\}=\infty$, we claim that $\lim_{\beta \to \infty } Ent^\beta \{u_t\}=\infty$. Indeed, suppose otherwise, then Proposition \ref{prop:Ent-beta-increase} implies that $Ent^\beta \{u_t\}\leq C$ for some $C>0$ independent of $\beta>0$. So \eqref{eq: lim_cons} implies that $d_1^c(\{\phi_{t,\beta}\}, \{u_t\}) \to 0$ as $\beta\to\infty$. Sending $\beta\to\infty$ in \eqref{eq:rad-K-phi-finite} and using that $K\{\cdot\}$ is $d_1^c$-lsc and $K^\beta \{u_t\}\leq C$, we obtain $K\{u_t\}<\infty$, a contradiction. So we conclude.
\end{proof}

\paragraph{Step 2. Integrability estimates for $\{\phi_{t,\beta}\}$.} 
For this step we again fix a geodesic ray $\{u_t\} \in \mathcal R^1$ normalized by $\sup u_t=0$. This part is devoted to proving the following integrability condition for the geodesic ray $\{\phi_{t,\beta}\}$ constructed in the previous step.

\begin{proposition}
    \label{prop:int-estimate-phi-u} For any $\varepsilon>0$ and $\beta>1$,
    $$
    \int_0^\infty e^{-2\varepsilon t}\int_Xe^{(1-\varepsilon)\beta(\phi_{t,\beta}-u_t)}\omega^n dt<\infty,
    $$
    where $\{\phi_{t,\beta}\}$ is the geodesic ray constructed in Step 1. 
\end{proposition}

Recall that, by Proposition \ref{prop:u-t-beta-is-subgeodesic}, there exists $A\geq 0$ (depending only on $(X,\omega)$), such that for all $\beta\geq A$ the curve 
$$t\mapsto u_{t,\beta}:=((1-A/\beta)u_t)^\beta$$ 
is a sublinear finite energy subgeodesic ray. When $\Ric(\omega)\geq 0$, we can simply take $A=0$.
Let
\begin{equation}
    \label{eq:def-v-t-beta}
    \{v_{t,\beta}\}
\end{equation}
be the smallest geodesic ray lying above $\{u_{t,\beta}\}$. Such $\{v_{t,\beta}\}$ can be explicitly constructed as the increasing limit of the geodesic segments $[0,k]\ni t\mapsto \phi_{k}(t)$ joining $0$ and $u_{k,\beta}$ (for details, see the proof of \cite[Proposition 3.13]{DZ22}).

Then for any $t>0$, using that K energy is $d_1$-lsc and convex we can write
\begin{equation*}
    \begin{split}
        \frac{K(v_{t,\beta})}{t}\leq\liminf_{k\to\infty}\frac{K(\phi_k(t))}{t}\leq\liminf_{k\to\infty}\frac{K(u_{k,\beta})}{k}.
    \end{split}
\end{equation*}
Then by Corollary \ref{cor:Kbeta-u>K-u-beta} and (the proof of) \eqref{eq:rad-ent-beta-finite},
$
K\{v_{t,\beta}\}<\infty.
$
Thus, $\{v_{t,\beta}\}$ must be  maximal in the sense of \cite{BBJ18}, due to \cite[Theorem 1.2]{Li20-cscK}. Since $\{v_{t,\beta}\}$ is the smallest possible ray dominating the subgeodesic $\{u_{t,\beta}\}$, in the jargon of \cite{BBJ18}, $\{v_{t,\beta}\}$ is the maximal geodesic ray associated to  the subgeodesic $\{u_{t,\beta}\}$.

Thus the following estimate holds, thanks to \cite[Lemma 3.1]{Li20-cscK}:

\begin{lemma}
\label{lme:Li-int-est}
    For any $\alpha>0$ we have that
    $$
    \int_0^\infty e^{-t}\int_X e^{\alpha(v_{t,\beta}-u_{t,\beta})}\omega^n dt<\infty.
    $$
\end{lemma}

\begin{proof}
\cite[Lemma 3.1]{Li20-cscK} proves this  for geodesic rays, but  the same argument works for the subgeodesic ray \(\{u_{t,\beta}\}\) (c.f. \eqref{eq:int-ineq-of-Li}), since \(u_{t,\beta} \le 0\) due to Lemma \ref{lem: u_beta_basic}(ii) and \(u_t \le 0\) .
\end{proof}

\begin{corollary}
\label{cor:int-maxray-subray<infty}
     For any $\varepsilon>0$, $\beta>A$ we have that 
    $$
    \int_0^\infty e^{-2\varepsilon t}\int_X e^{(1-\varepsilon)\beta(v_{t,\beta}-(1-A/\beta)u_t)}\omega^n dt<\infty.
    $$  
\end{corollary}

\begin{proof}

Put for simplicity $Y:=X\times(0,\infty)$ and $d\mu:=\omega^n\wedge dt$. By H\"older's inequality,
\begin{equation*}
    \begin{split}
        \int_Y & e^{(1-\varepsilon)\beta(v_{t,\beta}-(1-\frac{A}{\beta})u_t)-2\varepsilon t}d\mu= \int_Y e^{(1-\varepsilon)\beta(v_{t,\beta}-u_{t,\beta})-\varepsilon t}\cdot e^{(1-\varepsilon)\beta(u_{t,\beta}-(1-\frac{A}{\beta})u_t)-\varepsilon t}d\mu\\
        &\leq \left(\int_Y e^{\frac{1-\varepsilon}{\varepsilon}\beta(v_{t,\beta}-u_{t,\beta})- t}d\mu\right)^{\varepsilon}\left(\int_Y e^{\beta(u_{t,\beta}-(1-\frac{A}{\beta})u_t)-\frac{\varepsilon}{1-\varepsilon} t}d\mu\right)^{1-\varepsilon}<\infty,\\
    \end{split}
\end{equation*}
finishing the proof by the previous lemma and the definition of $u_{t,\beta}$.
\end{proof}

The next simple comparison principle will play a key role in our argument.

\begin{lemma}
\label{lem:1-delta-u-beta-comparison}
    For any $\delta\in (0,1)$, $\beta>0$ and $u\in\cE^1$, we have that
    $$
(1-\delta)u^\beta\leq ((1-\delta)u)^{\frac{\beta}{1-\delta}}-\frac{n(1-\delta)}{\beta}\log (1-\delta).
$$
\end{lemma}

\begin{proof}
    Consider $(1-\delta)u^{\beta}$, which satisfies
    $$
    (\omega+\ddc (1-\delta)u^{\beta})^n\geq (1-\delta)^n e^{\beta(u^{\beta}-u)}\omega^n= e^{\frac{\beta}{1-\delta}((1-\delta)u^{\beta} + \frac{n(1-\delta)}{\beta} \log(1-\delta)-(1-\delta)u)}\omega^n.
    $$
    On the other hand, $((1-\delta)u)^\frac{\beta}{1-\delta}$ satisfies
    $$
    \big(\omega+\ddc ((1-\delta)u)^\frac{\beta}{1-\delta}\big)^n=e^{\frac{\beta}{1-\delta}\big(((1-\delta)u)^{\frac{\beta}{1-\delta}}-(1-\delta)u\big)}\omega^n.
    $$
    So we conclude by the comparison principle \cite[Lemma 2.5]{Ber19}. 
\end{proof}

Let us fix $\delta \in (0,1)$ by
$
\delta:=\frac{A}{A+\beta}.
$
Then $\frac{\beta}{1-\delta}=\beta+A$ and Proposition \ref{prop:u-t-beta-is-subgeodesic}  implies that $t\mapsto ((1-\delta)u_t)^{\frac{\beta}{1-\delta}} = ((1-\frac{A}{\beta + A })u_t)^{\beta + A}$ is a subgeodesic ray for any $\beta>0$. By our notation \eqref{eq:def-v-t-beta}, the smallest geodesic ray lying above it was denoted by $\{v_{t,\beta+A}\}$. 

We claim that
$
(1-\delta)\phi_{t,\beta}\leq v_{t,\beta+A}.
$
 Indeed, recall that $\phi_{t,\beta}$ is constructed as the limit of geodesic segments $\phi_{j,t,\beta}$ joining $0$ and $u^\beta_{t_j}$. Then
$
[0,t_j]\ni t\mapsto (1-\delta)\phi_{j,t,\beta}
$
is a subgeodesic segment joining $0$ and $(1-\delta)u_{t_j}^\beta$. Due to Lemma \ref{lem:1-delta-u-beta-comparison} and the comparison principle of weak geodesics, this subgeodesic is bounded from above by the geodesic segment joining $0$ and $((1-\delta)u_{t_j})^{\frac{\beta}{1-\delta}}-\frac{n(1-\delta)}{\beta}\log (1-\delta)$. Sending $j\to\infty$, the constant term $\frac{n(1-\delta)}{\beta}\log (1-\delta)$ will disappear in this limit process, and we obtain the claimed key relation:
\begin{equation}\label{eq: ray_ineq}
(1-\delta)\phi_{t,\beta}\leq v_{t,\beta+A}.
\end{equation}

Now, using \eqref{eq: ray_ineq}, we obtain the estimate of  Proposition \ref{prop:int-estimate-phi-u} as follows:
\begin{flalign*}
\int_0^\infty e^{-2\varepsilon t}\int_Xe^{(1-\varepsilon)\beta(\phi_{t,\beta}-u_t)}\omega^n dt& \leq \int_0^\infty e^{-2\varepsilon t}\int_Xe^{(1-\varepsilon)\frac{\beta}{1-\delta}(v_{t,\beta+A}-(1-\delta)u_t)}\omega^n dt \\
&= \int_0^\infty e^{-2\varepsilon t}\int_Xe^{(1-\varepsilon) (\beta + A)(v_{t,\beta+A}-(1-\frac{A}{A+\beta})u_t)}\omega^n dt<\infty.
\end{flalign*}
That the last integral is finite follows from Corollary \ref{cor:int-maxray-subray<infty} with $\beta$ replaced with $\beta+A$. So the proof of Proposition \ref{prop:int-estimate-phi-u} is complete.

\begin{corollary}
\label{cor:rad-Ent-beta-approx-by-geod-rays}
    For any $\beta>1$ and any  bounded geodesic ray $\{u_t\}$,
    there exists a sequence of geodesic rays $\{\psi_{t,k}\}$ such that
    $
    Ent^\beta\{u_t\}=\lim_{k\to\infty}(L^\beta(\{\psi_{t,k}\},\{u_t\})+\beta(I\{\psi_{t,k}\}-I\{u_t\})).
    $
Moreover, the sequence $\{\psi_{t,k}\}$ can be chosen so that 
$
Ent\{\psi_{t,k}\}<\infty,
$
$$
\int_0^\infty e^{-2\varepsilon_k t} \int_Xe^{\frac{\beta}{1-\varepsilon_k}(\psi_{t,k}-u_t)}\omega^n dt<\infty \ \textup{ for some } \ \varepsilon_k\to 0,$$
and
$$
\lim_{k\to\infty}\beta(I\{\psi_{t,k}\}-I\{u_t\})= Ent^\beta\{u_t\}.
$$
\end{corollary}

\begin{proof}

    We apply Proposition \ref{prop:int-estimate-phi-u} to the parameter $\beta_\varepsilon:=\frac{\beta}{(1-\varepsilon)^2}$. Then for any $\varepsilon\in(0,1)$,
    $$
    \int_0^\infty e^{-2\varepsilon t}\int_Xe^{\frac{\beta}{1-\varepsilon}\left(\phi_{t,\beta_\varepsilon}-u_t\right)}\omega^n dt<\infty.
    $$
Due to \eqref{eq: phi_u_ineq} we obtain that
\begin{equation}\label{eq: Holder}
    \int_0^\infty e^{-2\varepsilon t}\int_Xe^{\beta\left(\phi_{t,\beta_\varepsilon}-u_t\right)}\omega^n dt<\infty.
\end{equation}
    
Due to \eqref{eq: lim_cons} and Proposition \ref{prop:Ent-beta-increase} the geodesic ray $\{\phi_{t,\beta_\varepsilon}\}$ satisfies 
    \begin{equation}
        \label{eq:beta-I-phi-I-u>(1-e)Ent-beta}
        \beta(I\{\phi_{t,\beta_\varepsilon}\}-I\{u_t\})= (1-\varepsilon)^2Ent^{\beta_\varepsilon}\{u_t\}\geq (1-\varepsilon)^2Ent^\beta\{u_t\}.
    \end{equation}
    Moreover, by Corollary \ref{cor:Ent-beta-concave-in-beta},
    $
Ent^{\beta_\varepsilon}\{u_t\}\leq \frac{1}{(1-\varepsilon)^2}Ent^\beta\{u_t\}.
    $
    Theorefore,
    \begin{equation}\label{eq: I_Ent_conv}
    \beta(I\{\phi_{t,\beta_\varepsilon}\}-I\{u_t\})\to Ent^{\beta}\{u_t\}\text{ as }\varepsilon\to 0.
    \end{equation}
    Additinally, using \eqref{eq: Ent_beta_def-E1} we can write
    \begin{equation*}
        \begin{split}
            \frac{1}{t}Ent^\beta(u_t)&\geq \frac{-1}{t}\log\int_Xe^{\beta(\phi_{t,\beta_\varepsilon}-u_t)}\frac{\omega^n}{V}+\frac{\beta}{t}(I(\phi_{t,\beta_\varepsilon})-I(u_t))).\\
        \end{split}
    \end{equation*}
    Taking liminf on both sides and using \eqref{eq:beta-I-phi-I-u>(1-e)Ent-beta}, we obtain that
    \begin{equation*}
        \begin{split}
            Ent^\beta\{u_t\}\geq \liminf_{t\to\infty}(\frac{-1}{t}\log\int_Xe^{\beta(\phi_{t,\beta_\varepsilon}-u_t)}\frac{\omega^n}{V})+(1-\varepsilon)^2Ent^\beta\{u_t\}.
        \end{split}
    \end{equation*}
This implies that
    $$
    (1-(1-\varepsilon)^2) Ent^\beta\{u_t\}\geq \liminf_{t\to\infty}(\frac{-1}{t}\log\int_Xe^{\beta(\phi_{t,\beta_\varepsilon}-u_t)}\frac{\omega^n}{V})\geq -2\varepsilon,
    $$
    where in the last inequality we used Proposition \ref{prop: liminf} and \eqref{eq: Holder}. 
Choosing $\varepsilon_k\searrow 0$ and putting
    $$
    \psi_{t,k}:=\phi_{t,\beta_{\varepsilon_k}},
    $$
by \eqref{eq:rad-ent-beta-finite} and \eqref{eq:def-L-beta} we get 
$
\lim_{k\to\infty}L^\beta(\{\psi_{t,k}\},\{u_t\})=0.
$
Then we conclude by \eqref{eq: I_Ent_conv} and \eqref{eq:rad-K-phi-finite}.
\end{proof}

\paragraph{Step 3. Approximation via test configurations.} In this final step we will approximate each $\{\psi_{t,k}\}$ from the above corollary by a sequence of geodesic rays from test configurations and then complete the proof of Theorem \ref{thm:rad-Ent-beta=sup}.

\begin{proposition}
\label{prop:approx-psi-using-TC-rays}
Let $\{u_t\}$ be a bounded geodesic ray with $\sup u_t=0$.
    Assume that $\{\psi_t\}$ is a geodesic ray such that
    $
    Ent\{\psi_t\}<\infty
    $
    and
    $$
    \int_0^\infty e^{-2\varepsilon t} \int_Xe^{\frac{\beta}{1-\varepsilon}(\psi_{t}-u_t)}\omega^n dt<\infty
    $$
    for some $\beta>0$ and $\varepsilon>0$. Then there exists a sequence of geodesic rays $\{w_{t,m}\}$ induced from test configurations such that
    $$
   \lim_{m\to\infty} I\{w_{t,m}\}= I\{\psi_t\}\text{ and } \int_0^\infty e^{-4\varepsilon t} \int_Xe^{\beta(w_{t,m}-u_t)}\omega^n dt<\infty\text{ for all }m\gg 1.
    $$
\end{proposition}

\begin{proof}
    During the proof, $\beta>0$ and $\varepsilon>0$ will be fixed.    Since  $Ent\{\psi_t\}<\infty$, \cite[Theorem 1.2]{Li20-cscK} is applicable, and applying Demailly's regularization to $\{\psi_t\}$ gives a sequence of geodesic rays $\{w_{t,m}\}$ induced from test configurations such that $d_1^c(\{w_{t,m}\},\{\psi_t\}) \to 0$, hence
    $$
    I\{w_{t,m}\}\to I\{\psi_t\}\text{ as }m\to \infty.
    $$
    Moreover, for some sufficiently large $m_0>0$, we have that (recall \eqref{eq:int-ineq-of-Li} and $t=-\log|z|^2$)
    $$
    \int_0^\infty e^{-t}\int_X e^{(m+m_0)w_{t,m}-m\psi_t}\omega^n dt<\infty
    $$
    for all $m\gg 1$. We can and will assume that
    $
    m_0>\beta.
    $

    Put for simplicity that $Y:=X\times(0,\infty)$ and $d\mu:=\omega^n\wedge dt$. Let $\delta\in(0,\varepsilon)$, which will be fixed below. Then we can start writing
    \begin{equation*}
        \begin{split}
            \int_Y & e^{\beta(w_{t,m}-u_t)-4\varepsilon t}d\mu\leq \int_Y e^{\beta(w_{t,m}-(1-\delta)\psi_t)-\delta t}e^{\beta(1-\delta)(\psi_t-u_t)-2\varepsilon t}e^{-\beta\delta u_t-\varepsilon t}d\mu\\
            &\leq\left(\int_Y e^{\beta pw_{t,m}-\beta p (1-\delta)\psi_t-\delta p t} d\mu\right)^{1/p}\left(\int_Y e^{\beta q(1-\delta)(\psi_t-u_t)-2\varepsilon q t} d\mu\right)^{1/q}\left(\int_Y e^{-\beta\delta r u_t-\varepsilon r t} d\mu\right)^{1/r}.\\
        \end{split}
    \end{equation*}
Here $p,q,r>0$ are chosen such that
$$
\frac{1}{p}+\frac{1}{q}+\frac{1}{r}=1,\ 
\beta p=m+m_0,\ \beta p(1-\delta)=m,\ \beta q (1-\delta)=\frac{\beta}{1-\varepsilon}.
$$
These relations are satisfied by choosing
$$
p=\frac{m+m_0}{\beta},\ \delta=\frac{m_0}{m+m_0},\ q=\frac{m+m_0}{(1-\varepsilon)m},\ r=\frac{m+m_0}{\varepsilon m+m_0-\beta}.
$$
Therefore, for $m\gg 1$, we have $\delta <\varepsilon$ and
$$
\delta p=\frac{m_0}{\beta}>1,\ 2\varepsilon q=\frac{2\varepsilon}{1-\varepsilon}\frac{m+m_0}{m}>2\varepsilon,\ \beta\delta r=\frac{\beta m_0}{\varepsilon m+m_0-\beta}<\alpha(X,\omega).
$$
Here $\alpha(X,\omega)$ denotes Tian's alpha invariant \cite{Tian87}.
Then,
$$
\left(\int_Y e^{\beta pw_{t,m}-\beta p (1-\delta)\psi_t-\delta p t} d\mu\right)^{1/p}\leq \left(\int_Y e^{(m+m_0)w_{t,m}-m\psi_t-t} d\mu\right)^{1/p}<\infty,
$$
$$
\left(\int_Y e^{\beta q(1-\delta)(\psi_t-u_t)-2\varepsilon q t} d\mu\right)^{1/q}\leq \left(\int_Y e^{\frac{\beta}{1-\varepsilon}(\psi_t-u_t)-2\varepsilon t} d\mu\right)^{1/q}<\infty,
$$
$$
\left(\int_Y e^{-\beta\delta r u_t-\varepsilon r t} d\mu\right)^{1/r}\leq \left(C_\alpha\int_Y e^{-\varepsilon r t} d\mu\right)^{1/r}<\infty.
$$
So we conclude that
$
\int_Y e^{\beta(w_{t,m}-u_t)-4\varepsilon t}d\mu<\infty \text{ for all }m\gg 1, 
$
finishing the proof.
\end{proof}

Now we can prove the main result of this section.

\begin{proof}[Proof of Theorem \ref{thm:rad-Ent-beta=sup}]
    After possibly adding a linear term, we can assume that $\sup u_t=0$ \cite[Lemma 3.2]{DX20}. Due to the definition of $Ent^\beta$ \eqref{eq: Ent_beta_def-E1} it suffices to find a sequence of geodesic rays $\{w_{t,j}\}$ arising from   test-configurations such that
    $$
    Ent^\beta\{u_t\}=\lim_{j\to\infty} \left(L^\beta(\{w_{t,j}\},\{u_t\})+\beta(I\{w_{t,j}\}-I\{u_t\})\right).
    $$
    By  Corollary \ref{cor:rad-Ent-beta-approx-by-geod-rays}, we already found a sequence of geodesic rays $\{\psi_{t,k}\}$ with $Ent\{\psi_{t,k}\}<\infty,$
    $$
    \int_0^\infty e^{-2\varepsilon_k t}\int_Xe^{\frac{\beta}{1-\varepsilon_k}(\psi_{t,k}-u_t)}\omega^n dt<\infty, \ \textup{ and } \ 
    \lim_{k \to \infty}\beta(I\{\psi_{t,k}\}-I\{u_t\})= Ent^\beta\{u_t\}.
    $$
    Additionally, by Proposition \ref{prop:approx-psi-using-TC-rays}, for each $k>0$, we can find $m_k\gg 1$ and a geodesic ray $\{w_{t,m_k}\}$ arising from a test configuration such that
    $$
    \int_0^\infty e^{-4\varepsilon_k t}\int_Xe^{\beta(w_{t,m_k}-u_t)}\omega^n dt<\infty
\ \  \textup{and } \ \  
    |I\{\psi_{t,k}\}-I\{w_{t,m_k}\}|\leq 2^{-k}.
    $$

    Consequently, we have that
\begin{equation}\label{eq: I_limit_ent}
    \beta(I\{w_{t,m_k}\}-I\{u_t\})\to Ent^\beta\{u_t\}\text{ as }k\to\infty.
\end{equation}
    Moreover, for some  $\delta_k\to 0$, \eqref{eq: Ent_beta_def-E1} gives
    \begin{equation*}
        \begin{split}
            Ent^\beta\{u_t\}&\geq \liminf_{t\to\infty}\left(\frac{-1}{t}\log\int_X e^{\beta(w_{t,m_k}-u
    _t)}\frac{\omega^n}{V}\right)+\beta(I\{w_{t,m_k}\}-I\{u_t\})\\
    &\geq -4\varepsilon_k+Ent^\beta\{u_t\}-\delta_k,
        \end{split}
    \end{equation*}
   where we used Proposition \ref{prop: liminf}. This implies that 
   $
   \lim_{k\to\infty} L^\beta(\{w_{t,m_k}\},\{u_t\})=0.
   $
Coupled with \eqref{eq: I_limit_ent}, this finishes the proof.
\end{proof}

\begin{remark}
    In this section, the polarization assumption for $(X,\omega)$ is only used to invoke Demailly's approximation when proving Lemma \ref{lme:Li-int-est} and Proposition \ref{prop:approx-psi-using-TC-rays}. However, such an approximation procedure also holds for the general K\"ahler case, thanks to \cite[Proposition 6.3.1]{Pic24}. As a consequence, the transcendental version of Theorem \ref{thm:rad-Ent-beta=sup} holds as well, where the sup in \eqref{eq:rad-Ent-beta-formula} is then over all geodesic rays $\{v_t\}$ attached to K\"ahler test configurations.
\end{remark}

\section{Threshold type interpretation of $K^\beta$-stability}
\label{sec:slop-TC}

Assume that $(X,L)$ is a polarized manifold. Let $\omega\in c_1(L)$, so $V=L^n$. 
Assume that $\{u_t\}$ is a geodesic ray arising from a test configuration $\mathcal T = (\cX,\cL)$. Then we write
$$
Ent^\beta(\cX,\cL):=Ent^\beta\{u_t\},\ \cJ(\cX,\cL):=\cJ_{\Ric(\omega)}\{u_t\},\ 
$$
$$
K^\beta_\mathcal T:=K^\beta(\cX,\cL):=Ent^\beta(\cX,\cL)-\cJ(\cX,\cL).
$$

It is known from \cite{PS08,BHJ19} (see also \cite[Proposition 2.38]{Li20-cscK}) that
\begin{equation}\label{eq: J_Ric_test_conf}
\cJ(\cX,\cL)=-\frac{\bar S \cL^{n+1}}{(n+1)V}-\frac{K_X\cdot \cL^n}{V}.
\end{equation}
Here the intersection numbers are taken after a trivial compactification of $(\cX,\cL)$ over $\infty\in\PP^1$ and then pull back all the line bundles in question to a common normal model.
In what follows we will provide a purely algebraic description for $Ent^\beta(\cX,\cL)$ as well, in terms of log canonical thresholds and intersection numbers. 

 By Theorem \ref{thm:rad-Ent-beta=sup}, for any $\beta>1$,
\begin{equation}\label{eq: Ent_bet_form}
  Ent^\beta(\cX,\cL)=\sup_{\{v_t\}}\left(L^\beta(\{v_t\},\{u_t\})+\beta (I\{v_t\}-I\{u_t\})\right),
\end{equation}
where the sup is over all geodesic rays $\{v_t\}$ arising from test-configurations.
Let us therefore take an auxiliary test configuration $(\cX',\cL')$, with $\{v_t\}$ being the associated geodesic ray. 

We know from the slope formulas in \cite{PS08,BHJ19} that
\begin{equation}\label{eq: I_test_conf}
\beta(I\{v_t\}-I\{u_t\})=\frac{\beta(({\cL'})^{n+1}-\cL^{n+1})}{(n+1)V}.
\end{equation}
Therefore, it remains to deal with the term $L^\beta(\{v_t\},\{u_t\})$. We will express it in terms of certain log canonical threshold, as shown in the next result. It recovers Berman's formula for the Ding invariant in the Fano case \cite[(3.19)]{Ber16}, by taking $(\cX',\cL'):=(\cX,(1-\frac{1}{\beta})\cL+\frac{1}{\beta}L)$.

\begin{theorem}
\label{thm:L-beta-for-two-TC}
    Take a common normal model:
    $$
\begin{tikzcd}[column sep=small]
& \cZ \arrow[dl,"\mu" left, near start] \arrow[dr, "\sigma"] \arrow[d, "\tau"] & \\
  \cX'  &  X\times\PP^1    & \cX
\end{tikzcd}
$$
There exist $\QQ$-Cartier divisors $F$ and $G$ supported on $\cZ_0$ such that (cf. \cite[Propostion 3.10]{SD18})
$$
\tau^*p_1^*L-\sigma^*\cL\sim_\QQ F,\ \tau^*p_1^*L-\mu^*\cL'\sim_\QQ G.
$$
Define a $\QQ$-Weil divisor $\cD$ supported in $\cZ_0$ by
\begin{equation}\label{eq: D_def}
\cD:=\beta(F-G)-K_{\cZ}+\tau^*K_{X\times\CC}.
\end{equation}
Then
\begin{equation}\label{eq: L_beta_lct}
L^\beta(\{v_t\},\{u_t\})=\lct_{(\cZ,\cD)}(\cZ_0)-1.
\end{equation}
(See \cite[\S 1.5]{BHJ17} for the definition of log canonical threshold (lct) for a pair.)

Moreover, when $\cZ$ is smooth with simple normal crossing (snc) central fiber, write
$$
\cZ_0=\sum_i a_i E_i,\  \
F=\sum_i b_i E_i,\  \
G=\sum_i c_i E_i, \  \ K_\cZ-\tau^*K_{X\times\PP^1}=\sum_{i}d_i E_i.
$$
In this case, we can write more explicitly that
\begin{equation}
    \label{eq:L-beta-min}
    L^\beta(\{v_t\},\{u_t\})=\min_i \frac{(d_i+1-a_i)+\beta(c_i-b_i)}{a_i}.
\end{equation}
\end{theorem}

\begin{proof}
Put $U_z:=u_{-\log|z|^2}$ and $V_z:=v_{-\log|z|^2}$.
Applying the change of variables \(t = -\log|z|^2\) in the statement of Proposition \ref{prop: liminf}, we obtain
\begin{equation}
    \label{eq:L_beta=lct}
    L^\beta(\{v_t\},\{u_t\})=\sup\left\{\tau\in\RR: \int_{X\times \Delta^*} e^{\beta(V_z-U_z)-(1+\tau)\log|z|^2}\omega^n\wedge i dz\wedge d\bar z<\infty\right\}.
\end{equation}
    The result then simply follows from expanding all the terms appearing in the above integral.
    
    Indeed, when $\cZ$ is smooth with snc central fiber, the complexified geodesic rays $z \mapsto U_z, \ V_z$ when pulled back to $\cZ$ is $L^\infty$ compatible with the divisors $F$ and $G$, respectively (recall \eqref{eq:sing-type-of-alg-ray}).
Then the integrability of $e^{\beta(V_z-U_z)-(1+\tau)\log|z|^2}$ is equivalent to requiring that the coefficients of the divisor  
$
\beta(F-G) - K_{\mathcal Z / X \times \PP^1} + (1+\tau)\mathcal Z_0
= \sum_i \big( \beta(b_i - c_i) + a_i(1+\tau) - d_i \big) E_i
$
are all strictly less than \(1\), giving \eqref{eq:L-beta-min}.

Lastly, by the pull back formula of pairs (see \cite[\S 1.5]{BHJ17}), for a general normal model \(\mathcal Z\) dominating \(\mathcal X\),\ $X\times\PP^1$ and \(\mathcal X'\), the quantity \(L^\beta(\{v_t\},\{u_t\}) + 1\) can then be expressed as the lct of the pair  
$(\mathcal Z, \ \beta(F-G) - K_{\mathcal Z / X \times \mathbb C})$ with respect to $\cZ_0$, giving \eqref{eq: L_beta_lct}. So we conclude.
\end{proof}

Let $\mathcal T = (\mathcal X, \mathcal L)$ be a test configuration dominating $X\times\PP^1$. Let $Div_\mathcal X$ be the set of $\Bbb Q$-Cartier divisors $G$ contained in the central fiber $\mathcal Z_0$ of a $\Bbb C^*$-equivariant smooth model $\pi: \mathcal Z \to \mathcal X$, with $\pi$ bijective on $\mathcal X \setminus \mathcal X_0$, and $\mathcal L_G:=\pi^* p_1^* L - G$ relatively semi-ample on $\mathcal Z$. 

Thus defined, $(\mathcal Z, \mathcal L_G)$ is also a test configuration of $(X,L)$, and  we define the corresponding divisor $\mathcal D_G$ as in  \eqref{eq: D_def}:
\begin{equation}
    \label{eq:def-D-G}
    \cD_G:=\beta(F-G)-K_{\cZ}+\tau^*K_{X\times\PP^1}.
\end{equation}

Now we plug \eqref{eq: L_beta_lct} and \eqref{eq: I_test_conf} into \eqref{eq: Ent_bet_form}. Together with \eqref{eq: J_Ric_test_conf}, this gives a divisorial threshold interpretation for $K^\beta _\mathcal T = Ent^\beta(\mathcal X, \mathcal L) - \cJ(\mathcal X, \mathcal L)$:

\begin{corollary}\label{cor: Ent_beta_div_thresh} Let $\cT:=(\mathcal X, \mathcal L)$ be a test configuration dominating $X\times\PP^1$. With the notation introduced above, for any $\beta>1$ we have
\begin{equation}
    \label{eq:K-T-beta-express}
    K^\beta_\mathcal T = \sup_{G \in Div_\mathcal X} \Big(\lct_{(\mathcal \cZ, \mathcal D_G)}(\mathcal Z_0) - 1 +  \frac{\beta( {\mathcal L_G}^{n+1}-{\mathcal L}^{n+1})}{(n+1)V}\Big) + \frac{\bar S \cL^{n+1}}{(n+1)V}+ \frac{K_X\cdot\cL^n}{V}.
\end{equation}
\end{corollary}

Being the supremum of divisorial data, the above expression is reminiscent of other known divisorial threshold type expressions from the Fano case, like the delta and alpha invariants. 

\paragraph{$K^\beta$-stability vs. Ding stability.} When $(X,L)=(X,-K_X)$ is Fano, our $K^\beta(\cX,\cL)$ is directly related to $\mathrm{Ding}(\cX,\cL)$, the Ding invariant introduced in \cite{Ber16,BBJ18}. Indeed, choosing $G:=(1-\frac{1}{\beta})F$ in \eqref{eq:def-D-G} (i.e., taking $\cL_G:=(1-\frac{1}{\beta})\cL+\frac{1}{\beta}(-K_X)$), one then has (by \cite[(2.18)]{Xubook})
$
\mathrm{Ding}(\cX,\cL)=\lct_{(\mathcal \cZ, \mathcal D_G)}(\mathcal Z_0) - 1-\frac{ \cL^{n+1}}{(n+1)V}.
$
Thus, \eqref{eq:K-T-beta-express} implies that (using $\bar S=n$ here)
\begin{equation*}
    \begin{split}
        K^\beta(\cX,\cL)&\geq \mathrm{Ding}(\cX,\cL)+\frac{\beta({\cL_G}^{n+1}-{\cL}^{n+1})}{(n+1)V}+\frac{({\cL}-(-K_X))\cdot{\cL}^n}{V}\\
        &\geq \mathrm{Ding}(\cX,\cL)+\frac{\beta}{V}({\cL_G}-{\cL})\cdot{\cL_G}^n+\frac{({\cL}-(-K_X))\cdot{\cL}^n}{V}\\
        &= \mathrm{Ding}(\cX,\cL)+\frac{1}{V}(-K_X-{\cL})\cdot({\cL_G}^n-{\cL}^n),
    \end{split}
\end{equation*}
where in the second line we used the Hodge index type estimate \cite[Lemma 1.74]{Xubook} (cf. also (2.4) therein) and the last line follows from the relation $\beta({\cL_G}-\overline{\cL})=-K_X-{\cL}$. 

We further argue that the error term $\frac{1}{V}(-K_X-{\cL})\cdot({\cL_G}^n-{\cL}^n)$ appearing above can be conveniently controlled from below when $\beta\gg 1$. Indeed, up to adding some multiple $\cX_0$ to $\cL$, we can assume that $(-K_X-{\cL})$, when pulled back to a common normal model $\cZ$, is an effective $\QQ$-divisor supported in $\cZ_0$ and it does not contain any multiple of $\cZ_0$ (so that $\overline{\cL}\cdot (-K_X)^n=0$). Then using that ${\cL_G}-(1-\frac{1}{\beta}){\cL}$ is nef, we deduce that
\begin{equation*}
    \begin{split}
        \frac{1}{V}(-K_X-{\cL})\cdot({\cL_G}^n-{\cL}^n)\geq -\frac{(1-(1-\beta^{-1})^n)}{V}(-K_X-{\cL})\cdot{\cL}^n=-\varepsilon_\beta\mathbf{I}(\cX,\cL),
    \end{split}
\end{equation*}
where $\varepsilon_\beta:=(1-(1-\beta^{-1})^n)$ and $\mathbf{I}(\cX,\cL)$ denotes the $\mathbf{I}$-norm of $(\cX,\cL)$ (see \cite[(2.2)]{Xubook}). Using the relation $\mathbf{I}\leq (n+1)\mathbf{J}$ \cite[Proposition 2.9]{Xubook}, we finally arrive at $$K^\beta(\cX,\cL)\geq \mathrm{Ding}(\cX,\cL)-\varepsilon'_\beta \mathbf{J}(\cX,\cL).$$

Consequently, uniform Ding stability (meaning that $\mathrm{Ding}(\cX,\cL)\geq\gamma \mathbf{J}(\cX,\cL)$ for some uniform $\gamma>0$) implies uniform $K^\beta$-stability for $\beta\gg1$. To show the reverse implication, one can use that uniform $K^\beta$-stability implies uniform $K$-stability (as $K\geq K^\beta$) and then invoke the deep fact that the latter in turn implies uniform Ding stability \cite{BBJ18,Fuj19} (which relies on the MMP techniques of Li--Xu \cite{LX14}). Therefore, the above purely algebraic consideration shows that stability notions defined using $K$, $K^\beta$ and $\mathrm{Ding}$ are all equivalent for Fano manifolds. 

Beyond the Fano setting, we hope that \eqref{eq:K-T-beta-express} will lead to new valuative criterions for the existence of cscK metrics, generalizing those for KE metrics (cf. \cite{Tian87,Li17,Fuj19,FO18,BJ17,AZ20,Fano-3fold-book,Xubook}).

\section{Connections with non-Archimedean geometry}

\label{sec:NA}

We point out the relationship of our work with the non-Archimedean point of view and give an alternative argument of a recent remarkable result of Boucksom--Jonsson \cite{BJ25preprint}. 

Assume that $(X,L)$ is a polarized manifold with $\omega\in c_1(L)$. 
Let us first give a non-Archimedean interpretation for the radial $\beta$-entropy. In light of Theorem \ref{thm:rad-Ent-beta=sup}, it suffices to give a non-Archimedean formula for $L^\beta(\{v_t\},\{u_t\})$. We achieve this by adapting the treatment for the Ding energy in \cite[Theorem 5.4]{BBJ18}. 

\begin{proposition}
\label{prop:NA-L-beta}
     Let $\{u_t\}$ be a bounded geodesic ray,
    and let $\{v_t\}$ be a geodesic ray arising from a test configuration. Then for any $\beta>0$,
    $$
    L^\beta(\{v_t\},\{u_t\})=\inf_{w\in X_\QQ^{div}}(A_X(w)+\beta(u^{an}(w)-v^{an}(w))).
    $$
\end{proposition}

Here $u^{an}$ and $v^{an}$ are non-Archimedean potentials associated with $\{u_t\}$ and $\{v_t\}$ in the sense of \cite[Definition 4.2]{BBJ18} (see also \cite[Proposition 3.1]{DXZ25}). Moreover, $X^{div}_\QQ$ denotes the $\QQ$-valued divisorial valuations on $X$ (including the trivial one).

\begin{proof}
    Since both $\{u_t\}$ and $\{v_t\}$ are bounded rays, $L^\beta(\{v_t\},\{u_t\})$ is a finite number. Moreover, since the statement we aim to show is translation invariant, we can assume that $u_t\leq 0$ and $v_t\leq 0$ so that $U_z:=u_{-\log|z|^2}$ and $V_z:=v_{-\log|z|^2}$ are qpsh functions on $X\times\Delta$. By Proposition \ref{prop: liminf} and using $t=-\log|z|^2$ we have that
    $$
    L^\beta(\{v_t\},\{u_t\})=\sup\left\{\tau\in\RR: \int_{X\times \Delta} e^{\beta(V_z-U_z)-(1+\tau)\log|z|^2}\omega^n\wedge i dz\wedge d\bar z<\infty\right\}.
    $$

    Up to further subtracting $Ct$ from $\{v_t\}$ we can assume that $e^{\beta(V_z-U_z)}\in L^1(X\times\Delta)$. Put for simplicity
    $\mu:=e^{\beta(V_z-U_z)}\omega^n\wedge idz\wedge d\bar z$. Then, in terms of the complex singularity exponent (see \cite[\S 2.2]{DZ22}), we can write
    $$
    L^\beta(\{v_t\},\{u_t\})+1=c_\mu[\log|z|^2]:=\sup\left\{\lambda>0:\int_{X\times\Delta}e^{-\lambda\log|z|^2}d\mu<\infty\right\}.
    $$
    
    Noting that $V_z$ has algebraic singularities (recall \eqref{eq:U-has-analytic-sing}) and that $V_z$, $U_z$ and $\log|z|^2$ are all $S^1$-invariant and bounded away from $X\times\{0\}$, we can
    apply the valuative criterion \cite[Theorem B.7]{BBJ18}  in the same way as in the short proof of \cite[Theorem 2.3]{DZ22} to deduce that
$$
c_\mu[\log|z|^2]=\inf_E\frac{A_{X\times\Delta}(E)+\beta(\nu(V_z,E)-\nu(U_z,E))}{\nu(\log|z|^2,E)},
$$
where $E$ runs through all $\CC^*$-invariant prime divisors $E$ over $X\times\{0\}$.

    Consequently,
    $L^\beta(\{v_t\},\{u_t\})$ is equal to the largest $\tau\in\RR$ such that
\begin{equation}\label{eq: tau_cond}
A_{X\times\Delta}(E)+\beta\nu(V_z,E)-\beta\nu(U_z,E)\geq(1+\tau)\nu(\log|z|^2,E)
\end{equation}
holds for all $\CC^*$-invariant prime divisors $E$ over $X\times\{0\}$. 
Following \cite[Definition 4.4]{BHJ17}, put
$
b_E:=\nu(\log|z|^2,E).
$
By \cite[Theorem 4.6]{BHJ17}, all divisorial valuations $w\in X^{div}_{\QQ}$ (including the trivial valuation) are in one-to-one correspondence with $\frac{1}{b_E}\ord_E$, via the Gauss map/extension. In fact, the Gauss extension $\sigma(w)$ of any $w\in X^{div}_{\QQ}$ is of the form
$
\sigma(w)=\frac{1}{b_E}\ord_E
$
for some $E$ over $X\times\{0\}$ and we have the relation 
$
A_X(w)=\frac{1}{b_E}A_{X\times\Delta}(E)-1
$ (by \cite[Theorem A.10]{BJ22}).
Moreover, we have that (cf. \cite[Definition 4.2]{BBJ18})
$
u^{an}(w)=-\frac{1}{b_E}\nu(U,E)\text{ and } v^{an}(w)=-\frac{1}{b_E}\nu(V,E).
$
Revisiting \eqref{eq: tau_cond}, we obtain that $L^\beta(\{v_t\},\{u_t\})$ is equal to the largest $\tau\in\RR$ such that
$$
A_X(w)+\beta (u^{an}(w)-v^{an}(w))\geq \tau\text{ for all }w\in X^{div}_\QQ.
$$
Therefore,
$
L^\beta(\{v_t\},\{u_t\})=\inf_{w\in X^{div}_{\QQ}}(A_X(w)+\beta (u^{an}(w)-v^{an}(w))),
$
finishing the proof. 
\end{proof}

\begin{corollary}\label{cor: L_beta_L_NA}
Let $\{u_t\},\{v_t\}$ be geodesic rays induced by a model and a test configuration, respectively. Then for any $\beta > 0$:
\begin{equation}\label{eq: L_beta_L_NA}
    L^\beta(\{v_t\},\{u_t\})=L^{NA}(\beta({u}^{an}-v^{an})),
\end{equation}
where $L^{NA}(f) :=\inf_{w\in X^{div}}(A_X(w)+f(w))$ is the natural non-Archimedean functional from \cite[\S 2.6]{BoJ18}, defined for continuous non-Archimedean functions.
\end{corollary}
\begin{proof} 
Since $\{u_t\}$ is induced by the filtration in \cite[Definition 2.7]{Li20-cscK}, it must be bounded \cite[\S 9]{RWN14}, hence the previous result is applicable. 
Both $v^{an}$ and $u^{an}$ are continuous on $X^{an}$ by the sentence below \cite[formula (28)]{Li20-cscK}, hence \eqref{eq: L_beta_L_NA} holds by the definition of $L^{NA}$. 
\end{proof}

\begin{corollary}
Let $\{u_t\}\in\cR^1_\cI=\mathcal E^{1,NA}$ be a maximal geodesic ray.   
    Then for any $\beta>1$,
\begin{equation}\label{eq: Ent_ineq}
    Ent^\beta\{u_t\}\leq Ent^{NA}(\mathrm{MA}(u^{an})).
\end{equation}
\end{corollary}

Here $\mathrm{MA}(u^{an})$ denotes the non-Archimedean Monge--Amp\`ere measure of $u^{an}$ and $Ent^{NA}$ denotes the non-Archimedean entropy (see \cite[\S 6.7 and Definition 7.17]{BHJ17}).

\begin{proof}
Due to Theorem \ref{thm:Ent-beta=lim-Ent-beta}, \cite[Proposition 6.3]{Li20-cscK} and \cite[Theorem 4.4.1]{Reb23}, we can approximate both sides of \eqref{eq: Ent_ineq} using a sequence of geodesic rays arising from models.
So in what follows we simply assume that $\{u_t\}$ itself a geodesic ray arising from a model . 
Pick any geodesic ray $\{v_t\}$ arising from a test configuration.
Corollary \ref{cor: L_beta_L_NA} yields 
$
L^\beta(\{v_t\},\{u_t\})=L^{NA}(\beta({u}^{an}-v^{an})).
$ 
Applying \cite[Proposition 2.8]{BoJ18} and \cite[(4.5)]{BoJ21a},
\begin{equation*}
    \begin{split}
        L^\beta(\{v_t\},\{u_t\})=L^{NA}(\beta({u}^{an}-v^{an}))&\leq Ent^{NA}(\mathrm{MA}(u^{an}))+\frac{\beta}{V}\int_{X^{an}}({u}^{an}-v^{an})\mathrm{MA}(u^{an}) \\
        &\leq Ent^{NA}(\mathrm{MA}(u^{an}))+\beta(I\{u_t\}-I\{v_t\}).\\
    \end{split}
\end{equation*}
Taking supremum over all $\{v_t\}$, we conclude by Theorem \ref{thm:rad-Ent-beta=sup}.
\end{proof}
Consequently, we obtain an alternative proof of a recent result of Boucksom--Jonsson \cite{BJ25preprint}:
\begin{corollary}\label{cor: BJ_cor}
   For any maximal geodesic ray $\{u_t\}\in\cR^1_\cI =\mathcal E^{1,NA}$ one has
   \begin{equation}
   \label{eq:rad-K=NA-K}
       Ent\{u_t\}=Ent^{NA}(u^{an}).
   \end{equation}
\end{corollary}
\begin{proof} 
Letting $\beta \to \infty$ in \eqref{eq: Ent_ineq}, we obtain from Theorem \ref{thm: K_beta_conv} that
$
    Ent\{u_t\} \leq Ent^{NA}(u^{an}).
$
Combining with the reverse inequality \cite[Theorem 1.2]{Li20-cscK}, we conclude the proof.
\end{proof}

\section{Intersection theoretic interpretation of $K^\beta$-stability}
\label{sec:int-for-of-K-beta}

In this section $\cT=(\cX,\cL)$ will be a smooth test configuration dominating $X\times\PP^1$ and $\beta\in\QQ_{>1}$.
It might be too optimistic to expect that the supremum in \eqref{eq: Ent_bet_form} is realized by some test configuration. In this section, relying on the non-Archimedean expression of $L^\beta(\cdot,\cdot)$ (Proposition \ref{prop:NA-L-beta}), we show that the log discrepancy model $(\cX,\cL_\beta)$ of $(\cX,\cL)$ actually attains the supremum, yielding an intersection formula for $Ent^\beta(\cX,\cL)$ (see Theorem \ref{thm:K_beta_intersection'}):
\begin{equation}
    \label{eq:Ent-T-beta-int-form}
    Ent^\beta(\cX,\cL) = \frac{\beta \langle{\mathcal L}_\beta^{n+1}\rangle-\beta {\mathcal L}^{n+1}}{(n+1)V},
\end{equation}
where ${\mathcal L_\beta} :=\cL+\frac{1}{\beta}K^{\log}_{\cX/X\times\PP^1}$
and
$\langle{\mathcal L}_\beta^{n+1}\rangle$ is the top movable intersection product of the (relatively) big line bundle $\cL_\beta$, understood in the following way:
\begin{equation}
    \label{eq:def-mov-top-prod}
    \langle\cL_\beta^{n+1}\rangle:=\vol(\cL_\beta+c\cX_0)-c(n+1)V,
\end{equation}
where $c\gg 0$ is so that $\cL_\beta+c\cX_0$ is big on $\cX$, hence $\vol(\cL_\beta+c\cX_0)$ makes sense (cf. \cite[(8)]{Li21}). The right hand side of \eqref{eq:def-mov-top-prod} does not depend on the choice of $c$, which can be checked by taking derivative with respect to $c$ and applying \cite[Theorem A]{BFJ09}.

To prove \eqref{eq:Ent-T-beta-int-form} let us fix some notation for the rest of the section. Put $\cT_\beta:=(\cX,\cL_\beta)$.  Let $\{u_t^\cT\}$, $\{u_t^{\cT_\beta}\}$ be the geodesic rays associated to $\cT$ and $\cT_\beta$, respectively.
Let $u^\cT$ and $u^{\cT_\beta}$ be the non-Archimedean potentials associated to $\cT$ and $\cT_\beta$, respectively (recalled  in \S \ref{sec:pre}). Let $\pi: \cX\to X\times\PP^1$ denote the domination map.
Assume that
\begin{equation}
    \label{eq:def-F&G}
    \pi^*p_1^*L-\cL\sim_\QQ F\text{ and }\pi^*p_1^*L-\cL_\beta\sim_\QQ G
\end{equation}
for some $\QQ$-Cartier divisors $F$ and $G$ supported in $\cX_0$. Then we have the relation
\begin{equation}
    \label{eq:F=G+K-log}
    F=G+\frac{1}{\beta}K^{\log}_{\cX/X\times\PP^1}.
\end{equation}
By \eqref{eq:def-NA-u-T} and \eqref{eq:def-NA-of-Model}, for $w\in X^{an}$ the evaluations $u^\cT(w)$  and $u^{\cT_\beta}(w)$ are respectively given by $
u^{\cT}(w)=-\sigma(w)(F)$ and 
\begin{equation}\label{eq: NA_envelope_model}
u^{\cT_\beta}(w)=\sup\{v(w)\in PSH^{NA}: v(y)\leq -\sigma(y)(G),\ \forall\, y\in X^{an}\}.
\end{equation}

Our first observation is that $u^{\cT_\beta}$ does not depend on the choice of smooth models for $\mathcal T$. 

\begin{lemma}
\label{lem:u-T-beta-indpd-of-model}
let $\mu: \cX'\rightarrow \cX$ be another smooth model. Put $\cL'_\beta:=\mu^*\cL+\frac{1}{\beta}K^{\log}_{\cX'/X\times\PP^1}$ and $\cT_\beta':=(\cX',\cL'_\beta)$. Then $u^{\mathcal T'_\beta} = u^{\mathcal T_\beta}$ and $\{u^{\cT'_\beta}_t\}=\{u^{\cT_\beta}_t\}$.
\end{lemma}
\begin{proof}

By \eqref{eq:def-NA-of-Model} and the fact $\mu^*K^{\log}_{\cX/X\times\PP^1}\leq K^{\log}_{\cX'/X\times\PP^1}$ (see \cite[(4.4)]{BHJ17}) we obtain 
$u^{\cT_\beta}\leq u^{\cT'_\beta}$, implying that $u_t^{\cT_\beta}\leq u_t^{\cT'_\beta}$ by \cite[\S 6.4]{BBJ18}.
To show that they are actually equal, it suffices to show that $d_1^c(\{u_t^{\cT'_\beta}\},\{u_t^{\cT_\beta}\})=I\{u_t^{\cT'_\beta}\}-I\{u_t^{\cT_\beta}\}=0.$ For this, we can assume that both $\cL_\beta$ and $\cL_\beta'$ are big $\QQ$-line bundles (since what we aim to show is translation invariant). Then by the slope formula \cite[(27)]{Li20-cscK}, it suffices to show that $\vol(\cL_\beta')=\vol(\cL_\beta)$, which holds due to the relation $\cL_\beta'=\mu^*\cL_\beta+\frac{1}{\beta}(K^{\log}_{\cX'/X\times\PP^1}-\mu^*K^{\log}_{\cX/X\times\PP^1})$ and that the effective divisor $(K^{\log}_{\cX'/X\times\PP^1}-\mu^*K^{\log}_{\cX/X\times\PP^1})$ is $\mu$-exceptional (see \cite[(4.4)]{BHJ17} and \cite[Lemma 3.2]{FKL}).
\end{proof}

\begin{remark}
    For any $w\in X^{div}_\QQ$, one has
$
   \sigma(w)(K^{\log}_{\cX/X\times\PP^1})=A_X(r_\cX(w))
$
 (see \cite[(144)]{Li20-cscK}),
where $r_\cX$ denotes the retraction map to the dual complex of $\cX$ \cite{BoJ18}. By continuity, this identity holds for any $w\in X^{an}$ as well.
Then by \eqref{eq: NA_envelope_model} we obtain that
\begin{equation}
\label{eq:NA-u-beta=P}
       u^{\cT_\beta}=P\left( u^{\cT}+\frac{1}{\beta}A_X\circ r_\cX\right):=\sup\left\{v\in PSH^{NA}: v\leq  u^{\cT}+\frac{1}{\beta}A_X\circ r_\cX\right\}.
   \end{equation}
This explicit expression for $u^{\cT_\beta}$ also implies Lemma \ref{lem:u-T-beta-indpd-of-model}, in light of Xia's identity \cite[(7.2)]{X23}.
\end{remark}

Next we investigate the $L^\beta$ functional (recall \eqref{eq:def-L-beta}) for the pair $(\{u^{\cT_\beta}_t\},\{u^\cT_t\})$:

\begin{lemma}
\label{lem:L-beta=0}
    One has
    $
    L^\beta(\{u^{\cT_\beta}_t\},\{u^\cT_t\}= 0.
    $
\end{lemma}

\begin{proof}

Since $u^\cT$ itself is a candidate for the envelope \eqref{eq: NA_envelope_model}, we obtain that
$
u^{\cT}\leq u^{\cT_\beta}.
$
Then we conclude from the discussions in \cite[\S 6.4]{BBJ18} that
\begin{equation}
    \label{eq:u-t-beta>u-t}
    u_t^{\cT}\leq u_t^{\cT_\beta} \text{ for all }t\geq 0.
\end{equation}
By the definition \eqref{eq:def-L-beta} of $L^\beta(\cdot,\cdot),$ this immediately implies that
$
L^\beta(\{u^{\cT_\beta}_t\},\{u^\cT_t\}\leq0.
$

It remains to argue that
$
L^\beta(\{u^{\cT_\beta}_t\},\{u^\cT_t\}\geq0.
$
This statement is translation invariant in the sense that we can replace $\{u_t^{\cT_\beta}\},\{u_t^\cT\}$ with $\{u_t^{\cT_\beta}-ct\},\{u_t^\cT-ct\}$. As a result, after possibly subtracting certain multiple of $\cX_0$ from $\cL$, we can assume that both $F$ and $G$ are effective. This will imply that $u_t^{\cT_\beta}\leq 0$ and $u_t^{\cT}\leq 0$ for all $t>0$, so that both  $U^{\cT_\beta}_z:=u^{\cT_\beta}_{-\log|z|^2}$ and $U^\cT_z:=u^{\cT}_{-\log|z|^2}$ extend to qpsh functions on $X\times\Delta$. 

Note that the singularity type of the geodesic ray $\{u_t^\cT\}$, when viewed on $\cX$, is $L^\infty$ compatible with the divisor $F$ (see \eqref{eq:sing-type-of-alg-ray}). In other words, we can find some bounded qpsh function $f$ near $\cX_0$ such that
$$
dd^c(\pi^*(U^{\cT}_z)-f)=[F].
$$
This also implies that
\begin{equation}
    \label{eq:nu=ord-E-F}
    \nu(U^{\cT}_z,E)=\ord_E(F)
\end{equation}
for any irreducible component $E$ of $\cX_0$.

The singularity type of $\{u_t^{\cT_\beta}\}$ is more complicated. However, putting together \eqref{eq:nu-T-along-E}, \eqref{eq:def-NA-of-Model}, \eqref{eq:nu-u-M-along-E}, \eqref{eq:F=G+K-log} and \eqref{eq:nu=ord-E-F}, we observe that
$$
\nu(U^{\cT_\beta}_z,E)\geq \nu(U^{\cT}_z,E)-\frac{1}{\beta}\ord_E(K^{\log}_{\cX/{X\times\PP^1}})=\ord_{E}\left(F-\frac{1}{\beta}K^{\log}_{\cX/{X\times\PP^1}}\right)= \ord_{E}(G)
$$
holds for any irreducible component $E$ of $\cX_0$.
So Siu's decomposition (\cite{Siu74}, \cite[(2.18)]{Dem12a} implies the existence of a qpsh function $g$ in a neighborhood of $\cX_0$, where $g$ is bounded from above, such that
$$
\ddc (\pi^*(U^{\cT_\beta}_z)-g)=[G].
$$

Now we plug the above analysis into \eqref{eq:L_beta=lct}. As in the proof of Theorem \ref{thm:L-beta-for-two-TC}, writing $\cX_0=\sum_i a_i E_i$ and using that
$\beta(F-G)=K_{\cX/X\times\PP^1}+\sum_i(1-a_i)E_i$  (recall \eqref{eq:F=G+K-log}),
the integral
$
\int_{X\times\Delta}e^{\beta(U^{\cT_\beta}_z-U^{\cT}_z)-(1+\tau)\log|z|^2}\omega^n\wedge i dz\wedge d\bar z,
$
when pulled back to $\cX$, is comparable to
$
\int \frac{e^{\beta(g-f)}}{\prod_i|h_i|^{2(1+a_i\tau)}},
$
where $h_i$ is some local holomorphic function cutting out $E_i$. Since $g-f$ is bounded from above, we immediately see that this integral is finite whenever $\tau<0$. Thus $L^\beta(\{u_t^{\cT_\beta}\},\{u_t^\cT\})\geq 0$, finishing the proof.
\end{proof}

Now we show that the supremum in \eqref{eq: Ent_bet_form} is achieved by the ray $\{u_t^{\cT_\beta}\}$.

\begin{proposition}
\label{prop:Ent-beta-u-T=beta(I(u-T-beta)-I(u-T))}
    One has 
     $
     Ent^\beta\{u_t^{\cT}\}=\beta(I\{u_t^{\cT_\beta}\}-I\{u_t^\cT\}).
     $
\end{proposition}
   
\begin{proof}
    Let $\mathcal T' = (\cX',\cL')$ be any test configuration with geodesic ray $\{v^{\mathcal T'}_t\}$. We normalize $\{v_t^{\cT'}\}$ such that
    $
    L^\beta(\{v^{\mathcal T'}_t\},\{u^\cT_t\})=L^\beta(\{u^{\cT_\beta}_t\},\{u^\cT_t\})=0
    $
    (recall Lemma \ref{lem:L-beta=0}).
    Therefore, in light of Theorem \ref{thm:rad-Ent-beta=sup}, it suffices to show that $ I\{u^{\cT_\beta}_t\}\geq I\{v^{\mathcal T'}_t\}$, which follows from the following stronger statement that we now turn to show:
\begin{equation}\label{eq: I_ineq}
    u^{\cT_\beta}_t\geq v^{\mathcal T'}_t \text{ for any } t\geq 0.
\end{equation}

In view of Lemma \ref{lem:u-T-beta-indpd-of-model}, after passing to a common model we can assume that $\cX=\cX'$ is smooth. Moreover, after subtracting $ct$ from both $v_t^{\cT'}$ and $u_t^\cT$, we can further assume that we are in the position to apply the extremal characterization \eqref{eq: tau_cond} of $L^\beta(\{v^{\mathcal T'}_t\},\{u^\cT_t\})$. 
Since $L^\beta(\{v^{\mathcal T'}_t\},\{u^\cT_t\})=0$, we obtain that  
$$ A_{X\times\Delta}(E)+\beta\nu(V^{\mathcal T'}_z,E)-\beta\nu(U^\mathcal T_z,E)\geq\nu(\log|z|^2,E)
$$
for any irreducible component $E$ of $\cX_0$. Define $F$ and $G$ as in \eqref{eq:def-F&G}. Also define $F'\subset\cX_0$ such that $\pi^*p_1^*L-\cL'\sim_\QQ F'$. Using \eqref{eq:sing-type-of-alg-ray}, \cite[(4.4)]{BHJ17} and \eqref{eq:F=G+K-log}, we then derive that\begin{equation}
    \begin{split}
        \ord_E(F')&=\nu(V_z^{\cT'},E)\geq \nu(U_z^{\cT},E)-\frac{1}{\beta}\ord_E(K^{\log}_{\cX/X\times\PP^1})\\
        &=\ord_E\left(F-\frac{1}{\beta}K^{\log}_{\cX/X\times\PP^1}\right)=\ord_E(G)
    \end{split}
\end{equation}
for any irreducible component $E$ of $\cX_0$, so that $F'-G$ is effective.

Let $v^{\cT'}$ and $u^{\cT_\beta}$ be the non-Archimedean potentials of $\cT'$ and $\cT_\beta$, respectively.
Recalling \eqref{eq:def-NA-u-T} and the envelope interpretation \eqref{eq:def-NA-of-Model} of $u^{\cT_\beta}$, we immediately obtain that
$v^{\mathcal T'}$ is a candidate for $u^{\cT_\beta}$. This in turn implies $v^{\mathcal T'}_t \leq u^{\mathcal T_\beta}_t$ (see  \cite[\S 6.4]{BBJ18}), concluding \eqref{eq: I_ineq}.
\end{proof}

As a consequence of Proposition \ref{prop:Ent-beta-u-T=beta(I(u-T-beta)-I(u-T))}, we obtain an intersection theoretic formula for $K^\beta_\cT$. 

\begin{theorem}[Theorem \ref{thm: K_beta_intersection}]
\label{thm:K_beta_intersection'}
For any $\beta\in\QQ_{>1}$ we have
\begin{equation}
        \label{eq:Ent-T=beta(vol(L)-vol(L))}
        Ent^\beta_\cT:=Ent^\beta\{u_t^\cT\}=\frac{\beta(\langle\cL_\beta^{n+1}\rangle-\cL^{n+1})}{(n+1)V}.
    \end{equation}
and
\begin{equation}
\label{eq:K-beta-T-intersection-formula}
     K^\beta_\cT=\frac{\beta(\langle\cL_\beta^{n+1}\rangle-\cL^{n+1})}{(n+1)V}+\frac{\bar S\cL^{n+1}}{(n+1)V}+\frac{K_X\cdot\cL^{n}}{V}.
\end{equation}
    Moreover, $\beta\mapsto K^\beta_\cT$ is increasing and
    $$
    \lim_{\beta\to\infty}K^\beta_\cT=K_\cT=V^{-1}\cL^n\cdot K^{\log}_{\cX/X\times\PP^1}+\frac{\bar S\cL^{n+1}}{(n+1)V}+\frac{K_X\cdot\cL^{n}}{V}.
    $$
\end{theorem}

\begin{proof}
    The first identity follows from Proposition \ref{prop:Ent-beta-u-T=beta(I(u-T-beta)-I(u-T))} and the slope formula \cite[(27)]{Li21} (cf. also \eqref{eq:def-mov-top-prod}). The second identity then follows from \eqref{eq: J_Ric_test_conf}. 
    The last identity follows from Theorem \ref{thm: K_beta_conv} and \cite[(8)]{Li20-cscK}. One can also see it directly by sending $\beta\to\infty$ in \eqref{eq:K-beta-T-intersection-formula} and then using the derivative formula \cite[Theorem A]{BFJ09}  (thus giving an alternative proof for \cite[(8)]{Li20-cscK}). The monotonicity of $\beta\mapsto K^\beta_\cT$ is a consequence of Proposition \ref{prop:Ent-beta-increase} (one can also prove it directly by differentiating \eqref{eq:K-beta-T-intersection-formula}). 
\end{proof}

Theorem \ref{thm:K_beta_intersection'} yields the following algebraic analogue of Lemma \ref{lem: Ent_beta_first}(i).

\begin{corollary}
\label{cor:Ent-beta-geq-Ent-P-alg-proof}
    For any $\beta\in\QQ_{>1}$ one has
    $
    Ent(\cX,\cL)\geq Ent^\beta(\cX,\cL)\geq Ent(\cX,\cL_\beta).
    $
\end{corollary}

 Here  $Ent(\cX,\cL_\beta):=\langle\cL^{n}_\beta\rangle\cdot K^{\log}_{\cX/X\times\PP^1}
    $
   denotes the non-Archimedean entropy of the model $(\cX,\cL_\beta)$, and the movable intersection product is taken after adding a large multiple of $\cX_0$ to $\cL_\beta$ so that it becomes big on $\cX$ (see \cite[(25)]{Li21}).

\begin{proof}
We only argue the second inequality.
    The first inequality follows from a similar argument, or one can simply use  Lemma \ref{lem: Ent_beta_first} and \cite[(149)]{Li20-cscK} to conclude.  

The statement we aim to show is translation invariant, in the sense that we can add a large multiple of $\cX_0$ to both $\cL$ and $\cL_\beta$. Thus we may assume that $\cL$ (resp. $\cL_\beta$) is a semi-ample (resp. big) $\QQ$-line bundle on $\cX$. According to \cite[(8)]{Li21}
and that the base locus of $\cL_\beta$ is contained in $\cX_0$, we can find a sequence of smooth models $\cX_m\xrightarrow{\mu_m}\cX$ such that
    $
    \mu_m^*\cL_\beta=\cL_{\beta,m}+E_m,
    $
    where $\cL_{\beta,m}$ is a nef $\QQ$-line bundle on $\cX_m$ and $E_m$ is an effective $\QQ$-divisor supported in the central fiber of $\cX_m$. Moreover, we have that
    $
    \cL^{n+1}_{\beta,m}\to\langle\cL_\beta^{n+1}\rangle\text{ as }m\to\infty.
    $
    By the orthogonality estimate \cite[Theorem 4.1] {BDPP} we also have
    \begin{equation}
        \label{eq:bdpp}
    \lim_{m\to\infty}E_m\cdot\cL^n_{\beta,m}=0.
    \end{equation}

Then from \eqref{eq:Ent-T-beta-int-form} we can write
\begin{equation*}
    \begin{split}
        V\cdot Ent^\beta(\cX,\cL)&=\frac{\beta}{n+1}(\langle\cL_\beta^{n+1}\rangle-\cL^{n+1})=\lim_{m\to\infty}\frac{\beta}{n+1}(\cL_{\beta,m}^{n+1}-\mu^*_m\cL^{n+1})\\
        &\geq \beta \lim_{m\to\infty}(\cL_{\beta,m}-\mu^*_m\cL)\cdot\cL_{\beta,m}^n=\lim_{m\to\infty}(\mu^*_mK^{\log}_{\cX/X\times\PP^1}-E_m)\cdot\cL^n_{\beta,m}\\
        &=\lim_{m\to\infty}\mu^*_mK^{\log}_{\cX/X\times\PP^1}\cdot\cL^n_{\beta,m}=\langle\cL_\beta^n\rangle\cdot K^{\log}_{\cX/X\times\PP^1}.
    \end{split}
\end{equation*}
In the second line we used the Hodge index type estimate $(\mathcal L_{\beta,m} - \mu^*_m \mathcal L) \cdot \mathcal L^{i-1}_{\beta,m} \cdot \mu_m^* \mathcal L^{n-i+1}\geq(\mathcal L_{\beta,m} - \mu^*_m \mathcal L) \cdot \mathcal L^i_{\beta,m} \cdot \mu_m^* \mathcal L^{n-i}$ multiple times (see \cite[Lemma 1.74]{Xubook} and also (2.4) therein), and the last line follows from \eqref{eq:bdpp} and the definition of  $\langle\cL_\beta^n\rangle\cdot K^{\log}_{\cX/X\times\PP^1}$ (see \cite[Theorem 3.5]{BDPP} and \cite[(14)]{Li21}).
\end{proof}

\section{The Boucksom--Jonsson regularization conjecture}

Using our $\beta$-entropy, we prove the following refinement of Chi Li's \cite[Proposition 6.3]{Li20-cscK}, making progress on the Boucksom--Jonsson regularization conjecture \cite{BoJ18,BoJ21a,BJ22}.

\begin{theorem}[Theorem \ref{thm:BJ-conj_main}]\label{thm:BJ-conj_main1}
\label{thm:BJ-conj}
Let $(X,L)$ be a polarized manifold with $\omega\in c_1(L)$.
    Let $\{u_t\}\in\cR^1$ be a geodesic ray with $Ent\{u_t\}<\infty$. Then there exists a sequence of smooth test configurations $(\cX^{(k)},\cL^{(k)})$ dominating $X\times\PP^1$ and rational numbers $\beta_k\to\infty$ such that the geodesic rays $\{v_{t,k}\}$ attached to the log discrepancy models $(\cX^{(k)},\cL^{(k)}_{\beta_k})$ satisfy
    $$
    \lim_{k\to\infty}d_1^c(\{u_t\},\{v_{t,k}\})=0\text{ and }Ent\{u_t\}=\lim_{k\to\infty}Ent\{v_{t,k}\}.
    $$
\end{theorem}

\begin{proof}
    Let $\beta\in\QQ_{>1}$, which is momentarily fixed. Then by Lemma \ref{lem: Ent_beta_first}(i),
    $$
    Ent\{u_t\}\geq Ent^\beta\{u_t\}.
    $$
    By \cite[Theorem 1.2]{Li20-cscK}, $\{u_t\}$ is a maximal geodesic ray, so that we can apply Demailly's regularization to $\{u_t\}$ to find a sequence of test configurations whose associated geodesic rays converge to $\{u_t\}$ in $d_1^c$-topology. Therefore, by the $d_1^c$-continuity of $Ent^\beta\{\cdot\}$ (Theorem \ref{thm:Ent-beta=lim-Ent-beta}) we can pick a test configuration $\cT$ (depending on $\{u_t\}$ and $\beta$) whose associated geodesic ray $\{u_t^\cT\}$ satisfies
    \begin{equation}
      d_1^c(\{u_t\},\{u_t^\cT\})\leq \frac{1}{\beta} \text{ and }  \label{eq:Ent>Ent-beta>Ent-beta-1/beta}
         Ent\{u_t\}\geq Ent^\beta\{u_t\}\geq Ent^\beta\{u_t^\cT\}-\frac{1}{\beta}.
    \end{equation}
    Moreover, up to passing to a resolution (which does not change the attached geodesic ray), we can assume that $\cT$ is smooth, dominating $X\times\PP^1$.
    Consider the log discrepancy model $\cT_\beta$ of $\cT$ at level $\beta$ (recall \eqref{eq: L_beta_def}), whose associated geodesic ray is denoted by $\{u_t^{\cT_\beta}\}$. Then by Corollary \ref{cor:Ent-beta-geq-Ent-P-alg-proof} and the slope formula \eqref{eq:rad-K=NA-K}, 
    \begin{equation}
        \label{eq:ead-Ent>Ent-beta-model-1/beta}
    Ent\{u_t\}\geq Ent^\beta\{u_t\}\geq Ent^\beta\{u_t^{\cT}\}-\frac{1}{\beta}\geq Ent\{u_t^{\cT_\beta}\}-\frac{1}{\beta}.
    \end{equation}
    On the other hand, Proposition \ref{prop:Ent-beta-u-T=beta(I(u-T-beta)-I(u-T))} together with \eqref{eq:u-t-beta>u-t} imply that
    \begin{equation}
    \label{eq:d1c-u-and-u-T-beta}
        \begin{split}
              d_1(\{u_t\},\{u_t^{\cT_\beta}\})&\leq d_1(\{u^\cT_t\},\{u_t^{\cT_\beta}\})+d_1(\{u_t\},\{u_t^{\cT}\})\\
              &\leq \frac{1}{\beta}(Ent^\beta\{u_t^\cT\}+1)
              \leq \frac{1}{\beta}(Ent\{u_t\}+\frac{1}{\beta}+1),\\
        \end{split}
    \end{equation}
where we used \eqref{eq:Ent>Ent-beta>Ent-beta-1/beta} in the second line.

Now choosing a rational sequence $\beta_k\to\infty$ as $k\to\infty$ and applying the above construction to each $\beta_k$, we obtain a sequence of geodesic rays $\{v_{t,k}\}$ arising from log discrepancy models such that
$$
d^c_1(\{u_t\},\{v_{t,k}\})\leq\frac{1}{\beta_k}(Ent\{u_t\}+\frac{1}{\beta_k}+1)
\text{ and }
Ent\{u_t\}\geq Ent\{v_{t,k}\}-\frac{1}{\beta_k}.
$$
Therefore, we obtain that
$$
\lim_{k\to\infty}d_1^c(\{u_t\},\{v_{t,k}\})=0\text{ and }Ent\{u_t\}\geq\limsup_{k\to\infty}Ent\{v_{t,k}\}.
$$
Using that $Ent\{\cdot\}$ is $d_1^c$-lsc \cite[Proposition 5.9]{CC3}, we complete the proof.
\end{proof}

When  $\GG:=\mathrm{Aut}_0(X,L)$ is reductive, let  $\KK$ be a fixed maximal compact subgroup.
Let us recall the terminology from \S 2, introducing the maximal algebraic torus $\TT$ in the center $Z(\GG)$, and $N_\GG(\KK) \leq \GG$, the normalizer of $\KK$.

We recall how $\GG$ acts on the space of K\"ahler metrics. Using averaging one can easily show that $\mathcal H^\KK$, the space of $\KK$ invariant K\"ahler metrics on $(X,L)$, is non-empty. Thus we can assume that $\omega$ is $\KK$-invariant, and  the potentials associated to $\omega$ are

$$\mathcal H^\KK_\omega:= \{ u \in \mathcal H_\omega, \ u \textup{ is } \KK\textup{-invariant}\}.$$
For applications it is convenient to view $\mathcal H^\KK_\omega$ as the following product :
\begin{equation}\label{eq: HK_splitting}
\mathcal H^\KK_\omega= \Bbb R \oplus \mathcal H^{\KK,0}_\omega := \Bbb R \oplus \{ u \in \mathcal H_\omega^\KK, \textup{ and }  I(u)=0\}.
\end{equation}
Unlike in the case of the Mabuchi $d_2$ geometry, we note that this is not a $d_1$-isometric splitting. That being said, as pointed out in many places in the literature (e.g. \cite[Section 5.2]{DR17}),  the action of $\GG$ on K\"ahler metrics translates to both $\mathcal H^{\KK,0}_\omega$ and $\mathcal H^{\KK}_\omega$ using the following formula:
$$g.u = u \circ g + g.0, \ g \in \GG,$$
where $g^* \omega_u = \omega_{g.u}$ and $g.0$ is uniquely determined by $I(g.0)=0$. Thus defined, $\GG$ will act by $d_1$-isometries on both spaces $\mathcal H^{\KK}_\omega$ and $\mathcal H^{\KK,0}_\omega$. Additionally, the action of $\GG$ extends $d_1$-isometrically to $\mathcal E^{1,\KK}_\omega$ and $\mathcal E^{1,\KK,0}_\omega$, the corresponding spaces of $\KK$-invariant finite energy potentials. 

By $\mathcal R^{1,\KK}_\omega$ and $\mathcal R^{1,\KK,0}_\omega$ we denote  the space of $\KK$-invariant finite energy geodesic rays (emanating from $0 \in \mathcal H_\omega^\KK$), respectively the $\KK$-invariant geodesic rays with zero Monge--Amp\`ere slope.

\begin{remark}\label{rem: BJ_conj_equiv}  If $\GG:=\mathrm{Aut}(X,L)_0$ is reductive, then the approximation procedure of Theorem \ref{thm:BJ-conj} can be carried out in a $\GG$-equivariant manner. Namely, if $\{u_t\} \in \mathcal R^{1,\KK}_\omega$ then the same conclusion holds for $\{v_{t,k}\} \in \mathcal R^{1,\KK}_\omega$ induced by $\GG$-equivariant log discrepancy models.

The proof follows the same pattern and we only point out the differences.
We start the same way: $Ent\{u_t\}\geq Ent^\beta\{u_t\}$. Since $\{u_t\}$ is $\KK$-invariant, so are all the multiplier ideal sheaves $\mathcal I(m U_z)$.
Due to coherence, we obtain that for any $g \in \mathcal I(m U_z)$ the derivatives of $g$ in the $\KK$-directions also stay in $\mathcal I(m U_z)$. Since $\GG$ is the complexification of $\KK$ (see \eqref{eq: G_decomp}, \eqref{eq: K_decomp}),  the derivatives of $g$ in the $\GG$-directions also stay in $\mathcal I(k U_z)$, hence $\mathcal I(m U_z)$ is $\GG$-invariant as well. 
Thus, applying Demailly's regularization to $\{u_t\}$ gives a smooth $\GG$-equivariant test configuration $\cT=(\cX,\cL)$ satisfying \eqref{eq:Ent>Ent-beta>Ent-beta-1/beta}. Since $\cT$ is $\GG$-equivariant, its log discrepancy models $\cT_\beta$ are also $\GG$-equivariant. So \eqref{eq:ead-Ent>Ent-beta-model-1/beta} and \eqref{eq:d1c-u-and-u-T-beta} holds in the $\GG$-equivariant setting as well.
\end{remark}

Let $H \leq \GG$ be a subgroup. For $u \in \mathcal H^\KK_\omega$ we introduce
$$J^{H}(u) := \inf_{h \in H} J(h.u).$$
Similarly, for a geodesic ray $\{u_t\} \in \mathcal R^{1,\KK}$, we introduce
$$J^{H}\{u_t\} := \limsup_{t \to \infty }\inf_{h \in H} \frac{J(h.u_t)}{t}.$$
In case $H = \TT$, it can be shown that the above limsup is actually a limit. Moreover, in case of a $\TT$-equivariant test configuration (or model) $\mathcal T$ one can show that the corresponding invariant
$$J^\TT_\mathcal T:= J^\TT\{u^\mathcal T_t\}$$ 
can be computed using only the algebraic data of $\mathcal T$ and $\TT$ \cite{His16a, Li20-cscK, BJ25preprint}.

\medskip 

Since $H$ acts by $d_1$-isometries on $\mathcal H_\omega^\KK$ (and also $\mathcal E_\omega^{1,\KK}$),  we get a natural pseudo-metric on the quotient spaces:
$$d_1^H(Hu,Hv) = \inf_{h \in H} d_1(h.u,v), \ \ u,v \in \mathcal H_\omega^\KK, \textup { or }u,v \in \mathcal E^{1,\KK}_\omega.$$

Let us recall that the growth of the $d_1$ distance and the $J$ functional are the same on $\mathcal H_\omega^0 = \mathcal H_\omega \cap I^{-1}(0)$, the space of K\"ahler potentials with zero Monge--Amp\`ere energy \cite[Proposition 5.5]{DR17}. Namely, there exists $C>0$ such that
$J(u) - C \leq d_1(0,u) \leq J(u) + C, \ \ u \in \mathcal H^{0}_\omega.$
Since $J(u) = J(u-I(u))$, we conclude that
$$J(u) - C \leq d_1(0,u - I(u)) \leq J(u) + C, \ \ u \in \mathcal H_\omega.$$As a result, we get the following interpretation of $J^H(u)$ in terms of the quotient metric $d_1^H$:
\begin{equation}\label{eq: J_d_1_H_equiv}
J^H(u)-C \leq  \inf_{g \in H} d_1(h.(u-I(u)),0) =  d_1^H(H(u-I(u)), H0) \leq J^H(u)+C , \ \ u \in \mathcal H_\omega.
\end{equation}

\medskip Using \cite{CC1,CC2} and the properness/existence principle of \cite{DR17} (sharpened and expanded in \cite{DarvasSurvey}) we can prove the following standard result:

\begin{theorem}\label{thm: DR_properness} Let $(X,L)$ be a polarized manifold. The following are equivalent:\\
\noindent (i) There exists a cscK metric in $(X,L)$;\\
\noindent (ii) $\GG = \mathrm{Aut}_0(X,L)$ is reductive. If $\KK \leq \GG$ is maximally compact, then for some $\gamma,\delta>0$
$$K(u) \geq \gamma J^{\TT}(u) - \delta, \ u \in \mathcal H^\KK_\omega;$$
\noindent (iii) $\GG = \mathrm{Aut}_0(X,L)$ is reductive. If $\KK \leq \GG$ is maximally compact, then for some $\gamma>0,$
$$K\{u_t\} \geq \gamma J^{\TT}\{u_t\} , \ \  \{u_t\} \in \mathcal R^{1,\KK}_\omega,$$
where $\TT = (\Bbb C^*)^m$ is the maximal torus in the center of $\GG$.
\end{theorem}
\begin{proof} 
To start, a classical result of Matsushima and Lichnerowitz \cite[Theorem 3.6.1]{Gaubook} implies that $\GG = \mathrm{Aut}_0(X,L)$ is reductive when (i) holds.

By $N_\GG(\KK) \leq \GG$ we denote the normalizer of $\KK$.  The data $(\mathcal H_\omega^{\KK,0}, d_1, K, N_\GG(\KK))$ satisfies the conditions of the properness/uniqueness principle \cite[Theorem 4.7]{DarvasSurvey} (having its roots in \cite[Theorem 3.4]{DR17}). Since minimizers of $K$ in $\mathcal E^{1,\KK}_\omega$ are smooth csck metrics \cite{CC1,CC2}, most of the conditions $(P1)-(P6)$ and $(P2)^*$ are standard. Regarding $(P5)$, since $\KK$ is maximally compact, we get that $\KK$ is the isometry group of any $\omega_u, \ u\in \mathcal H^{\KK}_\omega$. This implies that $N_\GG(\KK)$ acts transitively on the cscK potentials of $\mathcal H^\KK_\omega$.

As a result, we obtain that (i) is equivalent with $\GG = \mathrm{Aut}_0(X,L)$ being reductive and either one of the following two conditions
\begin{equation}\label{eq: prop_stab}K(u) \geq \gamma d^{N_\GG(\KK)}_1(N_\GG(\KK)u,N_\GG(\KK)0) - \delta, \ u \in \mathcal H^{\KK,0}_\omega,
\end{equation}
\begin{equation}\label{eq: geod_stab}K\{u_t\} \geq \gamma \limsup_{t \to \infty} d^{N_\GG(\KK)}_1(N_\GG(\KK)u_t,N_\GG(\KK)0), \ \  \{u_t\} \in \mathcal R^{1,\KK,0}_\omega.
\end{equation}
However, when $\GG = \mathrm{Aut}_0(X,L)$ is reductive, then due to Lemma \ref{lem: CLi_App}, $N_\GG(\KK) = Z(\GG) \KK = \TT \KK$. As a result, using \eqref{eq: J_d_1_H_equiv} we get that for $ u \in \mathcal H^{\KK}_\omega$ we have
$$ J^\TT(u)-C \leq d^{N_\GG(\KK)}_1(N_\GG(\KK)(u-I(u)),N_\GG(\KK)0) \leq J^\TT(u)+C.$$ 
Using this, and the fact that $K(u) = K(u - I(u))$, we can straightforwardly refine the conditions \eqref{eq: geod_stab} and \eqref{eq: prop_stab} so that they correspond to (ii) and (iii) respectively, thereby concluding the argument. 
\end{proof}

When $\GG:=\mathrm{Aut}(X,L)_0$ is reductive,   we say that and $(X,L)$ is \emph{$\GG$-uniformly K-stable for test configrations/models/log discrepancy models} if there exists $\gamma>0$ such that for all $\GG$-equivariant test configurations/models/log discrepancy models $\mathcal T =(\mathcal X, \mathcal L)$ we have that
\begin{equation}\label{eq: def_K-stab_log_discrep}
K_\mathcal T := K\{u^\mathcal T_t\} \geq \gamma J^\TT_{\mathcal T}.
\end{equation}
Recall that thanks to formula \eqref{eq: BJ_slope_formula}, the quantity $K_{\mathcal T}$ admits a natural intersection theoretic (or non-Archimedean) interpretation.  
Using the above we immediately obtain the following $\GG$-uniform YTD correspondence.

\begin{theorem}[Theorem \ref{thm:YTD'''}] \label{thm:YTD'''1}Let $(X,L)$ be a polarized manifold. Then the following are equivalent. \smallskip \\
\noindent (i) $X$ admits a cscK metric in $c_1(L)$. \smallskip\\
\noindent  (ii) $\GG:=\mathrm{Aut}(X,L)_0$ is reductive and $(X,L)$ is $\GG$-uniformly K-stable for log discrepancy models.
\end{theorem}

\begin{proof} Suppose that (i) holds. Then (ii) follows automatically, from the equivalence between (i) and (iii) in Theorem \ref{thm: DR_properness}. Indeed,  the space of the rays induced by $\GG$-equivariant log discrepancy models  is a subset of $\mathcal R^{1,\KK}_\omega$.

Now suppose that (ii) holds. Let $\{u_t\} \in \mathcal R^{1,\KK}_\omega$. Due to Theorem \ref{thm: DR_properness}, to argue (i), we only need to show that 
$$K\{u_t\} \geq \gamma J^{\TT}\{u_t\} , \ \  \{u_t\} \in \mathcal R^{1,\KK}_\omega.$$

Using the $\GG$-equivariant version of Theorem \ref{thm:BJ-conj} (obtained in Remark \ref{rem: BJ_conj_equiv})  we can find a sequence of geodesic rays $\{v_{t,k}\}$ arising from $\GG$-equivariant log discrepancy models such that
    $$
    \lim_{k\to\infty}d_1^c(\{u_t\},\{v_{t,k}\})=0\text{ and }K\{u_t\}=\lim_{k\to\infty}K\{v_{t,k}\}.
    $$
Let $\{u'_t\} := \{u_t - I(u_t)\}, \ \{v'_{t,k}\} := \{v_{t,k} - I(v_{t,k})\} \in \mathcal R^{1,\KK}_\omega$. Since $I\{\cdot\}$ is $d_1^c$-continuous we obtain that 

    $$
    \lim_{k\to\infty}d_1^c(\{u'_t\},\{v'_{t,k}\})=0\text{ and }K\{u_t\}=K\{u'_t\}=\lim_{k\to\infty}K\{v'_{t,k}\} = \lim_{k\to\infty} K\{v_{t,k}\}.
    $$

Using \eqref{eq: J_d_1_H_equiv}, we will be done if we can show that
\begin{equation}\label{eq: J_limit}
\lim_{k \to \infty} \limsup_{t \to \infty} \frac{d_1^{\TT} (\TT v'_{k,t},\TT 0)}{t} = \lim_{k \to \infty} J^\TT\{v_{t,k}\} = J^\TT\{u_t\} = \limsup_{t \to \infty } \frac{d_1^{\TT}(\TT u'_t,\TT 0)}{t}.  
\end{equation}
But this simply follows from the triangle inequality of the pseudo-metric $d_1^\TT$:
$$\frac{1}{t}|d_1^{\TT} (\TT v'_{k,t},\TT0) - d_1^{\TT}(\TT u'_t,\TT 0)| \leq \frac{1}{t}d_1^{\TT}(\TT u'_t,\TT v'_{k,t}) \leq \frac{1}{t}d_1(u'_t, v'_{k,t}) \leq d_1^c(\{u'_t\},\{v'_{k,t}\}) \to 0,$$
as $k \to \infty$.
\end{proof}

\let\oldthebibliography\thebibliography
\let\endoldthebibliography\endthebibliography
\renewenvironment{thebibliography}[1]{
  \begin{oldthebibliography}{#1}
    \setlength{\itemsep}{0pt}
    \setlength{\parskip}{0pt}
    \setlength{\parsep}{0pt}
}{\end{oldthebibliography}}

\small{
\bibliography{ref.bib}
\bibliographystyle{abbrv}

\noindent {\sc University of Maryland, College Park, USA}\\
{\tt tdarvas@umd.edu}\vspace{0.1in}\\
\noindent {\sc Beijing Normal University, Beijing, China}\\
{\tt kwzhang@bnu.edu.cn}\vspace{0.1in}
}

\end{document}